\documentclass[12pt,reqno]{amsart}
\usepackage{color}
\usepackage{amsmath}

\usepackage{amsthm}
\usepackage{amssymb}
\usepackage{graphics}
\usepackage{latexsym}
\usepackage{comment}
\usepackage{ulem}

\numberwithin{equation}{section}
\newtheorem{thm}{Theorem}[section]
\newtheorem{prop}[thm]{Proposition}
\newtheorem{lem}[thm]{Lemma}
\newtheorem{cor}[thm]{Corollary}
\newtheorem{hyp}{Hypothesis}

\theoremstyle{remark}
\newtheorem{rem}{Remark}[section]
\newtheorem{defn}{Definition}

\newcommand{\BBB}{\mathbb}
\newcommand{\R}{{\BBB R}}
\newcommand{\Z}{{\BBB Z}}
\newcommand{\T}{{\BBB T}}

\newcommand{\N}{{\BBB N}}
\newcommand{\C}{{\BBB C}}
\newcommand{\LR}[1]{{\langle {#1} \rangle }}

\newcommand{\al}{\alpha}
\newcommand{\be}{\beta}

\newcommand{\e}{\varepsilon}
\newcommand{\ta}{\tau}

\newcommand{\p}{\partial}

\newcommand{\de}{\delta}
\newcommand{\om}{\omega}
\newcommand{\Om}{\Omega}

\newcommand{\zz}{{\bf z}}

\newcommand{\supp}{\operatorname{supp}}

\newcommand{\I}{\infty}

\newcommand{\sgn}{\operatorname{sgn}}

\newcommand{\EQS}[1]{\begin{align} #1 \end{align}}
\newcommand{\EQQS}[1]{\begin{align*} #1 \end{align*}}

\newcommand{\F}{\mathcal{F}}
\newcommand{\1}{{\mathbf 1}}

\newcommand{\ti}{\widetilde}
\newcommand{\ha}{\widehat}

\title[One-dimensional dispersive equations with general nonlinearity]
{Unconditional well-posedness  for some nonlinear periodic one-dimensional dispersive equations}

\author[L. Molinet and T. Tanaka]{Luc Molinet and Tomoyuki Tanaka}
\address[L. Molinet]{Institut Denis Poisson, Universit\'e de Tours, Universit\'e d'Orl\'eans, CNRS, Parc Grandmont, 37200 Tours, France}
\email[L. Molinet]{Luc.Molinet@lmpt.univ-tours.fr}
\address[T. Tanaka]{Graduate School of Mathematics, Nagoya University,
Chikusa-ku, Nagoya, 464-8602, Japan}
\email[T. Tanaka]{d18003s@math.nagoya-u.ac.jp}

\begin{document}
\maketitle
\setcounter{page}{001}

\begin{abstract}
  We consider the Cauchy problem for one-dimensional dispersive equations with a general nonlinearity in the periodic setting. Our main hypotheses are both that the dispersive operator behaves for high frequencies as a Fourier multiplier by $ i  |\xi|^\alpha \xi $, with $ 1\le \alpha \le 2 $, and that the nonlinear term is of the form $ \partial_x f(u) $ where $ f $ is the sum of an entire series with infinite radius of convergence. Under these conditions, we prove the unconditional local well-posedness of the Cauchy problem in $ H^{s}(\mathbb{T}) $ for $ s\ge 1-\frac{\alpha}{2(\alpha+1)}$. This leads to some global existence results in the energy space
   $ H^{\alpha/2}(\mathbb{T})  $,  for $ \alpha \in [\sqrt{2},2] $.
\end{abstract}

%%%%%%%%%%%%%%%%%%%%%%%%%%%%%%%%%%%%%%%%%%%%%%%%%%
%%%%%%%%%%%%%%%%%%%%%%%%%%%%%%%%%%%%%%%%%%%%%%%%%%
%%%%%%%%%%%%%%%%%%%%%%%%%%%%%%%%%%%%%%%%%%%%%%%%%%
%%%%%%%%%%%%%%%%%%%%%%%%%%%%%%%%%%%%%%%%%%%%%%%%%%
\section{Introduction}
Our goal is  to establish low regularity well-posedness results for quite
general one-dimensional nonlinear dispersive equations of KdV type in the periodic setting.
Our class contains  in particular the generalized KdV equations with fractional dispersion
\EQS{
\label{eq11}
&\p_t u -\partial_x D_x^\alpha  u +\p_x (f(u))=0,\quad (t,x)\in \R\times\T,\\
\label{initial}
& u(0,x)=u_0(x),\quad x\in\T,
}
where   $\T=\R/2\pi\Z$,  $ f:\R\to \R $  is the sum of an entire series, $\al\in[1,2]$ and $ D_x^\alpha $ is the  Fourier multiplier by $|\xi|^\alpha $.

Famous equations in this class are the generalized Korteweg-de Vries equation (gKdV)
\EQQS{
  \p_t u + \p_x^3 u  + \p_x (f(u))=0
}
 that corresponds to $ \alpha=2 $ as well as  the generalized Benjamin-Ono equation (gBO)
\EQQS{
  \p_t u -  {\mathcal H} \p_x^2 u  + \p_x (f(u))=0,
}
where $ {\mathcal H} $ is the Hilbert transform (Fourier multiplier by $ -i \sgn(\xi) $), that corresponds to $ \alpha =1$.

The Cauchy problem for this type of dispersive equations has been extensively studied since more than thirty years starting with the work of Kato
\cite{Kato83} (see also \cite{ABFS89}).
It is worth noticing that the first methods employed did not make use of the dispersive effects and the local well-posedness was restricted to Sobolev spaces with index $s> 3/2 $.
At the end of the eighties Kenig, Ponce and Vega succeeded in exploiting dispersive effects of the linear part of these equations in order to lower the required Sobolev regularity on the initial data (see for instance \cite{KPV1}, \cite{Po1}). They obtained in particular the well-posedness of the KdV equation in the energy space and were able in \cite{KPV2}   to apply  a contraction mapping argument on the integral formulation for some $k$-generalized KdV equations ($f(x)=x^k$). Then in the early nineties,
Bourgain introduced the now so-called Bourgain's spaces that take into account the localization of the space-time Fourier transform of the function and where one can solve by a fixed point argument  a wide class of dispersive equations on $ \R$ or $ \T $ with very rough initial data (\cite{B93}).  At this stage,  let us underline that since the nonlinearity of these equations is in general algebraic,  the fixed point argument ensures the real analyticity of the solution-map.

  In the early 2000's,
 Saut, Tzvetkov and the first author \cite{MST} noticed that a large class of ``weakly'' dispersive equations posed on the real line, including in particular
\eqref{eq11} with $ \alpha<2 $ and $ f(x)=x^2 $, cannot be solved by a fixed point argument for initial data in any Sobolev spaces $ H^s (\R) $. Since then, except in the case of integrability of the equation where methods
linked to this remarkable property have been recently developed (see for instance  \cite{KT06}, \cite{GKT} in the periodic setting), mainly two types of approaches have been developed to solve this class of equation in low regularity Sobolev spaces. The first one consists in applying a gauge
transform (see \cite{Tao04} for the Benjamin-Ono equation and \cite{HIKK} for the dispersion  generalized  KdV equation) to eliminate the worst nonlinear interactions that are some high-low frequency interactions and then to solve the equation after the gauge transform by using dispersive tools  as Strichartz's and Bourgain's type estimates.
The second one consists in enhancing the energy method (or modified energy method) with a priori estimates that takes the dispersion
into account. For this one possibility is to   improve the dispersive estimates by localizing a solution in space frequency depending time intervals and then mixing them with classical energy estimates. This type of methods was first introduced by Koch and Tzvetkov  \cite{KT1} (see also \cite{KK} for
some
 improvements) in the framework of Strichartz's spaces and then by Koch and Tataru \cite{KTa} (see also \cite{IKT}) in the framework of Bourgain's
spaces.
 Another method that does not need to localize in space frequency depending time intervals has been recently introduced in \cite{MV15}. Note that  in \cite{MPV18}, \cite{MPV19} it is shown  that in some cases  this last method can be improved by adding some a priori estimates obtained by localizing in space frequency depending time intervals.
We implement this last approach enhanced with some other tools in this paper.
%It is this last  approach enhanced with some other tools  that we implement in this paper.

\subsection{Presentation of the main results}
In this work we consider roughly the same type of dispersion term as in \cite{Guo12} and  \cite{MV15}  but with a much more general nonlinear term. More precisely
we consider dispersive equation of the form
\begin{equation}\label{eq1}
\p_t u + L_{\alpha+1} u + \partial_x(f(u)) =0
\end{equation}
under the two following hypotheses on $ L_{\al+1} $ and $ f$.
\begin{hyp}\label{hyp1}
$ L_{\alpha+1} $ is the Fourier multiplier operator by  $  - i p_{\al+1} $ where
$p_{\al+1}\in C^{1}(\R)\cap C^2(\R\backslash\{0\})$ is a real-valued odd function satisfying, for some $\xi_0>0$, $p_{\al+1}'(\xi)\sim \xi^\al$ and $ p_{\al+1}''(\xi)\sim \xi^{\al-1}$ for all $\xi\ge \xi_0$.
\end{hyp}
\begin{hyp}\label{hyp2}
$ f:\R\to\R$ is the sum of an entire series with infinite radius of convergence.
\end{hyp}
\begin{rem}
We choose a sign in Hypothesis \ref{hyp1} only to fix in an easy way the defocussing case  that leads to a global existence result  ((2) of Theorem
\ref{theo3} below). Of course for all other results the sign will play no
role here.
\end{rem}
 \begin{rem}
 Hypothesis \ref{hyp1} is fulfilled by the purely dispersive operators $ L:=\partial_x D_x^\alpha $ with $ \alpha>0 $, the linear Intermediate Long Wave operator
 $ L:=\partial_x  D_x \coth(D_x) $ and  some perturbations of the Benjamin-Ono equation as the Smith operator $ L:=\partial_x (D_x^2+1)^{1/2}$
(see \cite{Smith}) for which $ \alpha=1$.
 \end{rem}

\begin{rem} The hypothesis on $ f$ ensures that $ f$ is of class $ C^\infty $ and is at each point of $ \R $ the sum of its Taylor expansion at  the origin, i.e.,
\begin{equation}\label{hypf}
f(x)=\sum_{n=0}^{+\infty} \frac{f^{(n)}(0)}{n!} x^n , \quad \forall x\in \R \; .
\end{equation}
Of course any polynomial function enters in this class but also the exponential functions as $ e^x$, $ \sin(x)$,  $\cos(x) $ and their products and compositions.
\end{rem}
Here we concentrate ourself on the periodic setting. Actually we notice that in the real line case, following \cite{KT1}, one can exploit the classical Strichartz estimates in $ L^4_t L^\infty_x $  to get the following
improved Strichartz estimates on smooth solutions to \eqref{eq11} in the purely dispersive case $L_{\al+1}=\p_x D_x^\al$:
\EQQS{
  \|D_x^{1/2} u\|_{L^2_T L^\infty_x} \lesssim C(\|u\|_{L^\infty_T H^{(1-\alpha/4)+}_x})\; .
}
With this estimate in hand, the same method as the one developed in this work but with a clearly more simple decomposition, lead to the unconditional local well-posedness of \eqref{eq11} in $ H^{(1-\alpha/4)+}(\R)$ (see
\cite{P21} where the  KdV case $ \alpha= 2$ is treated with initial data that may be non vanishing at infinity).

Our approach is based on the method introduced in \cite{MV15} enhanced with some arguments that can be found in \cite{MPV19}, \cite{KS20p} and  \cite{KT1} . As noticed in \cite{MV15}, this method is particulary well-adapted to solve the Cauchy problem  associated with  one-dimensional dispersive wave equations in the Sobolev spaces $ H^s $ for $ s>1/2$. It combines classical energy estimates with Bourgain's type estimates that measure
the localization of the space-time Fourier transform of the solution around the curve given by the Fourier symbol of the associated linear dispersive equation.
In our approach we combine this with improved Strichartz  estimates and symmetrization arguments.
The main strategy is to use symmetry arguments to distribute the lost derivative to several functions and then to recover it by either improved Strichartz estimates or Bourgain's type estimates
 depending on whether the  nonlinear interactions are resonant. Since the equation for the difference of two solutions enjoys less symmetries,  this difference will be estimated in a lower regularity space than the solution itself and we will use the frequency envelope approach to recover the continuity result with respect to initial data.

We emphasize that our work was motivated by \cite{KS20p} where the local well-posedness of the $k$-gBO equation ($L_2= \partial_x D_x  $ and $f(x)
=x^k$, $k\ge 2$) is proven in $ H^{3/4}(\T) $ by using another method that is short time Fourier transform restriction method combined with modified energy estimates. It is interesting to underline that our local well-posedness result in this case is exactly the same (except that we also get the unconditional uniqueness)\footnote{To be precise, they also deduced a priori estimate at the regularity $s>1/2$.}. Both methods have the advantage to not use a gauge transformation and thus to be more flexible with respect to perturbations of the equation.
However we believe that our approach  is easier to implement to solve \eqref{eq1}.

We are not aware of any low regularity results in the literature on the Cauchy problem associated with \eqref{eq11} with a general nonlinearity in
the periodic setting.
However we notice that, in \cite{CKSTT04}, the local well-posedness of the  $k$-gKdV equation  was proven to be well-posed  in $ H^{1/2}(\T) $ by a contraction mapping argument in Bourgain's spaces. This suggests that the
LWP of the gKdV equation (at least with $ f$ satisfying hypothesis \ref{hyp2}) should be locally well-posed in $ H^{1/2+}(\T) $ whereas we only get the unconditional LWP of the gKdV equation in $ H^{2/3}(\T) $ in this work. On the other hand, in the case  $ \alpha<2 $, we  do not know any  local well-posedness results for \eqref{eq11}  with $ f(x) =x^k $, $k\ge
2$, in Sobolev spaces  with lower regularity than the ones obtained in Theorem \ref{theo1} below. Actually the only results we know are \cite{KS20p}, mentioned above, and  \cite{MR09} where  Ribaud and the first author proved the local well-posedness of the $k$-gBO equation in $ H^1(\T) $ by using gauge transform and Strichartz estimates. We believe that this approach could be adapted to quite general nonlinearity $ \partial_x (f(u))$ to get the LWP of the gBO equation in $ H^1(\T) $. However it is not clear how to adapt such approach for more general dispersive term as it is done here. %\vspace*{2mm} \\

   Before stating our main result, let us recall  our notion of solutions:
    \begin{defn}\label{def} Let $s>1/2$. We will say that $u\in L^\infty(]0,T[;H^s(\T)) $ is a solution to \eqref{eq1}  associated with the initial datum $ u_0 \in H^s(\T)$  if
  $ u $ satisfies \eqref{eq1}-\eqref{initial} in the distributional sense, i.e. for any test function $ \phi\in C_c^\infty(]-T,T[\times \T) $,  it
holds
  \begin{equation}\label{weakeq}
  \int_0^\infty \int_{\T} \Bigl[(\phi_t +L_{\alpha+1}\phi )u +  \phi_x f(u) \Bigr] \, dx \, dt +\int_{\T} \phi(0,\cdot) u_0 \, dx =0
  \end{equation}
 \end{defn}
 \begin{rem} \label{rem2} Note that for $u\in L^\infty(]0,T[;H^s(\T)) $,  with $ s>1/2 $, $ f(u) $ is well-defined and
  belongs to $ L^\infty(]0,T[;H^s(\T))$. Moreover,
  Hypothesis \ref{hyp1}  forces
 \EQQS{
   L_{\alpha+1} u \in L^\infty(]0,T[;H^{s-\alpha-1}(\T)) \, .
 }
  Therefore $ u_t \in L^\infty(]0,T[; H^{s-\alpha-1}(\T)) $  and \eqref{weakeq} ensures that  \eqref{eq1} is satisfied in $ L^\infty(]0,T[; H^{s-\alpha-1}(\T)) $.
  In particular, $ u\in C([0,T]; H^{s-\alpha-1}(\T))$ and \eqref{weakeq} forces  the initial condition $ u(0)=u_0 $. Note that this ensures that
$u\in C_w([0,T];H^s(\T)) $ and thus
   $ \|u_0\|_{H^{s}}\le  \|u\|_{L^\infty_T H^s} $. Finally, we notice that this also ensures that $ u $ satisfies the Duhamel formula associated with \eqref{eq1}.
 \end{rem}
 Finally, let us recall  the notion of  unconditional well-posedness that
was introduced by Kato \cite{Kato95}, which is, roughly speaking, the local well-posedness with uniqueness of solutions in $L^\I (]0,T[;H^s(\T))$.

 \begin{defn}  We will say that the Cauchy problem associated with \eqref{eq1} is unconditionally locally well-posed in $ H^s(\T )$ if for any initial data $ u_0\in H^s( {\T}) $ there exists $ T=T(\|u_0\|_{H^s})>0 $ and a solution
 $ u \in C([0,T]; H^s(\T)) $ to \eqref{eq1} emanating from $ u_0 $. Moreover, $ u $ is the unique solution to  \eqref{eq1} associated with $ u_0 $
that belongs to  $ L^\infty(]0,T[; H^s(\T) )$. Finally, for any $ R>0$, the solution-map $ u_0 \mapsto u $ is continuous from the ball of $ H^s(\T) $  with radius $ R $ centered at the origin  into $C([0,T(R)]; H^s( \T)) $.
 \end{defn}
We note that Babin, Ilyin and Titi \cite{BIT11}  proved the unconditional uniqueness of the KdV equation in $L^2(\T)$ by  integration by parts in time. This method, which is actually a normal form reduction, has been now succesfully applied to  a variety of dispersive equations (see for instance  \cite{GKO13,K19p1,K19p,KO12,KOY20,MPp} and references therein).

Our main result is the following one:

\begin{thm}[Unconditional well-posedness]\label{theo1}
Assume that Hypothesies \ref{hyp1}-\ref{hyp2} are satisfied
 with $ 1\le \alpha\le 2$. Then for any $ s\ge s(\al)=1-\frac{\alpha}{2(\alpha+1)}$ the Cauchy problem associated with \eqref{eq1} is unconditionally locally  well-posed in $ H^s(\T ) $ with a maximal time of existence $T\ge g(\|u_0\|_{H^{s(\al)}})>0 $ where $ g$ is a smooth decreasing
function depending only on  $ L_{\alpha+1}$ and $f$.
\end{thm}
\begin{rem}
The limitation $ s\ge s(\al)=1-\frac{\alpha}{2(\alpha+1)}$ in the above theorem is due to the following nonlinear interactions:  there are 3 input high frequencies of the same  order than the output frequency whereas all other input frequencies are much lower. In this configuration the nonlinear term can be resonant and to recover the lost derivative we  put
 the  4 terms with high frequencies in $ L^4_{t,x} $ (whereas the other ones are put in $ L^\infty_{t,x} $)  and we use the improved Strichartz
estimates in Proposition \ref{prop_stri}  to control their norms in this space. This is the lost of derivatives in these Strichartz estimates that leads to the value of $ s(\al) $. More precisely, in  Proposition \ref{prop_stri}, it is proven that  any solution $ u\in L^\infty(]0,T[;H^s(\T)) $ to \eqref{eq1} satisfies $ D_x^{s-\beta(\alpha)} u \in L^4_T L^4_x $ with $ \beta(\al)
=\frac{1}{4(\alpha+1)} $. Then, in order to recover the lost derivative
in this configuration,  direct calculations lead to $ 4(s-\beta(\alpha)) \ge 2s+1 \Leftrightarrow s\ge 1/2+2 \beta(\alpha)  \Leftrightarrow s\ge s(\al) $.
\end{rem}
Equation \eqref{eq1} enjoys the following conservation laws at the $ L^2$
and at the $ H^{\al/2} $-level:
\EQQS{
M(u)=\int_{\T} u^2 \quad \text{and} \quad
E(u)=\frac{1}{2} \int_{\T} u  \partial_x^{-1} L_{\alpha+1} u +\int_{\T}
F(u)
}
where $ \partial_x^{-1} L_{\alpha+1} $ is the Fourier multiplier by  $ \frac{p_{\alpha+1}(k)}{k} \1_{k\neq 0} $
and
\begin{equation}\label{defF} F(x) :=\int_0^x f(y) \, dy \; .
\end{equation}
At this stage it is worth noticing that Hypothesis \ref{hyp1} ensures that the restriction of the quadratic part of the energy $ E $ to high frequencies  behaves as the $ H^{\alpha/2}(\T)$-norm whereas its restriction to the low frequencies can be  controlled by the $ L^2$-norm.
Therefore, gathering these conservation laws with the above local well-posedness result we are able to obtain the two following global existence results:
\begin{thm}[Global existence for small initial data] \label{theo2}
Assume that Hypotheses \ref{hyp1}-\ref{hyp2} are satisfied with
 $ \alpha \in [\sqrt{2},2] $. Then there exists $ A=A(L_{\alpha+1},f) >0 $ such that for any initial data $ u_0\in H^s(\T) $ with $ s\ge \alpha/2 $ such that $ \|u_0\|_{H^{\alpha/2}} \le A $, the solution constructed in Theorem \ref{theo1} can be extended for all times. Moreover its trajectory is bounded in $ H^{\alpha/2}(\T) $.
\end{thm}
\begin{thm}[Global existence for arbitrary large initial data] \label{theo3}\text{ }\\
Assume that Hypotheses \ref{hyp1}-\ref{hyp2} are satisfied with
 $ \alpha \in [\sqrt{2},2] $. Then the solution constructed in  Theorem \ref{theo1} can be extended for all times if the function $F$ defined
in \eqref{defF}  satisfies one of the following conditions:
 \begin{enumerate}
 \item There exists $C>0  $ such that $|F(x)|\le C (1+|x|^{p+1})$ for some $ 0<p<2\alpha+1  $.
\item   There exists $ B>0  $ such that $ F(x) \le B , \; \forall x\in \R
$.
 \end{enumerate}
 Moreover its trajectory is bounded in $ H^{\alpha/2}(\T) $.
 \end{thm}
 \begin{rem}
 Typical examples for the case (1) are \\
 $ \bullet $  $f(x)$ is a  polynomial function of degree strictly less than $ 2\alpha +1$. \\
  $\bullet $ $f(x)$ is a polynomial function of $ \sin(x) $ and $\cos(x) $.   \\
  Whereas typical example for the case (2) are \\
  $ \bullet $ $ f(x) $ is a  polynomial function of odd degree with $ \displaystyle\lim_{x\to +\infty} f(x)=-\infty $. \\
   % $ \bullet $ $ f(x)=p(x) $ where $ p $ is a  polynomial function of odd degree with $ \displaystyle\lim_{x\to +\infty} p(x)=-\infty $. \\
  $\bullet $ $f(x) =-\exp(x) $ or $ f(x)=-\sinh(x) $.
  \end{rem}
\section{Notation, function spaces and basic estimates}\label{notation}

  % {\bf For this moment, I mainly follow notation introduced by \cite{MPV19}. I specify some modifications to our problem in this section. In all the sequel $ G=G[f] $ denotes the sum of an entire series with non negative coefficients and infinite radius of convergence.}

\subsection{Notation}

Throughout this paper, $\N$ denotes the set of non negative integers.
For any positive numbers $a$ and $b$, we write $a\lesssim b$ when there exists a positive constant $C$ such that $a\le Cb$.
We also write $a\sim b$ when $a\lesssim b$ and $b\lesssim a$ hold.
Moreover, we denote $a\ll b$ if the estimate $b\lesssim a$ does not hold.
For two non negative numbers $a,b$, we denote $a\vee b:=\max\{a,b\}$ and $a\wedge b:=\min\{a,b\}$.
We also write $\LR{\cdot}=(1+|\cdot|^2)^{1/2}$.

For $u=u(t,x)$, $\F u=\tilde{u}$ denotes its space-time Fourier transform, whereas $\F_x u=\hat{u}$ (resp. $\F_t u$) denotes its Fourier transform in space (resp. time). We define the Riesz potentials by $D_x^s g:=\F_x^{-1}(|\xi|^s \F_x g)$.
We also denote the unitary group associated to the linear part of \eqref{eq1} by $U_\al(t)=e^{-tL_{\al+1}}$, i.e.,
\EQQS{
  U_\al(t)u=\F_x^{-1}(e^{itp_{\al+1}(\xi)}\F_x u).
}

In the present paper, we fix a smooth cutoff function $\chi$:
let $\chi\in C_0^\I(\R)$ satisfy
\EQQS{
  0\le \chi\le 1, \quad \chi|_{-1,1}=1\quad
  \textrm{and}\quad \supp\chi\subset[-2,2].
}
We set $\phi(\xi):=\chi(\xi)-\chi(2\xi)$.
For any $l\in\N$, we define
\EQQS{
  \phi_{2^l}(\xi):=\phi(2^{-l}\xi),\quad
  \psi_{2^l}(\ta,\xi):=\phi_{2^l}(\ta-p_{\al+1}(\xi)),
}
where $ip_{\al+1}(\xi)$ is the Fourier symbol of $L_{\al+1}$.
By convention, we also denote
\EQQS{
  \phi_0(\xi)=\chi(2\xi)\quad
  \textrm{and}\quad
  \psi_0(\ta,\xi)=\chi(2(\ta-p_{\al+1}(\xi))).
}
Any summations over capitalized variables such as $K, L, M$ or $N$ are presumed to be dyadic.
We work with non-homogeneous dyadic decompositions, i.e., these variables
ranges over numbers of the form $\{2^k; k\in\N\}\cup \{0\}$.
We call those numbers \textit{nonhomogeneous dyadic numbers}.
It is worth pointing out that $\sum_N\phi_N(\xi)=1$ for any $\xi\in\Z$,
\EQQS{
  \supp(\phi_N)\subset\{N/2\le |\xi|\le 2N\},\ N\ge 1,\quad
  \textrm{and}\quad
   \supp(\phi_0)\subset\{|\xi|\le 1\}.
}

Finally, we define the Littlewood--Paley multipliers $P_N$ and $Q_L$ by
\EQQS{
  P_N u=\F_x^{-1}(\phi_N \F_x u) \quad\textrm{and}
  \quad Q_Lu=\F^{-1}(\psi_L \F u).
}
We also set
$P_{\ge N}:=\sum_{K\ge N}P_K, P_{\le N}:=\sum_{K\le N}P_K, Q_{\ge N}:=\sum_{K\ge N}Q_K$ and $Q_{\le N}:=\sum_{K\le N}Q_K$.

\subsection{Function spaces}
For $1\le p\le \I$, $L^p(\T)$ is the standard Lebesgue space with the norm $\|\cdot\|_{L^p}$.

In this paper we will use the frequency envelope method (see for instance \cite{Tao04}  and \cite{KT1}) in order to show the continuity result with respect to initial data. To this aim, we have to slightly modulate the classical Sobolev spaces in the following way:
for $s\ge0$ and a dyadic sequence $\{\om_N\}$, we define $H_\om^s(\T)$ with the norm
\EQQS{
  \|u\|_{H_\om^s}
  :=\bigg(\sum_{N}\om_N^2 (1\vee N)^{2s}\|P_N u\|_{L^2}^2\bigg)^{1/2}.
}
Note that $H_\om^s(\T)=H^s(\T)$ when we choose $\om_N\equiv 1$.
Here, $H^s(\T)$ is the usual $L^2$--based Sobolev space.
If $B_x$ is one of spaces defined above, for $1\le p\le \I$ and $T>0$, we define the space--time spaces $L_t^p B_x :=L^p(\R;B_x)$ and $L_T^p B_x :=L^p([0,T];B_x)$ equipped with the norms (with obvious modifications for $p=\I$)
\EQQS{
  \|u\|_{L_t^p B_x}=\bigg(\int_\R\|u(t,\cdot)\|_{B_x}^p dt\bigg)^{1/p}\quad
  \textrm{and}\quad
  \|u\|_{L_T^p B_x}=\bigg(\int_0^T\|u(t,\cdot)\|_{B_x}^p dt\bigg)^{1/p},
}
respectively.
For $s,b\in\R$, we introduce the Bourgain spaces $X^{s,b}$ associated to the operator $L_{\al+1}$ endowed with the norm
\EQQS{
  \|u\|_{X^{s,b}}
  =\Bigg(\sum_{\xi=-\I}^\I \int_{-\I}^\I \LR{\xi}^{2s}\LR{\ta-p_{\al+1}(\xi)}^{2b}|\tilde{u}(\ta,\xi)|^2d\ta\Bigg)^{1/2}.
}
%Here, $\tilde{u}$ is the space-time Fourier transform of $u$.
We also use a slightly stronger space $X_\om^{s,b}$ with the norm
\EQQS{
  \|u\|_{X_\om^{s,b}}
  :=\bigg(\sum_{N}\om_N^2 (1\vee N)^{2s}\|P_N u\|_{X^{0,b}}^2\bigg)^{1/2}.
}
We define the fuction spaces $Z^s $ (resp. $Z^s_\om $), with $s\in \R$, as $Z^s:= L_t^\I H^s\cap X^{s-1,1}$ (resp. $Z^s_\om:= L_t^\I H_\om^s\cap X_\om^{s-1,1}$), endowed with the natural norm
\EQQS{
  \|u\|_{Z^s}=\|u\|_{L_t^\I H^s}+\|u\|_{X^{s-1,1}} \quad
  (\text{resp}.\  \|u\|_{Z^s_\om}=\|u\|_{L_t^\I H^s_\om}+\|u\|_{X^{s-1,1}_\om}) .
}
We also use the restriction in time versions of these spaces.
Let $T>0$ be a positive time and $B$ be a normed space of space-time functions.
The restriction space $B_T$ will be the space of functions $u:]0,T[\times\T\to\R$ or $\C$ satisfying
\EQQS{
  \|u\|_{B_T}
  :=\inf\{\|\tilde{u}\|_B \ |\ \tilde{u}:\R\times \T\to\R\ \textrm{or}\ \C,\ \tilde{u}=u\ \textrm{on}\ ]0,T[\times\T\}<\I.
}

Finally, we introduce a bounded linear operator from $X_{\om,T}^{s-1,1}\cap L_T^\I H_\om^s$ into $Z_\om^s$ with a bound which does not depend on $s$ and $T$. The existence of this operator ensures that actually $ Z^s_{\om,T}= L_T^\I H^s_\om\cap X^{s-1,1}_{\om,T}$.
Following \cite{MN08}, we define $\rho_T$ as
\EQS{\label{def_ext}
  \rho_T(u)(t):=U_\al(t)\chi(t)U_\al(-\mu_T(t))u(\mu_T(t)),
}
where $\mu_T$ is the continuous piecewise affine function defined by
\EQS{
  \mu_T(t)=
  \begin{cases}
    0 &\textrm{for}\quad t\notin]0,2T[,\\
    t &\textrm{for}\quad t\in [0,T],\\
    2T-t &\textrm{for}\quad t\in [T,2T].
  \end{cases}
}

\begin{lem}\label{extensionlem}
  Let $\de\ge 1$, and suppose that the dyadic sequence $\{\om_N\}$ of positive numbers satisfies $\om_N\le \om_{2N}\le \de\om_N$ for $N\ge1$.
  Let $0<T\le 1$ and $s\in\R$.
  Then,
  \EQQS{
    \rho_T:&X_{\om,T}^{s-1,1} \cap L_T^\I H_\om^s\to Z^s_\om\\
    &u\mapsto \rho_T(u)
  }
  is a bounded linear operator, i.e.,
  \EQS{\label{eq2.1}
    \|\rho_T(u)\|_{L_t^\I H_\om^s}
    +\|\rho_T(u)\|_{X^{s-1,1}_\om}\lesssim
    \|u\|_{L_T^\I H_\om^s}
    +\|u\|_{X_{T,\om}^{s-1,1}},
  }
  for all $u\in X^{s-1,1}_{\om,T}\cap L_T^\I H_\om^s$.
  Moreover, it holds that
  \EQS{\label{eq2.1single}
  \|\rho_T(u)\|_{L_t^\I H_\om^s}
  \lesssim
  \|u\|_{L_T^\I H_\om^s}
  }
  for all $u\in L_T^\I H_\om^s$.
  Here, the implicit constants in \eqref{eq2.1} and \eqref{eq2.1single} can be chosen independent of $0<T\le 1$ and $s\in\R$.
\end{lem}

\begin{proof}
  See Lemma 2.4  in \cite{MPV19} for $ \om_N\equiv 1$ but it is obvious that the result does not depend on $ \om_N$.
\end{proof}

\subsection{Basic estimates}

In this subsection, we collect some fundamental estimates.
Well-known estimates are adapted for our setting $H_\om^s(\T)$ and $f(u)$.

\begin{lem}
  Let $s>0$.
  Let $\de\ge 1$, and suppose that the dyadic sequence $\{\om_N\}$ of positive numbers satisfies $\om_N\le \om_{2N}\le \de\om_N$ for $N\ge1$.
  Then we have the estimate
  \EQS{\label{eq2.2}
    \|uv\|_{H_\om^s}
    \lesssim \|u\|_{H_\om^s}\|v\|_{L^\I}+\|u\|_{L^\I}\|v\|_{H_\om^s}.
  }
  In particular for any fixed real entire function $ f $ with $ f(0)=0 $, there exists a real entire function $ G=G[f] $ that is increasing non negative on $ \R_+ $  such that
    \EQS{\label{eq2.2ana}
    \|f(u)\|_{H_\om^s}
    \lesssim   G(\|u\|_{L^\I}) \|u\|_{H_\om^s}.
  }
 \end{lem}

\begin{proof}
  The  proof  of \eqref{eq2.2} is identical with that of Lemma A.8 in \cite{Tao}.
  For reader's convenience, we provide the proof.

  The triangle inequality gives
  \EQQS{
    \|uv\|_{H_\om^s}
    \lesssim \bigg(\sum_{N\lesssim1}\om_N^2 \|P_N(uv)\|_{L^2}^2\bigg)^{1/2}
    +\bigg(\sum_{N\gg1}\om_N^2 N^{2s}\|P_N(uv)\|_{L^2}^2\bigg)^{1/2}
    =:A+B.
  }
  It is clear that $A\lesssim \|u\|_{H_\om^s}\|v\|_{L^\I}$ since $\om_{2^j}\le \de^j\om_1$.
  We further split
  \EQQS{
    \|P_N(uv)\|_{L^2}
    \lesssim \|P_N((P_{\ll N}u)v)\|_{L^2}
    +\sum_{M\gtrsim N}\|P_N((P_{M}u)v)\|_{L^2}.
  }
  For the first term, by impossible frequency interactions, we see that
  \EQQS{
    \|P_N((P_{\ll N}u)v)\|_{L^2}
    \le \|(P_{\ll N}u)P_{\sim N}v\|_{L^2}
    \lesssim \|u\|_{L^\I}\sum_{M\sim N}\|P_{M}v\|_{L^2}.
  }
  So, the contribution of this term to $B$ is bounded by $C\|u\|_{L^\I}\|v\|_{H_\om^s}$.
  For the second term, we simply bound
  \EQQS{
    \sum_{M\gtrsim N}\|P_N((P_{M}u)v)\|_{L^2}
    \lesssim \|v\|_{L^\I}\sum_{M\gtrsim N}\|P_{M}u\|_{L^2}
  }
  and so we see from $2^s\om_N N^s\le \om_{2N}(2N)^s$ and the Young inequality that
  \EQQS{
    \bigg(\sum_{N\gg1}\bigg(\sum_{M\gtrsim N}\om_N N^s
      \|P_M u\|_{L^2}\bigg)^2\bigg)^{1/2}
    &\lesssim \bigg(\sum_{k\gg1}\bigg(\sum_{j\gtrsim1}2^{-sj}\om_{2^{k+j}}2^{s(j+k)}\|P_{2^{k+j}} u\|_{L^2}\bigg)^2\bigg)^{1/2}\\
    &\lesssim\|u\|_{H_\om^s},
  }
  which implies that $B\lesssim \|u\|_{H_\om^s}\|v\|_{L^\I}+\|u\|_{L^\I}\|v\|_{H_\om^s}$.

  Now by the Minkowski inequality and \eqref{eq2.2}
      \EQQS{
   \|f(u)\|_{H_\om^s}
   &\le \sum_{k\ge 1} \frac{|f^{(k)}(0)|}{k!} \|u^k \|_{H_\om^s} 
    \le    \sum_{k\ge 1} \frac{|f^{(k)}(0)|}{k!}  C^{k-1} \|u\|_{L^\I}^{k-1} \|u\|_{H_\om^s}\\
    &\le  G (\|u\|_{L^\I}) \|u\|_{H_\om^s},
  }
  with $ G(x):=  \sum_{k\ge 1} \frac{|f^{(k)}(0)|}{k!}  C^{k-1} x^{k-1} $.
\end{proof}
\begin{lem}
  Assume that $s_1+s_2\ge 0, s_1\wedge s_2\ge s_3, s_3<s_1+s_2-1/2$.
  Then
  \EQS{\label{eq2.3}
    \|uv\|_{H^{s_3}}\lesssim \|u\|_{H^{s_1}}\|v\|_{H^{s_2}}.
  }
  In particular, for $f,g\in H^s(\T) $ with $s>1/2 $ and any fixed real entire function $ F $, there exists a real entire function $ G=G[f] $ that is increasing  non negative on $ \R_+ $ such that
   \EQS{\label{eq2.3ana}
    \|f(u)-f(v)\|_{H^{s-1}}\le G(\|v\|_{H^{s}}+ \|v\|_{H^s})\|u-v\|_{H^{s-1}}.
  }
\end{lem}
\begin{proof} \eqref{eq2.3} can be found in [\cite{GLM14}, Lemma 3.4]. To
prove \eqref{eq2.3ana} it suffices to check that for $ s>1/2$, $(s-1)+s=2s-1>0 $ and $ s-1<s-1/2+(s-1) $. Therefore
  \EQQS{
   \|f(u)-f(v)\|_{H^{s-1}}  &\le  \sum_{k\ge 1} \frac{|f^{(k)}(0)|}{k!} \|u^k -v^k\|_{H^{s-1}} \\
    &\le    \sum_{k\ge 1} \frac{|f^{(k)}(0)|}{k!}  \sum_{j=0}^{k-1}
\|u^j v^{k-1-j} (u-v)\|_{H^{s-1}} \\
    &\le \sum_{k\ge 1} \frac{|f^{(k)}(0)|}{k!}   C^{k-1} (\|u\|_{H^s}+\|v\|_{H^s})^{k-1} \|u-v\|_{H^{s-1}}\\
   & \le G(\|u\|_{H^s}+\|v\|_{H^s}) \|u-v\|_{H^{s-1}}
      }
  with $ G(x):=  \sum_{k\ge 1} \frac{|f^{(k)}(0)|}{k!}  C^{k-1} x^{k-1} $.
\end{proof}

We will  frequently use the following lemma, which can be seen as a variant of the integration by parts.

\begin{lem}\label{lem_comm1}
  Let $N\in 2^{\N}\cup\{0\}$.
  Then,
  \EQQS{
    \bigg|\int_\T \Pi(u,v)wdx\bigg|
    \lesssim\|u\|_{L_x^2}\|v\|_{L_x^2}\|\p_x w\|_{L_x^\I},
  }
  where
  \EQS{\label{def_pi}
    \Pi(u,v):=v \p_x P_N^2u +u \p_x P_N^2v .
  }
\end{lem}

\begin{proof}
  It suffices to show that
  \EQS{\label{eq4.10}
    \|[\p_x P_N^2, w]v\|_{L^2}
    \lesssim \|\p_x w\|_{L^\I}\|v\|_{L^2}
  }
  since the integration by parts shows
  \EQQS{
    \int_\T \Pi(u,v)wdx
    =\int_\T u \p_x P_N^2 v wdx
      -\int_\T u \p_x P_N^2 (v w)dx
    =-\int_\T u [\p_x P_N^2, w]v dx.
  }
  The Poisson summation formula and the mean value theorem imply that
  \EQQS{
    |([\p_x P_N^2, w]v)(x)|
    &=\bigg|\int_\T \F_x^{-1}(\xi \phi_N^2(\xi))(x-y)(w(y)-w(x))v(y)dy\bigg|\\
    &=\bigg|\int_\T \sum_{k\in\Z} \F_{x,\R}^{-1}(\xi \phi_N^2(\xi))(x-y+2\pi k)(w(y)-w(x))v(y)dy\bigg|\\
    &\le\int_\R |\F_{x,\R}^{-1}(\xi \phi_N^2(\xi))(x-y)(w(y)-w(x))v(y)|dy\\
    &\le\|\p_x w\|_{L^\I}
      \int_\R |(x-y)\F_{x,\R}^{-1}(\xi \phi_N^2(\xi))(x-y)v(y)|dy,
  }
  where $\F_{x,\R}^{-1}$ is the inverse Fourier transform on $\R$.
  A direct calculation gives
  \EQQS{
  \|x\F_{x,\R}^{-1}(\xi \phi_N^2(\xi))(x)\|_{L^1(\R)}
   =\|x\F_{x,\R}^{-1}(\xi \phi^2(\xi))(x)\|_{L^1(\R)}<\I.
  }
  This and the Minkowski inequality show \eqref{eq4.10}, which completes the proof.
\end{proof}

\if0
\begin{lem}[Lemma 3.5 in \cite{GLM14}]
  Let $0<s<1$, $p,p_1,p_2\in(1,\I)$ with $1/p_1+1/p_2=1/p$ and $s_1,s_2\in[0,s]$ with $s=s_1+s_2$.
  Then
  \EQS{\label{eq2.4}
    \|D_x^s(fg)-fD_x^sg-gD_x^sf\|_{L^p}
    \lesssim \|D_x^{s_1}f\|_{L^{p_1}}\|D_x^{s_2}g\|_{L^{p_2}}.
  }
  Moreover, for $s_1=0$, the value $p_1=\I$ is allowed.
\end{lem}
\fi

\section{Linear and improved Strichartz estimates}

In this section, we establish improved Strichartz estimates which play an important role in our estimates. Since we treat the general operator $L_{\al+1}$, we first have to  restate the standard Strichartz estimate.

\begin{defn}
  For $\al\in[1,2]$, we define
  \EQQS{
    \be(\al):=\frac{1}{4(\al+1)}, \quad
    b(\al):=\be(\al)+\frac{1}{4}
     \Bigg(=\frac{\al+2}{4(\al+1)}\Bigg).
  }
\end{defn}

%The following linear estimate in Bourgain's space is established in \cite{B93}.

We make use of the Strichartz estimate in the Bourgain space, which is originally obtained in \cite{B93}.

\begin{lem}\label{lem_stri1}
  There exists $C>0$ such that for any $v\in X^{0,b(\al)}$,
  \EQQS{
    \|v\|_{L_{t,x}^4} \le C\|v\|_{X^{0,b(\al)}}.
  }
\end{lem}

Lemma \ref{lem_stri1} immediately follows from the following lemma (see the Appendix in \cite{M07} for similar considerations).
%Its proof is given by Lemma 3.1 in \cite{Schi20}.

\begin{lem}\label{lem_stri1_1}
  There exists $C=C(\xi_0)>0$ such that for any real-valued functions $u,v \in L^2(\R; l^2(\Z))$ and any $L_1,L_2\in 2^\N$,
  \EQQS{
    \|(\psi_{L_1}u)*_{\ta,\xi}(\psi_{L_2}v)\|_{L_\ta^2 l_\xi^2}
    \le C(L_1 \wedge L_2)^{1/2}(L_1 \vee L_2)^{2\be(\al)}
    \|\psi_{L_1}u\|_{L_\ta^2 l_\xi^2}
    \|\psi_{L_2}v\|_{L_\ta^2 l_\xi^2}.
  }
\end{lem}

For the proof of Lemma \ref{lem_stri1_1}, we use the following:

\begin{lem}\label{lem_counting}
  Let $I$ and $J$ be two intervals on the real line and $g\in C^1(J;\R)$.
  Then
  \EQQS{
    \# \{x\in J\cap\Z;g(x)\in I\}\le\frac{|I|}{\inf_{x\in J}|g'(x)|}+1.
  }
\end{lem}

\begin{proof}
  See Lemma 2 in \cite{ST01}.
\end{proof}

\begin{proof}[Proof of Lemma \ref{lem_stri1_1}]
  Following \cite{B93}, we may assume that $\supp_\xi \psi_{L_1} u, \supp_\xi \psi_{L_2} v \subset \N$ since $u$ and $v$ are real-valued.
  Following the argument of Lemma 3.1 in \cite{Schi20},
  %Lemma 3.1 in \cite{M07},
  the Cauchy-Schwarz inequality gives
  \EQQS{
  &\|(\psi_{L_1}u)*_{\ta,\xi}(\psi_{L_2}v)\|_{L_\ta^2 l_\xi^2}^2\\
  &=\sum_{\xi\ge0}\int_\ta \bigg|\sum_{\xi\ge \xi_1}\int_{\ta_1}\psi_{L_1}(\ta_1,\xi_1)u(\ta_1,\xi_1)\psi_{L_2}(\ta-\ta_1,\xi-\xi_1) v(\ta-\ta_1,\xi-\xi_1)d\ta_1\bigg|^2 d\ta\\
  &\lesssim\sup_{(\ta,\xi)\in\R\times\N}A(\ta,\xi)
    \|\psi_{L_1} u\|_{L_\ta^2 l_\xi^2}^2
    \|\psi_{L_2} v\|_{L_\ta^2 l_\xi^2}^2,
  }
  where
  \EQQS{
    A(\ta,\xi)
    &\lesssim mes\{(\ta_1,\xi_1)\in \R\times\N;\xi-\xi_1\ge 0, \LR{\ta_1 - p_{\al+1}(\xi_1)}\sim L_1\\
    &\quad\ \textrm{and}\ \LR{\ta-\ta_1 - p_{\al+1}(\xi-\xi_1)}\sim L_2\}\\
    &\lesssim (L_1\wedge L_2)\#B(\ta,\xi)
  }
  with
  \EQQS{
    B(\ta,\xi)=\{\xi_1\ge 0;\xi-\xi_1\ge 0\ \textrm{and}\ \LR{\ta - p_{\al+1}(\xi_1) - p_{\al+1}(\xi-\xi_1)}\lesssim L_1\vee L_2\}.
  }
  For simplicity, we put $L:=L_1\vee L_2$.
  When $\xi\le L^{1/(\al+1)}+2\xi_0+2$, it is clear that
  \EQQS{
    \# B(\ta,\xi)
    \le L^{1/(\al+1)}+2\xi_0+2\le C(\xi_0) L^{1/(\al+1)}
  }
  since $0\le\xi_1\le \xi$ and $L\ge 1$.
  On the other hand, when $\xi\ge L^{1/(\al+1)}+2\xi_0+2$,
  putting $g(\xi_1):=\ta - p_{\al+1}(\xi_1) - p_{\al+1}(\xi-\xi_1)$,
  we obtain
  \EQQS{
    \# B(\ta,\xi)
    &\lesssim \#\{\xi_1\ge 0; \xi\ge 2\xi_1\ \textrm{and}\ |g(\xi_1)|\lesssim L\}\\
    &\le [\xi_0]+1+\#\{\xi_1\ge 0; \xi\ge 2\xi_1,\xi_1\ge \xi_0\ \textrm{and}\ |g(\xi_1)|\lesssim L\}\\
    &\lesssim \#\{\xi_1\in[\xi_0,\xi/2]; |\xi-2\xi_1|^{\al+1}\ge L\ \textrm{and}\ |g(\xi_1)|\lesssim L\}+[\xi_0]+L^{1/(\al+1)},
  }
  where $[a]$ is the integral part of $a$.
  Here, we used the symmetry in the first inequality.
  It is worth noticing that the set in the right hand side is not empty since $(\xi-L^{1/(\al+1)})/2\ge \xi_0$ by the assumption.
  Lemma \ref{lem_counting} implies that
  \EQQS{
    \#\{\xi_1\in[\xi_0,\xi/2]; |\xi-2\xi_1|^{\al+1}\ge L\ \textrm{and}\ |g(\xi_1)|\lesssim L\}\lesssim L^{1/(\al+1)}.
  }
  Indeed, we have
  \EQQS{
    |g'(\xi_1)|
    %=|p_{\al+1}'(\xi_1)-p_{\al+1}'(\xi-\xi_1)|
    =\bigg|\int_{\xi_1}^{\xi-\xi_1}p_{\al+1}''(\theta)d\theta\bigg|
    \sim (\xi-\xi_1)^{\al}-\xi_1^{\al}\ge (\xi-2\xi_1)^\al
    \ge L^{\al/(\al+1)}
  }
  since for $\theta\ge\xi_1$, $p_{\al+1}''$ does not change sign since $|p_{\al+1}''(\theta)|\sim|\theta|^{\al-1}$ and $p_{\al+1}''$ is continuous
outside $0$.
  This completes the proof.
\end{proof}
 Lemma \ref{lem_stri1}  enables to establish the following Strichartz estimate:
\begin{lem}\label{lem_stri2}
  Let $T>0$.
  Then there exists $C>0$ such that for any $u\in L^2(\T)$,
  \EQQS{
    \|U_\al(t)u\|_{L_{T,x}^4}
    \le CT^{1/2-b(\al)}\|u\|_{L_x^2}.
  }
\end{lem}

\begin{proof}
  See Lemma 2.1 in \cite{MR09}.
\end{proof}

We are now ready to prove our improved Strichartz estimates for solutions
to \eqref{eq1}.
We point out that it is crucial to state estimates in $l^4(\N)$--form since we are not allowed to lose any derivatives in order to reach $s=s(\al)$. For that purpose, how to choose $c_{j,N}$ plays an important role in the proof below. This type of argument can be found in [Lemma 2.4, \cite{NW07}] for instance.

\begin{prop}\label{prop_stri}
  Let $\de\ge 1$, and suppose that the dyadic sequence $\{\om_N\}$ of positive numbers satisfies $\om_N\le \om_{2N}\le \de\om_N$ for $N\ge1$.
  Let $s>1/2$, $\al\in[1,2]$ and $0<T<1$.
  Let $u\in C([0,T]; H_\om^{s}(\T))$ satisfy \eqref{eq1}--\eqref{initial}
with $u_0\in H_\om^{s}(\T)$ on $[0,T]$.
  Then,
  \EQS{\label{eq_stri1}
    \Bigg(\sum_{N} \om_N^4\|D_x^{s-\be(\al)}P_N u\|_{L_{T,x}^4}^4\Bigg)^{1/4}
    \lesssim T^{1/8}  G(\|u\|_{L^\I_{T,x}}) \|u\|_{L_T^\I H_\om^s}
  }
  and
  \EQS{\label{eq_stri2}
  \Bigg(\sum_N\|D_x^{1/3}P_N u\|_{L_{T}^3 L_x^\infty}^3\Bigg)^{1/3}
  \lesssim T^{5/24}G(\|u\|_{L^\I_{T,x}}) \|u\|_{L_T^\I H_x^{7/12+\be(\al)}},
  }
  where $ G=G[f] $ is an entire function that is increasing  and non negative on $ \R_+ $.
\end{prop}

\begin{proof}
  It suffices to consider the case $N\gg1$.
  % \CG{First we divide the summation over $N\gg1$ into two cases $NT< 1$ and $NT\ge1$.
  % % That is, we see from the Minkowski inequality that
  % % \EQQS{
  % %   &\Bigg(\sum_{N\gg1} \om_N^4\|D_x^{s-\be(\al)}P_N u\|_{L_{T,x}^4}^4\Bigg)^{1/4}\\
  % %   &\le \Bigg(\sum_{N\gg1, N< T^{-1}} \om_N^4\|D_x^{s-\be(\al)}P_N u\|_{L_{T,x}^4}^4\Bigg)^{1/4}
  % %    +\Bigg(\sum_{N\gg1, N\ge T^{-1}} \om_N^4\|D_x^{s-\be(\al)}P_N u\|_{L_{T,x}^4}^4\Bigg)^{1/4}.
  % % }
  % For the case $NT<1$, the H\"older inequality (in time) and the Bernstein inequality give
  % \EQQS{
  %   \|D_x^{s-\be(\al)}P_N u\|_{L_{T,x}^4}
  %   \lesssim T^{1/4}N^{-\be(\al)+1/4}\|D_x^{s}P_N u\|_{L_{T}^\I L_x^2}
  %   \lesssim T^{1/24}N^{-1/24}\|D_x^{s}P_N u\|_{L_{T}^\I L_x^2}.
  % }
  % Here, we used $\be(\al)\ge 1/12$ for $\al\in[1,2]$.
  % Then, thanks to a negative power of $N$, the left hand side of \eqref{eq_stri1} restricted to $N\gg1$ and $N<T^{-1}$ is bounded by $T^{1/24}\|u\|_{L_T^\I H_\om^s}$. }
  We divide the interval in small intervals of length $\sim T N^{-1}$.
  In other words, we define $\{I_{j,N}\}_{j\in J_N}$ so that $\bigcup_{j\in J_N}I_{j,N}=[0,T]$, $|I_{j,N}|\sim T N^{-1}$ and $\# J_N\lesssim N$.
  By the hypothesis, we see that $\|D_x^s P_N u(t)\|_{L_x^2}^4\in C([0,T])$.
  For $j\in J_N$, we choose $c_{j,N}\in I_{j,N}$ at which $\|D_x^s P_N u(t)\|_{L_x^2}^4$ attains its minimum on $I_{j,N}$.
  We see from \eqref{eq1}--\eqref{initial} that
  \EQQS{
    P_N u(t)
    =e^{-(t-c_{j,N})L_{\al+1}}P_N u(c_{j,N})+\int_{c_{j,N}}^t e^{-(t-t')L_{\al+1}}P_N \p_x (f(u))(t')dt'
  }
  for $t\in I_{j,N}$.
  Lemma \ref{lem_stri2} and the Bernstein inequality show that
  \EQQS{
    &\Bigg(\sum_{N\gg1}\sum_{j}\om_N^4\|D_x^{s-\be(\al)}e^{-(t-c_{j,N})L_{\al+1}}P_N u(c_{j,N})\|_{L^4(I_{j,N};L_x^4)}^4\Bigg)^{1/4}\\
    &\lesssim\Bigg(\sum_{N\gg1}\sum_{j}\om_N^4
     |I_{j,N}|^{2-4 b(\alpha)} N^{-4\beta(\alpha)}
      \|D_x^s P_N u(c_{j,N})\|_{L_x^2}^4\Bigg)^{1/4}\\
    &\lesssim T^{-\be(\al)} \Bigg(\sum_{N\gg1}\sum_{j}\om_N^4|I_{j,N}|
      \|D_x^s P_N u(c_{j,N})\|_{L_x^2}^4\Bigg)^{1/4}\\
    &\lesssim T^{-\be(\al)}
    \Bigg(\sum_{N\gg1}\int_0^T\om_N^4
      \|D_x^s P_N u(t)\|_{L_x^2}^4dt\Bigg)^{1/4}
    \lesssim T^{1/4-\be(\al)}\|u\|_{L_T^\I H_\om^s}.
  }
  In the last inequality, we used the fact $l^2(\N)\hookrightarrow l^4(\N)$.
  Next, we estimate the contribution of the Duhamel term. To simplify the
expressions we set
   $\tilde{f}=f-f(0) $.
  Lemma \ref{lem_stri2}, the H\"older inequality in time and \eqref{eq2.2ana} imply that
  \EQQS{
    &\Bigg(\sum_{N\gg1}\sum_{j}\om_N^4\bigg\|\int_{c_{j,N}}^t e^{-(t-t')L_{\al+1}}D_x^{s-\be(\al)}P_N \p_x (\tilde{f}(u))(t')dt'\bigg\|_{L^4(I_{j,N};L_x^4)}^4\Bigg)^{1/4}\\
    &\lesssim\Bigg(\sum_{N\gg1}\sum_{j}
     \om_N^4|I_{j,N}|^{1-4\be(\al)}N^{4-4\be(\al)}
     \bigg(\int_{I_{j,N}}
     \|D_x^s P_N  \tilde{f}(t')\|_{L_x^2}dt'\bigg)^4\Bigg)^{1/4}\\
    &\lesssim T^{1-\be(\al)}
     \Bigg(\sum_{N\gg1}\sum_{j}\om_N^4\int_{I_{j,N}}
     \|D_x^s P_N   \tilde{f}(t')\|_{L_x^2}^4dt'\Bigg)^{1/4}\\
    &\lesssim T^{1-\be(\al)} \Bigg(\int_0^T \sum_{N\gg1}\om_N^4     \|D_x^s P_N   \tilde{f}(t')\|_{L_x^2}^4dt'\Bigg)^{1/4}\\
    &\lesssim T^{1/4+(1-\be(\al))}\| \tilde{f}(u)\|_{L_T^\I H_\om^s}
    \lesssim T^{1/4+(1-\be(\al))}
     G(\|u\|_{L^\I_{T,x}})\|u\|_{L_T^\I H_\om^s}.
  }
This concludes the proof of  \eqref{eq_stri1} by using  that $ \beta(\alpha) \le 1/8 $ for $ 1\le \alpha\le 2 $.
  To prove \eqref{eq_stri2}, we notice that
  % \CG{when $NT<1$, noting that
  % \EQQS{
  %   \|D_x^{1/3}P_N u\|_{L_T^3 L_x^\infty}
  %   &\lesssim T^{1/3}N^{-\be(\al)+1/4}\|D_x^{7/12+\be(\al)}P_N u\|_{L_T^\I L_x^2}\\
  %   &\lesssim T^{1/8}N^{-1/24} \|D_x^{7/12+\be(\al)}P_N u\|_{L_T^\I L_x^2},
  % }
  % which implies that the left hand side of \eqref{eq_stri2} restricted to $N\gg1$ and $NT<1$ is bounded by $T^{1/8}\|u\|_{L_T^\I H_x^{7/12+\be(\al)}}$.
  % On the other hand, when $NT\ge 1$, }
  the Bernstein inequality in space and the H\"older inequality in time give
  \EQQS{
    \|D_x^{1/3}P_N u\|_{L^3(I_{j,N}; L_x^\infty)}
    \lesssim |I_{j,N}|^{1/12}\|D_x^{7/12}P_N u\|_{L^4(I_{j,N}; L_x^4)}
    \lesssim T^{1/12} \|D_x^{1/2}P_N u\|_{L^4(I_{j,N}; L_x^4)},
  }
  where $|I_{j,N}|\sim T N^{-1}$ and $\# J\sim N$.
  This implies that
  \EQQS{
    \Bigg(\sum_{N\gg1}\|D_x^{1/3}P_N u\|_{L_{T}^3 L_x^\I}^3\Bigg)^{1/3}
    \lesssim T^{1/12} \Bigg(\sum_{N\gg1} \sum_j \|D_x^{1/2}P_N u\|_{L^4(I_{j,N}; L_x^4)}^3\Bigg)^{1/3}.
  }
  We can then  obtain \eqref{eq_stri2} by the same way as \eqref{eq_stri1} with
$\om_N=1$.
\end{proof}

For the estimate for the difference, we shall use the following estimates.

\begin{cor}\label{cor_stri2}
  Assume that $s>1/2$.
  Let $\al\in[1,2]$ and $0<T<1$.
  Let $u,v \in C([0,T]; H^{s}(\T))$ satisfy \eqref{eq1}--\eqref{initial} with $u_0,v_0\in H^s(\T)$ on $[0,T]$, respectively.
  Then,
  \EQS{\label{eq_stri3}
  \begin{aligned}
    \Bigg(\sum_{N} [(1\vee N)^{s-1-\be(\al)}\|P_N w\|_{L_{T,x}^4}]^4\Bigg)^{1/4}
    \lesssim T^{1/8} G(K) \|w\|_{L_T^\I H_x^{s-1}}
  \end{aligned}
  }
  and
  \EQS{\label{eq_stri4}
  \Bigg(\sum_{N} [(1\vee N)^{-5/12}\|P_N w\|_{L_{T}^3 L_x^4}]^3\Bigg)^{1/3}
  \lesssim T^{5/24}G(K)\|w\|_{L_T^\I H_x^{-5/12+\be(\al)}},
  }
  where $w=u-v$, $K=\|u\|_{L_T^\I H_x^{s}}+\|v\|_{L_T^\I H_x^{s}}$ and $G=G[f]$ is an entire function that is increasing
and non negative on $\R_+$.
\end{cor}

\begin{proof}
  The proof is the same as that of Proposition \ref{prop_stri}, but with using \eqref{eq2.3ana} instead of \eqref{eq2.2}.
\end{proof}

\section{A priori estimate}
\subsection{Preliminary technical estimates}

Let us denote by $\1_T$ the characteristic function of the interval $]0,T[$.
As pointed out in \cite{MV15}, $\1_T$ does not commute with $Q_L$.
To avoid this difficulty, following \cite{MV15}, we further decompose $\1_T$ as
\EQS{
  \1_T=\1_{T,R}^{\textrm{low}}+\1_{T,R}^{\textrm{high}},
  \quad \textrm{with} \quad
  \F_t(\1_{T,R}^{\textrm{low}})(\ta)=\chi(\ta/R)\F_t(\1_T)(\ta),
}
for some $R>0$ to be fixed later.

\begin{rem}
  In the proofs of Propositions \ref{prop_apri} and \ref{difdif}, since  we integrate on an interval of time $ ]0,t[ $ with $ 0<t<T $, we will  put the cutoff function $\1_t$ on two functions of the integral.
  Actually, it is enough to put $\1_t$ on one function if one merely seeks the commutativity of $ Q_L $ with  the low frequency part of  $\1_t$  such as in \eqref{eq4.6}.
  The advantage of putting $\1_t$ on two functions is that
  one can get a positive power of $T$ in the right hand side in \eqref{P} and \eqref{eq_difdif}, which enables us to bypass the scaling argument in the proof of the well-posedness.  This is particulary comfortable here since \eqref{eq1} is in general not scaling invariant.
\end{rem}

\begin{lem}[Lemma 3.5 in \cite{MPV19}]
  Let $1\le p\le \I$.
  Let $L$ be a nonhomogeneous dyadic number.
  Then the operator $Q_{\le L}$ is bounded in $L_t^{p} L_x^2 $ uniformoly in $L$.
  In other words,
  \EQS{\label{eq4.3}
    \|Q_{\le L}u\|_{L_t^{p} L_x^2}\lesssim \|u\|_{L_t^{p} L_x^2},
  }
  for all $u\in L_t^{p} L_x^2$ and the implicit constant appearing in \eqref{eq4.3} does not depend on $L$.
\end{lem}

\begin{lem}[Lemma 3.6 in \cite{MPV19}]
  For any $R>0$ and $T>0$, it holds
  \EQS{\label{eq4.1}
    \|\1_{T,R}^{\textrm{high}}\|_{L^1}\lesssim T\wedge R^{-1},
  }
  and
  \EQS{\label{eq4.2}
    \|\1_{T,R}^{\textrm{high}}\|_{L^\I}
    +\|\1_{T,R}^{\textrm{low}}\|_{L^\I}\lesssim 1.
  }
\end{lem}

\begin{lem}[Lemma 3.7 in \cite{MPV19}]
  Assume that $T>0$, $R>0$, and $L\gg R$.
  Then, it holds
  \EQS{\label{eq4.6}
    \|Q_L(\1_{T,R}^{\textrm{low}}u)\|_{L_{t,x}^2}
    \lesssim \|Q_{\sim L} u\|_{L_{t,x}^2},
  }
  for all $u\in L^2(\R_t\times\T_x)$.
\end{lem}

It is well-known that the resonance function ($\Om_{k+1}$ defined as below) with respect to higher order nonlinear terms such as $\p_x(u^{k+1})$ can be resonant (i.e., very small). In what follows, we clarify non-resonant contributions in which we can recover the derivative loss by  using a priori estimates in Bourgain's spaces of solution to \eqref{eq1} proved in Lemma \ref{lem1}.

\begin{defn}
  Let $j\in\N$.
  We define $\Om_j(\xi_1,\dots,\xi_{j+1}):\Z^{j+1}\to\R$ as
  \EQQS{
    \Om_j(\xi_1,\dots,\xi_{j+1})
    :=\sum_{n=1}^{j+1}p_{\al+1}(\xi_{n})
  }
  for $(\xi_1,\dots,\xi_{j+1})\in\Z^{j+1}$, where $p_{\al+1}$ satisfies Hypothesis 1.
\end{defn}

\begin{lem}\label{lem_res1}
  Let $k\ge 1 $ and $(\xi_1,\dots,\xi_{k+2})\in\Z^{k+2}$ satisfy $\sum_{j=1}^{k+2}\xi_j=0$.
  Assume that $|\xi_1|\sim |\xi_2|\gtrsim |\xi_3| $ if $ k= 1 $ or $|\xi_1|\sim |\xi_2|\gtrsim |\xi_3|\gg  k  \max_{j\ge4 }|\xi_j|$ if $ k\ge 2 $.
  Then,
  \EQQS{
    |\Om_{k+1}(\xi_1,\dots,\xi_{k+2})|\gtrsim |\xi_3||\xi_1|^{\al}
  }
  for $|\xi_1|\gg (\max_{\xi\in[0,\xi_0]}|p_{\al+1}'(\xi)|)^{1/\al}$.
\end{lem}

\begin{proof}
  First we consider the case $k\ge 2$.
  We separate different cases:

  \noindent
  \textbf{Case 1:} $|\xi_2|\gg|\xi_3|$.
  By hypothesis, it follows that $|\xi_2|\gg \xi_0$.
  The mean value theorem implies that there exists $\eta\in\R$ such that $|\eta|\sim |\xi_2|$ and that
  \EQQS{
    |p_{\al+1}(\xi_2+\cdots+\xi_{k+2})-p_{\al+1}(\xi_2)|
    =|\xi_3+\cdots+\xi_{k+2}||p_{\al+1}'(\eta)|
    \sim |\xi_3||\xi_1|^{\al}.
  }
  Here, we used $|\xi_1|\sim|\xi_2|$ and $|\xi_3|\gg  k  \max_{j\ge4 }|\xi_j|$.
  Now, if $|\xi_j|\le \xi_0$ for $j\ge4$, then
  \EQQS{
    |p_{\al+1}(\xi_j)|
    \le |\xi_j|\max_{\xi\in[0,\xi_0]}|p_{\al+1}'(\xi)|
    \ll \frac{|\xi_3||\xi_1|^\al}{k}
  }
  since $p_{\al+1}(0)=0$.
  On the other hand, if $|\xi_j|\ge\xi_0$ for $j\ge 4$, then
  \EQQS{
    |p_{\al+1}(\xi_j)|
    &\le |p_{\al+1}(\xi_j)-p_{\al+1}(\xi_0)|+|p_{\al+1}(\xi_0)|\\
    &\lesssim |\xi_0|\max_{\xi\in[0,\xi_0]}|p_{\al+1}'(\xi)|
      +|\xi_j|^{\al+1}\ll \frac{|\xi_3||\xi_1|^\al}{k}.
  }
  Similarly, we can get $|p_{\al+1}(\xi_3)|\ll |\xi_3||\xi_1|^\al$.
  Gathering these estimates leads to $|\Om_{k+1}|\gtrsim |\xi_3||\xi_1|^\al$.

  \noindent
  \textbf{Case 2:} $|\xi_2|\sim |\xi_3|$.
  Then we have $|\xi_3|\gg\xi_0$.
  By impossible interactions, $\xi_1,\xi_2$ and $\xi_3$ do not have the same sign.
  By the symmetry and $|\xi_1|\sim|\xi_2|\sim|\xi_3|$, it suffices to consider the case $\xi_2,\xi_3>0$.
  We notice that
  \EQQS{
    -\Om_{k+1}
    &=\int_{\xi_0}^{\xi_2}(p_{\al+1}'(\theta+\xi_3+\cdots+\xi_{k+2})-p_{\al+1}'(\theta))d\theta\\
    &\quad+p_{\al+1}(\xi_0+\xi_3+\cdots+\xi_{k+2})
      -p_{\al+1}(\xi_3)-p_{\al+1}(\xi_0)
      -\sum_{j=4}^{k+2}p_{\al+1}(\xi_j)
  }
  with
  \EQQS{
    |p_{\al+1}(\xi_0+\xi_3+\cdots+\xi_{k+2})
      -p_{\al+1}(\xi_3)|
    \lesssim (|\xi_0|+k\max_{j\ge4}|\xi_j|)|\xi_3|^\al
    \ll |\xi_3||\xi_1|^\al
  }
  and
  \EQQS{
    p_{\al+1}'(\theta+\xi_3+\cdots+\xi_{k+2})-p_{\al+1}'(\theta)
    =\int_0^{\xi_3+\cdots+\xi_{k+2}}p_{\al+1}''(\theta+\mu)d\mu.
  }
  For $\theta\ge\xi_0$, $p_{\al+1}''$ does not change sign since $|p_{\al+1}''(\theta)|\sim|\theta|^{\al-1}$ and $p_{\al+1}''$ is continuous outside $0$.
  Therefore, for $\theta\in[\xi_0,\xi_2]$, we get
  \EQQS{
    \int_0^{\xi_3+\cdots+\xi_{k+2}}p_{\al+1}''(\theta+\mu)d\mu
    \sim \int_0^{\xi_3+\cdots+\xi_{k+2}}(\theta+\mu)^{\al-1}d\mu
    \sim \xi_3^\al
  }
  since $\xi_3\gg k \max_{j\ge 4}|\xi_4|$.
  Gathering these estimates, we obtain $|\Om_{k+1}|\gtrsim |\xi_3||\xi_1|^\al$.

  For the case $k=1$, we can argue exactly as above.
\end{proof}

\begin{lem}\label{lem_res2}
  Let $ k\ge 2 $ and $(\xi_1,\dots,\xi_{k+2})\in\Z^{k+2}$ satisfy $\sum_{j=1}^{k+2}\xi_j=0$.
  Assume that  $|\xi_1|\sim |\xi_2|\gg  |\xi_3|\gtrsim |\xi_4|$  if $ k=2 $ or $|\xi_1|\sim |\xi_2|\gg   |\xi_3|\gtrsim  |\xi_4| $ with $  |\xi_3+\xi_4|\gg k  \displaystyle \max_{j\ge5} |\xi_j|$ if $ k\ge 3$.
  Then,
  \EQQS{
    |\Om_{k+1}(\xi_1,\dots,\xi_{k+2})|
    \gtrsim |\xi_3+\xi_4||\xi_1|^\al
  }
  for $|\xi_1|\gg (\max_{\xi\in[0,\xi_0]}|p_{\al+1}'(\xi)|)^{1/\al}$.
\end{lem}

\begin{proof}
  First we consider the case $k\ge 3$.
  We separate different cases:

  \noindent
  \textbf{Case 1:} $|\xi_3|\gg|\xi_4|$.
  Then we have $|\xi_3|\gg k\max_{j\ge 5}|\xi_j|$ and $|\xi_3|\sim|\xi_3+\xi_4|$.
  Therefore, we can argue exatly as in Case 1 of Lemma \ref{lem_res1} and
obtain $|\Om_{k+1}|\gtrsim |\xi_3||\xi_1|^\al\sim |\xi_3+\xi_4||\xi_1|^\al$.

  \noindent
  \textbf{Case 2:} $|\xi_3|\sim|\xi_4|$.
  When $\xi_3\xi_4\ge0$, it holds that $|\xi_3+\xi_4|=|\xi_3|+|\xi_4|$.
  Then we have $|\xi_3|,|\xi_4|\gg k\max_{j\ge5}|\xi_j|$.
  So, we can still argue as in Case 1 of Lemma \ref{lem_res1} and obtain $|\Om_{k+1}|\gtrsim |\xi_3||\xi_1|^\al\sim |\xi_3+\xi_4||\xi_1|^\al$.
  When $\xi_3\xi_4<0$, by the mean value theorem, there exist $\eta_1,\eta_2\in\R$ such that
  $|\eta_1|\sim |\xi_1|$, $|\xi_4|\lesssim|\eta_2|\lesssim|\xi_3|$ and
  \EQQS{
    -\Om_{k+1}
    =(\xi_3+\xi_4+\cdots+\xi_{k+2})p_{\al+1}'(\eta_1)
     -(\xi_3+\xi_4)p_{\al+1}'(\eta_2)-\sum_{j=5}^{k+2}p_{\al+1}(\xi_j)
  }
  since $p_{\al+1}(\xi)=-p_{\al+1}(-\xi)$ for any $\xi\in\R$.
  Note that $|p_{\al+1}'(\eta_2)|\ll |\xi_1|^\al$.
  Indeed, $|p_{\al+1}'(\eta_2)|\le \max_{\xi\in[0,\xi_0]}|p_{\al+1}'(\xi)|\ll|\xi_1|^\al$ when $|\eta_2|\le\xi_0$, and $|p_{\al+1}'(\eta_2)|\sim|\eta_2|^\al\lesssim|\xi_3|^\al\ll|\xi_1|^\al$ when $|\eta_2|\ge \xi_0$.
  As in Case 1 of Lemma \ref{lem_res1}, we also have $k\max_{j\ge 5}|p_{\al+1}(\xi_j)|\ll |\xi_3+\xi_4||\xi_1|^\al$.
  From these estimates, we obtain $|\Om_{k+1}|\gtrsim |\xi_3+\xi_4||\xi_1|^\al$.

  For the case $k=2$, we can argue exactly as above.
\end{proof}

\begin{lem}\label{lem_envelope}
  Let $\de> 1$, and suppose that the dyadic sequence $\{\om_N\}$ of positive numbers satisfies $\om_N\le \om_{2N}\le \de\om_N$ for $N\ge1$ and $\om_N\to\I$ as $N\to\I$.
  Then, for any $1<\de'<\de$, there exists a dyadic sequence $\{\ti{\om}_N\}$ such that
  $\ti{\om}_N\le \om_N$, $\ti{\om}_N\le \ti{\om}_{2N}\le \de'\ti{\om}_N$ for $N\ge 1$ and $\ti{\om}_N\to\I$ as $N\to\I$.
\end{lem}

\begin{proof}
  Let $1<\de'<\de$.
  Set $\de_N:=\om_{2N}/\om_N$ for $N\ge1$.
  Define $\{\ti{\om}_{N}\}$ so that $\ti{\om}_0:=\om_0$, $\ti{\om}_1:=\om_1$ and $\ti{\om}_{2N}:=(\de_{N}\wedge \de')\ti{\om}_N$ for $N\ge 1$.
  Since $\de_N,\de'\ge 1$, it is clear that $\ti{\om}_N\le \ti{\om}_{2N}\le \de'\ti{\om}_N$.
  We can show $\ti{\om}_N\le \om_N$ by the induction on $N$.
  Indeed, $\ti{\om}_0=\om_0$, $\ti{\om}_1= \om_1$ and
  \EQQS{
    \ti{\om}_{2N}=(\de_N\wedge \de') \ti{\om}_N
    \le \de_N  \ti{\om}_N\le \de_N \om_N =\om_{2N}
  }
  for $N\ge 1$.
  Now, we show $\ti{\om}_N\to\I$ as $N\to\I$.

  \noindent
  \textbf{Case 1:} $\#\{N;\de_N>\de'\}<\I$.
  Put $N_0:=1\vee\max\{N;\de_N>\de'\}$
  and $C:=\ti{\om}_{2N_0}/\om_{2N_0}$.
  Then, we can deduce inductively that $\ti{\om}_N=C\om_N$ for $N\ge 2N_0$, which shows $\ti{\om}_N\to\I$ as $N\to\I$ by the hypothesis.

  \noindent
  \textbf{Case 2:} $\#\{N;\de_N>\de'\}=\I$.
  Define an increasing sequence $\{j_n\}_{n\ge 2}\subset\N$ so that $j_n:=\#\{m\in\N;1\le m\le n-1, \de_{2^m} >\de' \}$ for $n\ge 2$.
  Then, we have $j_n\to\I$ as $n\to\I$.
  Observe that
  \EQQS{
    \ti{\om}_{2^n}=\prod_{m=0}^{n-1}(\de_{2^m}\wedge \de')\ti{\om}_1
    \ge (\de')^{j_n}\om_1\to\I
  }
  as $n\to\I$ since $\de'>1$, which completes the proof.
\end{proof}

\begin{rem}\label{rem_envelope}
  For the given dyadic sequence $\{\om_N\}$ of positive numbers, Lemma \ref{lem_envelope} enables us to assume $\de\le 2$, by defining a new dyadic sequence.
  We use this fact in the proof of Proposition \ref{prop_apri}.
  Let $N, M$ be dyadic numbers satisfying $lN\gtrsim M\gtrsim 1$ for some
$l\ge 2$.
  Using $\om_{2N}\le \de\om_N$, it holds
  \EQQS{
    \frac{\om_M}{\om_N}\lesssim \de^{\log_2l}\lesssim l
  }
  which is uniformly in $\de$.
\end{rem}

\subsection{Estimates on solutions to \eqref{eq1}}

\begin{lem}\label{lem1}
Let $1\le \de\le 2$, and suppose that the dyadic sequence $\{\om_N\}$ of positive numbers satisfies $\om_N\le \om_{2N}\le \de\om_N$ for $N\ge1$.
Let $ 0<T<1$, $ s>1/2 $  and $ u\in L^\infty_T H^s_\om $  be a solution to \eqref{eq1} associated with an initial datum
 $ u_0\in H^s_\om(\T) $.
Then $ u\in Z^s_{\om,T}$ and it holds
\begin{equation}\label{estXregular}
\|u\|_{Z^s_{\om,T}}  \lesssim \|u\|_{L^\infty_T H^s_\om} +G(\|u\|_{L^\infty_{T,x}}) \|u\|_{L^\infty_T H^{s}_\om}\;.
\end{equation}
 Moreover, for any couple $(u, v) \in   (L^\infty_T H^s)^2 $ of solutions
to \eqref{eq1} associated with a couple of initial data
 $ (u_0,v_0)\in (H^s(\T))^2 $  it holds
\begin{equation}\label{estdiffXregular}
\|u-v\|_{Z^{s-1}_{T}}  \lesssim \|u-v\|_{L^\infty_T H_x^{s-1}} +  G(\|u\|_{L^\infty_T H_x^{s}}+\|v\|_{L^\infty_T H_x^{s}})  \|u-v\|_{L^\infty_T H_x^{s-1}} \; .
\end{equation}
\end{lem}
\begin{proof}
According to the extension Lemma \ref{extensionlem}, it is clear that we only have to estimate the $ X^{s-1,1}_{\om,T} $-norm of $ u $ to prove \eqref{estXregular}.
  As noticed in Remark \ref{rem2}, $ u $ satisfies the Duhamel formula of
\eqref{eq1} and $\|u_0\|_{H^\theta_\om} \le \|u\|_{L^\infty_T H^\theta_\om} $ for any $ \theta\le s $. Hence, standard linear estimates in Bourgain's spaces lead to
  \begin{eqnarray*}
  \|u\|_{X^{s-1,1}_{\om,T}} & \lesssim &  \|u_0\|_{H^{s-1}_\om}+\| \partial_x(f(u)) \|_{X^{s-1,0}_{\om,T}} \\
  & \lesssim &  \|u_0\|_{H^{s-1}_\om}+\|f(u)-f(0) \|_{L^2_T H^s_\om} \\
 &  \lesssim   &  \|u\|_{L^\infty_T H^{s-1}_\om}+ G(\|u\|_{L^\infty_{T,x}}) \| u \|_{L^\infty_T H^s_\om}  \; ,
 \end{eqnarray*}
by using \eqref{eq2.2ana}.\\
In the same way, using this times   \eqref{eq2.3ana}, we get
  \begin{eqnarray*}
  \|u-v\|_{X^{s-2,1}_T}
  & \lesssim &  \|u_0-v_0\|_{H_x^{s-1}}+\|f(u)-f(v)\|_{L^2_T H_x^{s-1}} \\
 &  \lesssim   &  \|u-v\|_{L^\infty_T H_x^{s-1}}+
 G(\|u\|_{L^\infty_T H_x^{s}}+\|v\|_{L^\infty_T H_x^{s}})  \|u-v\|_{L^\infty_T H_x^{s-1}},
 \end{eqnarray*}
 which completes the proof.
\end{proof}

The following proposition is one of main estimates in the present paper.

\begin{prop}[A priori estimate]\label{prop_apri}
  Let $\de\ge 1$, and suppose that the dyadic sequence $\{\om_N\}$ of positive numbers satisfies $\om_N\le \om_{2N}\le \de\om_N$ for $N\ge1$.
  Let $0<T<1$, $\al\in[1,2]$ and $2\ge s\ge s(\al):=1/2+2\be(\al)$.
  Let $u_0\in H^s(\T)$ and let $u\in L^\infty(0,T;H^s(\T)) $ be a solution to \eqref{eq1}--\eqref{initial} on $[0,T]$.  Then there exists an entire function $ G=G[f] $ that is increasing and non negative on $ \R_+ $  such that
  \EQS{\label{P}
    \|u\|_{L_T^\I H_\om^s}^2
    \le \|u_0\|_{H_\om^s}^2 + T^{1/4} G(\|u\|_{Z_{T}^{s(\al)}})
    \|u\|_{Z_{\om,T}^{s}}
    \|u\|_{L_T^\I H_\om^s}.
  }
\end{prop}

\begin{proof} First we notice that according to Lemma \ref{lem1} it holds
$ u\in Z_{\om,T}^{s}$.
  By using \eqref{eq1}, we have
  \EQQS{
    \frac{d}{dt}\|P_N u(t,\cdot)\|_{L_x^2}^2
    =-2\int_\T P_N\p_x(f(u))P_N udx.
  }
 Fixing $ t\in ]0,T[ $,  integration in time between $0$ and $t$, multiplication by $ \om_N^2 (1\vee N)^{2s} $ and summation over $N$ yield
  \EQS{\label{PP}
    \begin{aligned}
      \|u(t)\|_{H_\om^s}^2
      &=\sum_{N}\om_N^2 (1\vee N)^{2s}
       \bigg\{\|P_N u_0\|_{L_x^2}^2
       -2\int_0^t\int_\T P_N\p_x(f(u))  P_N u dxdt'\bigg\}\\
      &\le \|u_0\|_{H_\om^s}^2
       +2\sum_{N} \om_N^2 (1\vee N)^{2s}
       \bigg|\int_0^t\int_\T P_N\p_x (f(u)-f(0)) P_N u dxdt'\bigg|\\
      &\le \|u_0\|_{H_\om^s}^2
        +2\sum_{N\ge 1}  \om_N^2 N^{2s} \bigg|\int_0^t\int_\T (f(u)-f(0)) P_N^2\p_x u dxdt'\bigg|
    \end{aligned}
  }
 since $P_0\p_x u =0$ and $\p_x (f(u))=\p_x (f(u)-f(0))$.
 Now we rewrite $f(u)-f(0) $ as $ \sum_{k\ge 1} \frac{f^{(k)}(0)}{k!} u^k $
   and we notice that for any fixed $ N\in 2^{\N}  $,
    \EQS{\label{eq4.4F}
      \int_0^t\int_\T (f(u)-f(0)) P_N^2\p_x u\,  dxdt'\
   =  \sum_{k\ge 1} \frac{f^{(k)}(0)}{k!}\int_0^t\int_\T u^k P_N^2\p_x u \, dxdt' \; .
   }
   Indeed
   \EQQS{
    \begin{aligned}
       \sum_{k\ge 1} \frac{|f^{(k)}(0)|}{k!}\int_0^t\int_\T |u^k P_N^2\p_x u| \, dxdt'
       & \lesssim N  \sum_{k\ge 1} \frac{|f^{(k)}(0)|}{k!}\int_0^t \|u^k\|_{L^2_x} \|u\|_{L^2_x} dt' \\
       & \lesssim N  \sum_{k\ge 1} \frac{|f^{(k)}(0)|}{k!}\int_0^t \|u\|_{L^\I_x}^{k-1} \|u\|_{L^2_x}^2 dt' \\
     &  \lesssim N T G(\|u\|_{L_{T,x}^\I}) \|u\|_{L_T^\I L_x^2}^2<\infty ,
    \end{aligned}
  }
  that proves \eqref{eq4.4F} by Fubini-Lebesgue's theorem. \eqref{eq4.4F}
 together with Fubini-Tonelli's theorem then ensure that
  \EQS{\label{eq4.4FF}
    \begin{aligned}
      &\sum_{N\ge 1}\om_N^2 N^{2s}
       \bigg|\int_0^t\int_\T (f(u)-f(0)) P_N^2\p_x u dxdt'\bigg|\\
      &=\sum_{N\ge 1}\om_N^2 N^{2s}
       \bigg|\sum_{k\ge 1}\frac{f^{(k)}(0)}{k!}\int_0^t\int_\T u^k P_N^2\p_x u dxdt'\bigg|\\
      &\le \sum_{N\ge 1}\sum_{k\ge 1}\om_N^2 N^{2s}
       \frac{|f^{(k)}(0)|}{k!}
       \bigg|\int_0^t\int_\T u^k P_N^2\p_x u dxdt'\bigg|
      =\sum_{k\ge 1}  \frac{|f^{(k)}(0)|}{k!} I_{k}^t,
    \end{aligned}
  }
  where
  \EQQS{
    I^t_{k} := \sum_{N\ge 1} \om_N^2 N^{2s} \bigg|\int_0^t\int_\T u^k P_N^2\p_x u \, dxdt'\bigg|.
  }
  By integration by parts it is easy to check that $ I_1^t =0 $.
  We set
  \EQQS{
    C_0:=\|u\|_{Z_T^{s(\al)}}.
  }
Let us now  prove that for any $ k\ge 1$ it holds
  \EQS{\label{eq4.4FFF}
  I_{k+1}^t \le C^k T^{1/4} G(C_0) C_0^k
    (\|u\|_{X_{\om,T}^{s-1,1}}+\|u\|_{L_T^\I H_\om^s})
    \|u\|_{L_T^\I H_\om^s}.
   }
  that  clearly  leads \eqref{P}  by taking  \eqref{PP} and  \eqref{eq4.4FF} into account
  since  $\sum_{k\ge 1}  \frac{|f^{(k+1)}(0)|}{(k+1)!} C^k C_0^k <\infty$.

 In the sequel we  fix $ k\ge 1$.
 For simplicity, for any positive numbers $a$ and $b$, the notation $a\lesssim_k b$ means there exists a positive constant $C>0$ independent of $k$ such that
 \EQS{\label{eq_lesssim}
   a\le C^k b.
 }
  Remark that $a\le k^m b$ for $m\in\N$ can be expressed by $a\lesssim_k b$ too since an elementary calculation shows $k^m\le m! e^k$ for $m\in\N$.
  Here, $e$ is Napier's constant.
  The contribution of the sum over  $N\lesssim 1 $ in  $ I^t_{k+1} $  is easily estimated by
  \EQS{\label{eq4.4}
    \begin{aligned}
     &  \sum_{N\lesssim 1}\om_N^2 N^{2s}
      \bigg|\int_0^t\int_\T  u^{k+1}  P_N^2\p_x u dxdt'\bigg|\\
     &\le  T\sum_{N\lesssim 1}\|u\|_{L_{T,x}^\I}^k \|u\|_{L_T^\I L_x^2}\|P_N^2 u\|_{L_T^\I L_x^2}
     \lesssim_k    T C_0^k \|u\|_{L_T^\I H^s_\om}^2.
    \end{aligned}
  }
It thus remains to bound the contribution of the sum over $N \gg 1  $ in $ I^t_{k+1} $.
  Putting
  \EQQS{
    A(\xi_1,\dots,\xi_{k+2})
    &:=\sum_{j=1}^{k+2}\phi_N^2(\xi_j)\xi_j,\\
    A_1(\xi_1,\xi_2)
    &:=\phi_N^2(\xi_1)\xi_1+\phi_N^2(\xi_2)\xi_2,\\
    A_2(\xi_4,\dots,\xi_{k+2})
    &:=\sum_{j=4}^{k+2}\phi_N^2(\xi_j)\xi_j,
  }
  we see from the symmetry that
  \EQS{\label{eq_4.8}
  \begin{aligned}
    \int_\T u^{k+1}P_N^2\p_x u dx
    &=\frac{i}{k+2}\sum_{\xi_1+\cdots+\xi_{k+2}=0}A(\xi_1,\dots,\xi_{k+2})\prod_{j=1}^{k+2}\ha{u}(\xi_j)\\
    &=\frac{i}{k+2}\sum_{N_1,\dots,N_{k+2}}\sum_{\xi_1+\cdots+\xi_{k+2}=0}A(\xi_1,\dots,\xi_{k+2})\prod_{j=1}^{k+2}\phi_{N_j}(\xi_j)\ha{u}(\xi_j).
  \end{aligned}
  }
  By symmetry  we can assume  that $ N_1\ge  N_2 \ge N_3 $ if $ k=1$,  $N_1\ge N_2 \ge N_3\ge N_4  $ if $ k=2 $ and $N_1\ge N_2\ge N_3\ge N_4\ge N_5=\max_{j\ge5}N_j$ if  $ k\ge 3$. We notice that
   the cost of this choice is a constant factor  less than $(k+2)^4 $.
   It is also worth noticing that the frequency projection operator $ P_N
$ ensures that the contribution of any $ N_1\le  N/4 $ does cancel. We thus can assume that $ N_1 \ge N/4 $ and that $ N_2\gtrsim N_1/k $ with $ N_2\ge 1$.

  First, we consider the contribution of $A_2$. Note that we must have $k\ge 2 $ since otherwise $ A_2=0 $.
  Also  it suffices to consider the contribution of $(\phi_N(\xi_4))^2\xi_4$ since the contributions of
   $(\phi_N(\xi_j))^2\xi_j$  for $ j\ge 5 $ are clearly  simplest.
  Note that $N_4 \sim  N $ in this case.
  By the Bernstein inequality, we have
  \EQS{\label{eq4.9}
    \sum_{K}\|P_{K}u\|_{L_{T,x}^\I}
    \lesssim\sum_{K}(1\vee K^{1/2-s(\al)})\|u\|_{L_{T}^\I H_{x}^{s(\al)}}
    \lesssim\|u\|_{L_{T}^\I H_{x}^{s(\al)}}\lesssim C_0.
  }
  This together with H\"older's and Young's inequalities and \eqref{eq_stri1} gives
  \EQQS{
    &\sum_{N\gg 1}\sum_{N_1,\dots,N_{k+2}} \om_N^2 N^{2s}\bigg|\int_0^t\int_\T (\p_x P_N^2 P_{N_4}u) \prod_{j=1,j\neq4}^{k+2}P_{N_j}u  dxdt'\bigg|\\
    &\lesssim_k\|u\|_{L_{T}^\I H_{x}^{s(\al)}}^{k-2}
      \sum_{N_1\ge N_2\ge N_3\ge  N_4\gg 1}
      \om_{N_4}^2 N_4^{2s+1}\prod_{j=1}^4\|P_{N_j}u\|_{L_{T,x}^4}\\
    &\lesssim_k C_0^{k-2}
      \sum_{N_1\ge N_2\ge N_3\ge  N_4\gg 1 }
      \prod_{j=1}^2\bigg(\frac{N_4}{N_j}\bigg)^{s-\be(\al)}
       \om_{N_j}\|D_x^{s-\be(\al)} P_{N_j}u\|_{L_{T,x}^4}\\
    &\quad\quad\times\prod_{j=3}^4
      \bigg(\frac{N_4}{N_j}\bigg)^{1/2+\be(\al)}
      \|D_x^{1/2+\be(\al)} P_{N_j}u\|_{L_{T,x}^4}\\
    &\lesssim_k C_0^{k-2}
      \bigg(\sum_{N_4}\|D_x^{1/2+\be(\al)}P_{N_4} u\|_{L_{T,x}^4}^4\bigg)^{1/4}\\
    &\quad\quad\times\bigg(\sum_{N_4}\bigg(\sum_{N_3 \gtrsim N_4}
     \bigg(\frac{N_4}{N_3}\bigg)^{1/2+\be(\al)}\|D_x^{1/2+\be(\al)}P_{N_3} u\|_{L_{T,x}^4}\bigg)^4\bigg)^{1/4}\\
    &\quad\quad\times  \bigg(\sum_{N_4}\bigg(\sum_{K\gtrsim N_4}
         \bigg(\frac{N_4}{K}\bigg)^{2(s-\be(\al))}\om_{K}^2
         \|D_x^{s-\be(\al)}P_{K} u\|_{L_{T,x}^4}^2\bigg)^2\bigg)^{1/2}\\
    &\lesssim_k T^{1/2} C_0^{k} G(C_0) \|u\|_{L_{T}^\I H_\om^s}^2.
  }
  Next, we consider the contribution $A_1$. We notice that the frequency projector in $ A_1 $ ensures that either $ N_1\sim N$ or $ N_2\sim N $ and
thus in any case $ N \gtrsim N_3$.  Moreover we can also assume that $ N_3\ge 1 $ since otherwise the contribution of $ A_1$ cancelled by integration by parts.

 We divide the contribution $A_1$ into three cases:  1. $N_2\lesssim    N_3\lesssim k N_4$, 2. $ N_3\gg k N_4$ or $ k=1$  and 3. $N_2\gg    N_3 $.
  Set
  \EQQS{
    J_t
    :=\sum_{N\gg 1}\sum_{N_1,\dots,N_{k+2}}\om_N^2 N^{2s}
      \bigg|\int_0^t\int_\T\Pi(P_{N_1}u,P_{N_2}u)\prod_{j=3}^{k+2}P_{N_j}udxdt'\bigg|,
  }
  where $\Pi(f,g)$ is defined by \eqref{def_pi}.
  Note that $ N\gg 1 $ ensures that  $N_1\gg1$.\\
  \noindent
  \textbf{Case 1: $N_2\lesssim  N_3\lesssim k N_4$.}
  Since $N \lesssim N_1\lesssim k  N_2 \lesssim  k N_3 \lesssim k^2 N_4  $, H\"older's, Bernstein's  and Young's inequalities and \eqref{eq_stri1}
show that
  \EQQS{
    J_t
    &\lesssim k^{2(2s+1)} \sum_{N_1,\dots,N_{k+2}\atop N_1\lesssim k^2 N_4, N_1\ge N_4, N_2\ge N_4,N_3\ge N_4} \om_{N_1}^2 N_4^{2s+1}
    \prod_{j=1}^4\|P_{N_j}u\|_{L_{T,x}^4}
    \prod_{j=5}^{k+2}\|P_{N_j}u\|_{L_{T,x}^\I}\\
    &\lesssim_k C_0^{k-2}
    \sum_{N_1\ge N_4,N_2\ge N_4,N_3\ge N_4}
    \frac{\om_{N_1}}{\om_{N_2}} \om_{N_1} \om_{N_2}  \Bigl(\frac{N_4}{N_1}\Bigr)^{s-\be(\al)} \Bigl(\frac{N_4}{N_2}\Bigr)^{s-\be(\al)} \Bigl(\frac{N_4}{N_3}\Bigr)^{\be(\al)+1/2}\\
    &\quad \times \prod_{j=1}^2
      \|D_x^{s-\be(\al)} P_{N_j}u\|_{L_{T,x}^4}
     \prod_{n=3}^4
     \|D_x^{\be(\al)+1/2} P_{N_n}u\|_{L_{T,x}^4}\\
    &\lesssim_k C_0^{k-2} \Bigl( \sum_{K}\om_K^4\|D_x^{s-\be(\al)}P_K u\|_{L_{T,x}^4}^4\Bigr)^{1/2}
   \Bigl( \sum_{K}\|D_x^{\be(\al)+1/2}P_K u\|_{L_{T,x}^4}^4\Bigr)^{1/2}\\
    &\lesssim_k  T^{1/2} C_0^{k} G(C_0) \|u\|_{L_T^\I H_\om^{s}}^2,
  }
where we used that $ s\le 2  $ so that $ k^{2(2s+1)} \le k^{10} $ and that since $ N_1\lesssim k N_2 $ we have  $  \frac{\om_{N_1}}{\om_{N_2}} \lesssim k $ since $\de\le 2$.\\
  \noindent
  \textbf{Case 2: $ N_3\gg k  N_4$ or $ k=1$.} By impossible frequency interections, we must have $ N\sim N_1\sim N_2$.
  We take the extensions $\check{u}=\rho_T(u)$ of $u$ defined in \eqref{def_ext}.
  For simplicity, with a slight abuse of notation, we define the following functional:
  \EQS{\label{def_J1}
    J_\I^{(2)}(u_1,\cdots,u_{k+2})
    :=\sum_{N\gg 1}\sum_{N_1,\dots,N_{k+2}}\om_N^2 N^{2s}
      \bigg|\int_\R\int_\T\Pi(u_1,u_2)
      \prod_{j=3}^{k+2}u_jdxdt'\bigg|.
  }
  Setting $R=N_1^{1/3}N_3^{4/3}$, we split $J_t$ as
  \EQQS{
    J_t
    &\le J_\I^{(2)}(P_{N_1}\1_{t,R}^{\textrm{high}}\check{u},
      P_{N_2}\1_t\check{u},P_{N_3}\check{u},\cdots, P_{N_{k+2}}\check{u})\\
    &\quad+J_\I^{(2)}(P_{N_1}\1_{t,R}^{\textrm{low}}\check{u},
      P_{N_2}\1_{t,R}^{\textrm{high}}\check{u},
      P_{N_3}\check{u},\cdots,
      P_{N_{k+2}}\check{u})\\
    &\quad+J_\I^{(2)}(P_{N_1}\1_{t,R}^{\textrm{low}}\check{u},
      P_{N_2}\1_{t,R}^{\textrm{low}}\check{u},
      P_{N_3}\check{u},\cdots,
      P_{N_{k+2}}\check{u})
    =:J_{\I,1}^{(2)}+J_{\I,2}^{(2)}+J_{\I,3}^{(2)}.
  }
  % \EQQS{
  %   J_t
  %   &\le \sum_{N\gg 1}\sum_{N_1,\dots,N_{k+2}}\om_N^2 N^{2s}
  %     \bigg|\int_\R\int_\T
  %      \Pi(P_{N_1}\1_{t,R}^{\textrm{high}}\check{u},P_{N_2}\1_t\check{u})
  %      \prod_{j=3}^{k+2}P_{N_j}\check{u}dxdt'\bigg|\\
  %   &\quad+\sum_{N\gg 1}\sum_{N_1,\dots,N_{k+2}}\om_N^2 N^{2s}
  %     \bigg|\int_\R\int_\T
  %      \Pi(P_{N_1}\1_{t,R}^{\textrm{low}}\check{u},
  %         P_{N_2}\1_{t,R}^{\textrm{high}}\check{u})
  %      \prod_{j=3}^{k+2}P_{N_j}\check{u}dxdt'\bigg|\\
  %    &\quad+\sum_{N\gg 1}\sum_{N_1,\dots,N_{k+2}}\om_N^2 N^{2s}
  %      \bigg|\int_\R\int_\T
  %       \Pi(P_{N_1}\1_{t,R}^{\textrm{low}}\check{u},
  %         P_{N_2}\1_{t,R}^{\textrm{low}}\check{u})
  %       \prod_{j=3}^{k+2}P_{N_j}\check{u}dxdt'\bigg|
  %   =:\sum_{n=1}^3 J_{\I,n}^{(2)}.
  % }
  For $J_{\I,1}^{(2)}$, we see from  \eqref{eq4.1} that $\|\1_{t,R}^{{\rm
high}}\|_{L^1}\lesssim T^{1/4}N_1^{-1/4}N_3^{-1}$, which gives
  \EQQS{
    J_{\I,1}^{(2)}
    &\lesssim\sum_{N_1,\dots,N_{k+2}} \om_{N_1}^2
      N_1^{2s}N_3\|\1_{t,R}^{\textrm{high}}\|_{L_t^1}
      \|P_{N_1}\check{u}\|_{L_t^\I L_x^2}\|P_{N_2}\check{u}\|_{L_t^\I L_x^2}
      \prod_{j=3}^{k+2}\|P_{N_j}\check{u}\|_{L_{t,x}^\I}\\
    &\lesssim_k T^{1/4}\|\check{u}\|_{L_t^\I H_x^{s(\al)}}^k\|\check{u}\|_{L_t^\I H_\om^{s}}^2\sum_{N_1} N_1^{-1/4}
    \lesssim_k T^{1/4}C_0^k\|u\|_{L_T^\I H_\om^{s}}^2
  }
  since $N\sim N_1 \sim N_2$.
  In the last inequality, we used \eqref{eq2.1single}.
  By \eqref{eq4.2}, $J_{\I,2}^{(2)}$ can be estimated by the same bound as above.
  For $J_{\I,3}^{(2)}$, we see from Lemma \ref{lem_res1} that $|\Om_{k+1}|\gtrsim N_3N_1^\al
   \gg R$.
  Then, defining $L:=N_3 N_1^\al$, we decompose $J_{\I,3}^{(2)}$ as
  \EQQS{
    J_{\I,3}^{(2)}
    &\le J_\I^{(2)}(P_{N_1}Q_{\gtrsim L}(\1_{t,R}^{\textrm{low}}\check{u}),
      P_{N_2}\1_{t,R}^{\textrm{low}}\check{u},
      P_{N_3}\check{u},\cdots,
      P_{N_{k+2}}\check{u})\\
    &\quad +J_\I^{(2)}(P_{N_1}Q_{\ll L}(\1_{t,R}^{\textrm{low}}\check{u}),
      P_{N_2}Q_{\gtrsim L}(\1_{t,R}^{\textrm{low}}\check{u}),
      P_{N_3}\check{u},\cdots,
      P_{N_{k+2}}\check{u})\\
    &\quad +J_\I^{(2)}(P_{N_1}Q_{\ll L}(\1_{t,R}^{\textrm{low}}\check{u}),
      P_{N_2}Q_{\ll L}(\1_{t,R}^{\textrm{low}}\check{u}),
      P_{N_3}Q_{\gtrsim L}\check{u},\cdots,
      P_{N_{k+2}}\check{u})\\
    &\quad+\cdots
      +J_\I^{(2)}(P_{N_1}Q_{\ll L}(\1_{t,R}^{\textrm{low}}\check{u}),
        P_{N_2}Q_{\ll L}(\1_{t,R}^{\textrm{low}}\check{u}),
        P_{N_3}Q_{\ll L}\check{u},\cdots,
        P_{N_{k+2}}Q_{\gtrsim L}\check{u})\\
    &=:J_{\I,3,1}^{(2)}+\cdots+J_{\I,3,k+2}^{(2)}.
  }
  % \EQQS{
  %   J_{\I,3}^{(2)}
  %   &\le\sum_{N\gg 1}\sum_{N_1,\dots,N_{k+2}}\om_{N}^2 N^{2s}
  %     \bigg|\int_\R\int_\T
  %      \Pi(P_{N_1}Q_{\gtrsim N_3N_1^\al}(\1_{t,R}^{\textrm{low}}\check{u}),P_{N_2}\1_{t,R}^{\textrm{low}}\check{u})
  %      \prod_{j=3}^{k+2}P_{N_j}\check{u}dxdt'\bigg|\\
  %   &\quad+\sum_{N\gg 1}\sum_{N_1,\dots,N_{k+2}}\om_{N}^2 N^{2s}
  %     \bigg|\int_\R\int_\T
  %      \Pi(P_{N_1}Q_{\ll N_3N_1^\al}(\1_{t,R}^{\textrm{low}}\check{u}),P_{N_2}Q_{\gtrsim N_3N_1^\al}(\1_{t,R}^{\textrm{low}}\check{u}))
  %      \prod_{j=3}^{k+2}P_{N_j}\check{u}dxdt'\bigg|\\
  %   &\quad+\sum_{n=3}^{k+2}\sum_{N\gg 1}\sum_{N_1,\dots,N_{k+2}}\om_{N}^2 N^{2s}
  %    \bigg|\int_\R\int_\T
  %     \Pi(P_{N_1}Q_{\ll N_3N_1^\al}(\1_{t,R}^{\textrm{low}}\check{u}),P_{N_2}Q_{\ll N_3N_1^\al}(\1_{t,R}^{\textrm{low}}\check{u}))\\
  %   &\quad\quad\times\bigg(\prod_{j=3}^{n-1}P_{N_j}
  %       Q_{\ll N_3N_1^\al}\check{u}\bigg)P_{N_n}Q_{\gtrsim N_3N_1^\al}\check{u}
  %       \prod_{j=n+1}^{k+2}P_{N_j}\check{u} dxdt'\bigg|
  %   =:\sum_{j=1}^{k+2}J_{\I,3,j}^{(2)}.
  % }
  %It is worth noting that $R\ll N_3N_1^\al$ since $N_1\gg 1$.
  We also see from \eqref{eq4.1} that
  \EQS{\label{eq4.7}
    \begin{aligned}
      \|P_{N_2}\1_{t,R}^{\textrm{low}}\check{u}\|_{L_{t,x}^2}
      &\le\|P_{N_2}\1_{t}\check{u}\|_{L_{t,x}^2}
        +\|P_{N_2}\1_{t,R}^{\textrm{high}}\check{u}\|_{L_{t,x}^2}\\
      &\lesssim \|P_{N_2}\1_{t}\check{u}\|_{L_{t,x}^2}
        +T^{1/4}R^{-1/4}\|P_{N_2}\check{u}\|_{L_{t}^\I L_x^2}.
    \end{aligned}
  }
  For $J_{\I,3,1}^{(2)}$, Lemmas \ref{lem_comm1} and \ref{extensionlem}, the H\"older inequality, \eqref{eq4.6} and \eqref{eq4.7} imply that
  \EQQS{
    J_{\I,3,1}^{(2)}
    &\lesssim \sum_{N_1,\dots,N_{k+2}}\om_{N_1}^2 N_1^{2s}N_3
      \|P_{N_1}Q_{\gtrsim L} (\1_{t,R}^{\textrm{low}}\check{u})\|_{L_{t,x}^2}
      \|P_{N_2}\1_{t,R}^{\textrm{low}}\check{u}\|_{L_{t,x}^2}
      \prod_{j=3}^{k+2}\|P_{N_j}\check{u}\|_{L_{t,x}^\I}\\
    &\lesssim_k \|\check{u}\|_{L_t^\I H_x^{s(\al)}}^k
      \sum_{N_1\gtrsim1}\om_{N_1}^2 N_1^{2s-1}
      \|P_{N_1} \check{u}\|_{X^{0,1}}
        \|P_{N_1} \1_t \check{u}\|_{L_{t,x}^2}\\
    &\quad+T^{1/4}\|\check{u}\|_{L_t^\I H_x^{s(\al)}}^{k-1}
      \sum_{N_1\gtrsim N_3}\om_{N_1}^2 N_1^{2s-13/12}N_3^{-1/3}\|P_{N_1} \check{u}\|_{X^{0,1}}
        \|P_{N_1} \check{u}\|_{L_{t}^\I L_x^2}
        \|P_{N_3}\check{u}\|_{L_{t,x}^\I}\\
    &\lesssim T^{1/4}C_0^{k}
      \|\check{u}\|_{L_t^\I H_\om^s}\|\check{u}\|_{X_\om^{s-1,1}}
      \lesssim_k T^{1/4} C_0^{k}
        \|u\|_{L_T^\I H_\om^s}\|u\|_{Z_{\om,T}^{s}}.
  }
  Here, we used $N_1^{-\al}\le N_1^{-1}$ since $\al\in[1,2]$.
  We can evaluate the contribution $J_{\I,3,2}^{(2)}$ by the same way with \eqref{eq4.7}.
  Next, we consider the contribution $J_{\I,3,3}^{(2)}$.
  Lemmas \ref{lem_comm1} and \ref{extensionlem}, the H\"older inequality and \eqref{eq4.3} show
  \EQQS{
    J_{\I,3,3}^{(2)}
    &\lesssim\sum_{N_1,\dots,N_{k+2}}\om_{N_1}^2  N_1^{2s}N_3
      \|P_{N_1}Q_{\ll L}
      (\1_{t,R}^{\textrm{low}} \check{u})\|_{L_{t,x}^2}
      \|P_{N_2}Q_{\ll L}(\1_{t,R}^{\textrm{low}}\check{u})\|_{L_{t}^\I L_x^2}\\
    &\quad\times\|P_{N_3}Q_{\gtrsim L}
      \check{u}\|_{L_{t}^2 L_x^\I}
      \prod_{j=4}^{k+2}\|P_{N_j}\check{u}\|_{L_{t,x}^\I}\\
    &\lesssim_k T^{1/2} \|\check{u}\|_{L_{t}^\I H_x^{s(\al)}}^{k-1}
      \sum_{N_1\gtrsim N_3\ge 1} \om_{N_1}^2
      N_1^{2s-\al}\|P_{N_1}\check{u}\|_{L_{t}^\I L_x^2}^2
      \|D_x^{1/2}P_{N_3}\check{u}\|_{X^{0,1}}\\
    &\lesssim T^{1/2}C_0^{k-1}
      \sum_{N_1\gtrsim N_3\ge 1}
      N_1^{-\be(\al)}N_3^{-\be(\al)}
      \om_{N_1}^2 N_1^{2s}\|P_{N_1}\check{u}\|_{L_{t}^\I L_x^2}^2
      \|P_{N_3}\check{u}\|_{X^{s(\al)-1,1}}\\
    &\lesssim T^{1/2} C_0^{k-1}
      \|\check{u}\|_{X^{s(\al)-1,1}}\|\check{u}\|_{L_t^\I H_\om^s}^2
    \lesssim_k T^{1/2} C_0^{k}\|u\|_{L_T^\I H_\om^s}^2
  }
  since $s(\al)-1<0$.
  In a similar manner, we can evaluate the contribution $J_{\I,3,j}^{(2)}$ for $j=4,\dots,k+2$ by the same bound.

  \noindent
  \textbf{Case 3: $N_2\gg  N_3$.}
  In this case, roughly speaking, we compare  $|\xi_3+\xi_4| $ and $ k |\xi_5|$ where $ \xi_i $ is the $ i$-th largest frequency. If $|\xi_3+\xi_4| \gg k |\xi_5| $, we have a suitable non resonance relation (Lemma \ref{lem_res2}) whereas otherwise we can share the lost derivative between  three functions.
  We split $J_t$ as
  \EQQS{
    J_t
    &\le\sum_{N\gg 1}\sum_{N_1,\dots,N_{k+2}}\om_{N}^2 N^{2s}
      \bigg|\int_0^t\int_\T\Pi(P_{N_1}u,P_{N_2}u)P_{\lesssim k N_5 }\bigg(\prod_{j=3}^{k+2}P_{N_j}u\bigg)dxdt'\bigg|\\
    &\quad+\sum_{N\gg 1}\sum_{N_1,\dots,N_{k+2}}
      \sum_{ k N_5\ll M\lesssim N_3}\om_{N}^2 N^{2s}
      \bigg|\int_0^t\int_\T\Pi(P_{N_1}u,P_{N_2}u)P_{M}\bigg(\prod_{j=3}^{k+2}P_{N_j}u\bigg)dxdt'\bigg|\\
    &=:I_{t}^{(3)}+J_{t}^{(3)}.
  }
  Remark that if $k=2$, then the term $I_{t}^{(3)}$ does not appear, and it holds that $J_t\le J_{t}^{(3)}$ with the summation $\sum_{M\lesssim N_3}$ instead of $\sum_{k N_5\ll M\lesssim  N_3}$.
   Note also that the contribution in respectively $ I_{t}^{(3)}$ and $J_{t}^{(3)}$  of respectively $ N_5=0 $ and $ N_5=M=0$ does vanish by integration by parts. Therefore we can always assume that $ N_5\ge 1$ in $I_{t}^{(3)}$ and that
   $ M\ge 1 $ in $J_{t}^{(3)}$.
  For $I_{t}^{(3)}$, since either $N_1\sim N$ or $N_2\sim N$, we see from
the Young inequality and the assumption that $\de\le 2$ that
  \EQQS{
    &\sum_{N_1\ge N_2,kN_2\gtrsim N_1}
    (\om_{N_1}^2N_1^{2s}+\om_{N_2}^2N_2^{2s})\|P_{N_1}u\|_{L_x^2}
    \|P_{N_2}u\|_{L_x^2}\\
    &\lesssim k^s \sum_{kN_2\gtrsim N_1}
     \frac{\om_{N_1}}{\om_{N_2}}\om_{N_1}\om_{N_2}\bigg(\frac{N_1}{kN_2}\bigg)^s \|D_x^s P_{N_1}u \|_{L_x^2}
     \|D_x^s P_{N_2}u \|_{L_x^2}\\
    &\quad+\sum_{N_2\le N_1}
      \frac{\om_{N_2}}{\om_{N_1}}\om_{N_1}\om_{N_2}\bigg(\frac{N_2}{N_1}\bigg)^s \|D_x^s P_{N_1}u \|_{L_x^2}
      \|D_x^s P_{N_2}u \|_{L_x^2}
    \lesssim k^3 \|u\|_{H_\om^s}^2.
  }
   This, Lemma \ref{lem_comm1}, H\"older's and Young's inequalities and \eqref{eq_stri2} show
  \EQQS{
    I_{t}^{(3)}
    &\lesssim \sum_{N_1\ge N_2,kN_2\gtrsim N_1} \sum_{N_3,\dots,N_{k+2}} (\om_{N_1}^2N_1^{2s}+\om_{N_2}^2N_2^{2s})  k N_5\\
    &\quad\times\int_0^t \|P_{N_1}u\|_{L_x^2}
      \|P_{N_2}u\|_{L_x^2}
      \prod_{j=3}^{k+2}\|P_{N_j}u\|_{L_x^\I}dt'\\
    &\lesssim_k \|u\|_{L_T^\I H_x^{s(\al)}}^{k-3}
      \|u\|_{L_T^\I H_\om^{s}}^2\sum_{N_3\ge  N_4\ge N_5}
       \prod_{j=3}^{5}\bigg(\frac{N_5}{N_j}\bigg)^{1/3}
       \|D_x^{1/3} P_{N_j}u\|_{L_{T}^3 L_x^\I}\\
    &\lesssim_k C_0^{k-3}
      \|u\|_{L_T^\I H_\om^{s}}^2
      \sum_{K}\|D_x^{1/3}P_K u\|_{L_T^3 L_x^\I}^3
    \lesssim_k  T^{5/8} C_0^{k} G(C_0) \|u\|_{L_T^\I H_\om^{s}}^{2}
  }
  since $s(\al)=1/2+2\be(\al)\ge 7/12+\be(\al)$ for $\al\in[1,2]$.
  For $J_{t}^{(3)}$, we take the extensions $\check{u}=\rho_T(u)$ of $u$ defined in \eqref{def_ext}.
  Note that we have $ N_1\sim N_2 \sim N $.
  We further decompose $J_{t}^{(3)}$ as in Case 2.
  In the same sprit as \eqref{def_J1}, with a slight abuse of notation, we define the functional for the sake of notation:
  \EQQS{
    &J_\I^{(3)}(u_1,\cdots,u_{k+2})\\
    &:=\sum_{N\gg 1}\sum_{N_1,\dots,N_{k+2}}
      \sum_{k N_5\ll M\lesssim N_3}
      \om_N^2 N^{2s}
      \bigg|\int_\R\int_\T\Pi(u_1,u_2)
      P_{M}\bigg(\prod_{j=3}^{k+2}u_j\bigg)dxdt'\bigg|.
  }
  Putting $R=N_1^{1/3}M^{4/3}$, we obtain
  \EQQS{
    J_t^{(3)}
    &\le J_\I^{(3)}(P_{N_1}\1_{t,R}^{\textrm{high}}\check{u},
      P_{N_2}\1_t\check{u},P_{N_3}\check{u},\cdots, P_{N_{k+2}}\check{u})\\
    &\quad+J_\I^{(3)}(P_{N_1}\1_{t,R}^{\textrm{low}}\check{u},
      P_{N_2}\1_{t,R}^{\textrm{high}}\check{u},
      P_{N_3}\check{u},\cdots,
      P_{N_{k+2}}\check{u})\\
    &\quad+J_\I^{(3)}(P_{N_1}\1_{t,R}^{\textrm{low}}\check{u},
      P_{N_2}\1_{t,R}^{\textrm{low}}\check{u},
      P_{N_3}\check{u},\cdots,
      P_{N_{k+2}}\check{u})
    =:J_{\I,1}^{(3)}+J_{\I,2}^{(3)}+J_{\I,3}^{(3)}.
  }
  % \EQQS{
  %   J_{t}^{(3)}
  %   &\le\sum_{N\gg 1}\sum_{N_1,\dots,N_{k+2}}
  %     \sum_{k N_5\ll M\lesssim N_3}\om_{N}^2 N^{2s}
  %     \bigg|\int_\R \int_\T
  %     \Pi(P_{N_1}\1_{t,R}^{\textrm{high}} \check{u},P_{N_2}\1_t\check{u})
  %     P_{M}\bigg(\prod_{j=3}^{k+2}P_{N_j}\check{u}\bigg)dxdt'\bigg|\\
  %   &\quad+\sum_{N\gg 1}\sum_{N_1,\dots,N_{k+2}}
  %     \sum_{k N_5\ll M\lesssim N_3}\om_{N}^2 N^{2s}
  %     \bigg|\int_\R \int_\T
  %     \Pi(P_{N_1}\1_{t,R}^{\textrm{low}} \check{u},P_{N_2}\1_{t,R}^{\textrm{high}}\check{u})P_{M}
  %     \bigg(\prod_{j=3}^{k+2}P_{N_j}\check{u}\bigg)dxdt'\bigg|\\
  %   &\quad+\sum_{N\gg 1}\sum_{N_1,\dots,N_{k+2}}
  %     \sum_{k N_5\ll M\lesssim N_3}\om_{N}^2 N^{2s}
  %     \bigg|\int_\R \int_\T
  %     \Pi(P_{N_1}\1_{t,R}^{\textrm{low}} \check{u},P_{N_2}\1_{t,R}^{\textrm{low}}\check{u})P_{M}
  %     \bigg(\prod_{j=3}^{k+2}P_{N_j}\check{u}\bigg)dxdt'\bigg|\\
  %   &=: J_{\I,1}^{(3)}+J_{\I,2}^{(3)}+J_{\I,3}^{(3)}.
  % }
  Obviously, we have for any $\e>0$ (for instance, $\e=\be(\al)$)
  \EQS{\label{eq4.5}
    \sum_{M\lesssim N _3}1\lesssim N_3^{\e},
  }
  where the implicit constant does not depend on $N_3$.
  By \eqref{eq4.5}, \eqref{eq2.1single} and the same argument as that of $J_{\I,1}^{(2)}$, we obtain
  \EQQS{
    J_{\I,1}^{(3)}+J_{\I,2}^{(3)}
    \lesssim_k T^{1/4}\|\check{u}\|_{L_t^\I H_x^{s(\al)}}^k
      \|\check{u}\|_{L_t^\I H_\om^s}^2
    \lesssim_k T^{1/4} C_0^{k} \|u\|_{L_T^\I H_\om^s}^2.
  }
  For $J_{\I,3}^{(3)}$, Lemma \ref{lem_res2} implies that $|\Om_{k+1}|\gtrsim MN_1^\al\sim |\xi_3+\xi_4|N_1^\al$.
  So, defining $L:=MN_1^\al$, we have
  \EQQS{
    J_{\I,3}^{(3)}
    &\le J_\I^{(3)}(P_{N_1}Q_{\gtrsim L}(\1_{t,R}^{\textrm{low}}\check{u}),
      P_{N_2}\1_{t,R}^{\textrm{low}}\check{u},
      P_{N_3}\check{u},\cdots,
      P_{N_{k+2}}\check{u})\\
    &\quad +J_\I^{(3)}(P_{N_1}Q_{\ll L}(\1_{t,R}^{\textrm{low}}\check{u}),
      P_{N_2}Q_{\gtrsim L}(\1_{t,R}^{\textrm{low}}\check{u}),
      P_{N_3}\check{u},\cdots,
      P_{N_{k+2}}\check{u})\\
    &\quad +J_\I^{(3)}(P_{N_1}Q_{\ll L}(\1_{t,R}^{\textrm{low}}\check{u}),
      P_{N_2}Q_{\ll L}(\1_{t,R}^{\textrm{low}}\check{u}),
      P_{N_3}Q_{\gtrsim L}\check{u},\cdots,
      P_{N_{k+2}}\check{u})\\
    &\quad+\cdots
      +J_\I^{(3)}(P_{N_1}Q_{\ll L}(\1_{t,R}^{\textrm{low}}\check{u}),
        P_{N_2}Q_{\ll L}(\1_{t,R}^{\textrm{low}}\check{u}),
        P_{N_3}Q_{\ll L}\check{u},\cdots,
        P_{N_{k+2}}Q_{\gtrsim L}\check{u})\\
    &=:J_{\I,3,1}^{(3)}+\cdots+J_{\I,3,k+2}^{(3)}.
  }
  % \EQQS{
  %   J_{\I,3}^{(3)}
  %   &\le\sum
  %     \om_N^2 N^{2s}\bigg|\int_\R\int_\T
  %      \Pi(P_{N_1}Q_{\gtrsim MN_1^\al}(\1_{t,R}^{\textrm{low}}\check{u}),P_{N_2} (\1_{t,R}^{\textrm{low}}\check{u}))
  %      P_{M}\bigg(\prod_{j=3}^{k+2}P_{N_j}\check{u}\bigg)dxdt'\bigg|\\
  %   &\quad+\sum\om_N^2 N^{2s}
  %     \bigg|\int_\R\int_\T
  %      \Pi(P_{N_1}Q_{\ll MN_1^\al}(\1_{t,R}^{\textrm{low}}\check{u}),P_{N_2}Q_{\gtrsim MN_1^\al}(\1_{t,R}^{\textrm{low}}\check{u}))P_{M}\bigg(\prod_{j=3}^{k+2}P_{N_j}\check{u}\bigg)dxdt'\bigg|\\
  %   &\quad+\sum_{n=3}^{k+2}\sum\om_N^2 N^{2s}
  %    \bigg|\int_\R\int_\T
  %     \Pi(P_{N_1}Q_{\ll MN_1^\al}(\1_{t,R}^{\textrm{low}}\check{u}),P_{N_2}Q_{\ll MN_1^\al}(\1_{t,R}^{\textrm{low}}\check{u}))\\
  %   &\quad\quad\times P_{M}
  %       \bigg(\bigg(\prod_{j=3}^{n-1}P_{N_j}
  %       Q_{\ll MN_1^\al}\check{u}\bigg)P_{N_n}Q_{\gtrsim MN_1^\al}\check{u}
  %       \prod_{j=n+1}^{k+2}P_{N_j}\check{u}\bigg) dxdt'\bigg|
  %   =:\sum_{j=1}^{k+2}J_{\I,3,j}^{(3)},
  % }
%   where $\sum$ denotes $\sum_{N\gg 1}\sum_{N_1,\dots,N_{k+2}\atop N_1\sim
% N_2 \sim N}\sum_{k N_5\ll M\lesssim N_3}$ for simplicity.
  It is worth noting that $R\ll L=MN_1^\al$ since $N_1\gg1$.
  For $J_{\I,3,1}^{(3)}$, we use the argument of $J_{\I,3,1}^{(2)}$.
  Lemmas \ref{lem_comm1} and \ref{extensionlem}, the H\"older inequality,
\eqref{eq4.6}, \eqref{eq4.7} and \eqref{eq4.5} show
  \EQQS{
    J_{\I,3,1}^{(3)}
    &\lesssim \sum_{N_1,\dots,N_{k+2}}\sum_{M\lesssim N_3}
      \om_{N_1}^2 N_1^{2s}M
      \|P_{N_1}Q_{\gtrsim L}(\1_{t,R}^{\textrm{low}}\check{u})\|_{L_{t,x}^2}
      \|P_{N_2}\1_{t,R}^{\textrm{low}}\check{u}\|_{L_{t,x}^2}
      \prod_{j=3}^{k+2}\|P_{N_j}\check{u}\|_{L_{t,x}^\I}\\
    &\lesssim_k \sum_{N_1,N_3\dots,N_{k+2}}\sum_{M\lesssim N_3}
      \om_{N_1}^2 N_1^{2s-\al}\|P_{N_1}\check{u}\|_{X^{0,1}}
      \|P_{\sim N_1}\1_{t}\check{u}\|_{L_{t,x}^2}
      \prod_{j=3}^{k+2}\|P_{N_j}\check{u}\|_{L_{t,x}^\I}\\
    &\quad+T^{1/4} C_0^k \sum_{N_1\ge N_3}\sum_{M\lesssim N_3}
      \om_{N_1}^2 N_1^{2s-\al-1/12}M^{-1/3}\|P_{N_1}\check{u}\|_{X^{0,1}}
      \|P_{\sim N_1}\check{u}\|_{L_{t}^\I L_x^2}\\
    &\lesssim_k T^{1/4} C_0^{k} \|\check{u}\|_{L_t^\I H_\om^{s}}
      \|\check{u}\|_{X_\om^{s-1,1}}
    \lesssim T^{1/4} C_0^{k} \|u\|_{L_T^\I H_\om^{s}}
      \|u\|_{Z_{\om,T}^s}.
  }
  We can estimate the contribution $J_{\I,3,2}^{(3)}$ by the same way.
  For the contribution $J_{\I,3,3}^{(3)}$,
  similarly to $J_{\I,3,3}^{(2)}$, we see from \eqref{eq4.3},  \eqref{eq4.5} and Lemma \ref{extensionlem} that
  \EQQS{
    J_{\I,3,3}^{(3)}
    &\lesssim_k T^{1/2}
      \|\check{u}\|_{L_{t}^\I H_x^{s(\al)}}^{k-1}
      \sum_{N_1\gtrsim N_3\ge 1}
      \om_{N_1}^2 N_1^{2s-\al}N_3^{\be(\al)}
      \|P_{N_1}\check{u}\|_{L_{t}^\I L_x^2}^2
      \|D_x^{1/2}P_{N_3}\check{u}\|_{X^{0,1}}\\
    &\lesssim T^{1/2} C_0^{k-1}
      \|\check{u}\|_{X^{s(\al)-1,1}}
      \|\check{u}\|_{L_t^\I H_\om^s}^2
    \lesssim_k T^{1/2} C_0^{k} \|u\|_{L_T^\I H_\om^s}^2
  }
  In a similar manner, we can estimate the contribution $J_{\I,3,j}^{(3)}$ for $j=4,\dots,k+2$ by the same bound.

  Finally, we consider the contribution of $\phi_N^2(\xi_3)\xi_3$.
  We may assume that $N_3\gg k N_4$. Otherwise the proof is the same as the contribution of $A_2$.
  When $N_3\gg k N_4$, we can obtain the desired estimate as in Case 2.
  This completes the proof.
\end{proof}

\section{Estimate for the difference}

We provide the estimate (at the regularity $s-1$) for the difference $w$ of two solutions $u,v$ of \eqref{eq1}. In this section, we do not use the frequency envelope, so we always argue on the standard Sobolev space $H^s(\T)$.

\begin{prop}\label{difdif}
  Let $0<T<1$, $\al\in[1,2]$ and $2\ge s\ge s(\al):=1/2+2\be(\al)$.
  Let $u$ and $v$ be two solutions of \eqref{eq1} belonging to $Z^s_T$ and associated with the inital data $u_0\in H^s(\T)$ and $v_0\in H^s(\T)$, respectively.   Then there exists an entire function $ G=G[f] $ that is
increasing  and non negative on $ \R_+ $  such that
  \EQS{\label{eq_difdif}
    \|w\|_{L_T^\I H_x^{s-1}}^2
    \le \|u_0-v_0\|_{H_x^{s-1}}^2 +T^{1/4}
     G(\|u\|_{Z^s_T}+\|v\|_{Z^s_T})
    \|w\|_{Z^{s-1}_T}
    \|w\|_{L_T^\I H_x^{s-1}},
  }
  where we set $w=u-v$.
\end{prop}

\begin{proof}
  According to Lemma \ref{lem1}, we notice that $u,v\in Z_T^s$.
  Observe that $w$ satisfies
  \EQS{\label{eq_w}
    \p_t w+L_{\al+1}w=-\p_x(f(u)-f(v))
  }
 Rewriting $ f(u)-f(v) $ as
 \EQQS{
   f(u)-f(v)=\sum_{k\ge 1} \frac{f^{(k)}(0)}{k!} (u^k-v^k)=\sum_{k\ge 1} \frac{f^{(k)}(0)}{k!}w \sum_{i=0}^{k-1} u^i  v^{k-1-i}
 }
 and arguing as in the proof of Proposition \ref{prop_apri}, we see from \eqref{eq_w} that for $t\in[0,T]$
  \EQQS{
    \|w(t)\|_{H_x^{s-1}}^2
    \le \|w_0\|_{H_x^{s-1}}^2
      +  2\sum_{k\ge 1}    \frac{|f^{(k)}(0)|}{(k-1)!}  \max_{i\in \{0,..,k-1\}} I_{k,i}^t ,
  }
  where $w_0=u_0-v_0$ and
  \EQQS{
    I_{k,i}^t:= \sum_{N\ge 1}  N^{2(s-1)} \bigg|\int_0^t\int_\T  u^i  v^{k-1-i} wP_N^2\p_x w\,  dxdt'\bigg| \; .
  }
  It is clear that $I^t_{1,0}=0$ by the integration by parts.
  Therefore we are reduced to estimate  the contribution of
  \EQS{\label{Ik}
  I_{k+1}^t=  \sum_{N\ge 1}  N^{2(s-1)} \bigg|\int_0^t\int_\T  \textbf{z}^k wP_N^2\p_x w\,  dxdt'\bigg| \;
  }
  where $\textbf{z}^k $ stands for $ u^i v^{k-i} $ for some $i\in \{0,..,k\}
$.  We set
  \EQQS{
    C_0:=\|u\|_{Z^s_T}+\|v\|_{Z^s_T}.
  }
  We claim  that for any $ k\ge 1$ it holds
  \EQS{\label{eq4.4FFFdif}
  I_{k+1}^t \le C^k T^{1/4} C_0^k G(C_0)
    \|w\|_{Z^{s-1}_T}
    \|w\|_{L_T^\I H^{s-1}_x}.
   }
  that  clearly  leads \eqref{eq_difdif} by taking  \eqref{PP} and  \eqref{eq4.4FF} into account
  since  $\sum_{k\ge 1}  \frac{|f^{(k+1)}(0)|}{k!} C^k C_0^k <\infty$.

 In the sequel we fix $ k\ge 1 $ and we estimate $ I_{k}^t $.
 We also use the notation $a\lesssim_k b$ defined in \eqref{eq_lesssim}.
 The contribution of the sum over  $N\lesssim 1$ in \eqref{Ik}  is easily
estimated thanks to \eqref{eq2.2} by
  \EQS{\label{eq5.1}
    \begin{aligned}
      &\sum_{N\lesssim1} (1\vee N)^{2(s-1)}\bigg|\int_0^t\int_\T \zz^{k}wP_N^2\p_x w dxdt'\bigg|\\
      &\lesssim T\sum_{N\lesssim 1}\|w\|_{L_T^\I H_x^{s-1}}
      \|\zz^k P_N^2\p_x w\|_{L_T^\I H_x^{1-s}}
      \lesssim_k T C_0^{k} \|w\|_{L_T^\I H_x^{s-1}}^2.
    \end{aligned}
  }
  In the last inequality, we used $1-s<1/2$.
  Therefore, in what follows, we can assume that $N\gg 1.$
  A similar argument to \eqref{eq_4.8} yields
  \EQQS{
    &\sum_{N\gg1} N^{2(s-1)}\bigg|\int_0^t\int_\T \zz^{k}wP_N^2\p_x w dxdt'\bigg|\\
    &\le \sum_{N\gg1}\sum_{N_1,\dots,N_{k+2}}N^{2(s-1)}\bigg|\int_0^t\int_\T \Pi(P_{N_1}w,P_{N_2}w)\prod_{j=3}^{k+2}P_{N_j}z_j dxdt'\bigg|
    =:J_t,
  }
  where $\Pi(f,g)$ is defined by \eqref{def_pi} and $ z_i\in \{u,v\} $ for $ i\in\{3,..,k+2\} $.
  By  symmetry, we may assume that $N_1\ge N_2$. Moreover, we may assume that $ N_3\ge N_4 $ for $k=2 $ and  $N_3\ge N_4\ge N_5=\max_{j\ge 5}N_j$ for $ k\ge 3$. Note again that
   the cost of this choice is a constant factor  less than $(k+2)^4 $. It
is also worth noticing that the frequency projectors in  $\Pi(\cdot,\cdot)$ ensure that $ N_1\sim N $ or $ N_2\sim N $ and in particular
   $ N_1\gtrsim N $. We also remark that we can assume that $ N_3\ge 1 $ since the contribution of $ N_3=0 $ does vanish by integration by parts. Finally we note that we can also assume that $ N_2\ge 1 $ since in the case $ N_2=0 $ we must have $ N_3 \gtrsim N_1/k $ and it is easy to check that
\EQQS{
J_t & \lesssim k T\|w\|_{L^\infty_T H_x^{s-1}} \|z_3 \|_{L^\infty_T H_x^s}  \|P_0 w\|_{L^\I_{T,x}}
\prod_{j=4}^{k+2} \|z_j\|_{L^\infty_T H_x^{s}}
\lesssim_k T C_0^{k} \|w\|_{L^\infty_T H_x^{s-1}}^2 \; .
}

  We consider the contribution of $J_t$, dividing it into three cases:
  \begin{itemize}
    \item $N_1\lesssim k N_4$   ($k\ge 2$),
    \item $N_1\gg k N_4$ and $N_2\gtrsim N_3$ (or $k=1 $ and $N_2\gtrsim N_3$),
    \item $N_1\gg k N_4$ and $N_2\ll N_3$ (or $k=1 $ and $N_2\ll N_3$).
  \end{itemize}

  \noindent
  \textbf{Case 1: $N_1\lesssim k N_4$ ($k\ge 2$).}

  H\"older's and Young's inequalities, \eqref{eq4.9}, \eqref{eq_stri1} and \eqref{eq_stri3} imply that
   \EQQS{
     J_t
     &\le\sum_{N_1,\dots,N_{k+2}}(N_1^{2s-1}+N_2^{2s-1})
       \prod_{j=1}^2\|P_{N_j}w\|_{L_{T,x}^4}
       \prod_{n=3}^4\|P_{N_n}z_n\|_{L_{T,x}^4}
       \prod_{j=5}^{k+2}\|P_{N_j}z_j\|_{L_{T,x}^{\I}}\\
     &\lesssim_k C_0^{k-2}
       \sum_{N_3\ge N_4\gtrsim k^{-1} N_1\ge  k^{-1}  N_2}N_1^{2s-1}
       \|P_{N_1}w\|_{L_{T,x}^4}\|P_{N_2}w\|_{L_{T,x}^4}
       \|P_{N_3}z_3\|_{L_{T,x}^4}  \|P_{N_4}z_4\|_{L_{T,x}^4}\\
     &\lesssim_k C_0^{k-2}
       \sum_{N_3\ge N_4\gtrsim k^{-1}  N_1\ge k^{-1}  N_2} \hspace*{-10mm} k^{2s-2\be(\al)}
       \bigg(\frac{N_1}{k N_4}\bigg)^{2s-\be(\al)-1/2}
       \bigg(\frac{N_2}{k N_4}\bigg)^{1/2-\be(\al)}   \bigg(\frac{N_4}{N_3}\bigg)^{s-\beta(\al)}
       \\
     &\quad\times \prod_{j=1}^2
       \|D_x^{-1/2+\be(\al)} P_{N_j}w\|_{L_{T,x}^4}
       \prod_{n=3}^4
       \|D_x^{s-\be(\al)} P_{N_n}z_n\|_{L_{T,x}^4}\\
     &\lesssim_k T^{1/2} C_0^{k} G(C_0) \|w\|_{L_T^\I H_x^{s(\al)-1}}^2
   }
   since $2s-\beta(\al)-1/2>0 $ and  $ s-\beta(\al)>1/2-\be(\al) >0 $. Note that in the last step we also used that $-1/2+\beta(\alpha)=s(\alpha)-1-\beta(\alpha) $ since  $s(\alpha)=1/2+2\beta(\alpha)$.

  \noindent
  \textbf{Case 2: $N_1\gg k N_4$ and $N_2\gtrsim N_3$ (or $ k=1 $ and $N_2\gtrsim  N_3$).}

  %It holds that $N_1\sim N$ or $N_2\sim N$, which implies that $N\sim N_1\sim N_2$.
  In this case, the contribution of $J_t$ in this case can be estimated by the same way as the contribution of $A_1$ in Proposition \ref{prop_apri},
  replacing $N^{2s}$, $P_{N_1}u$, $P_{N_2}u$, $P_{N_j}u$ for $j=3,\dots,k+2$ by $N^{2(s-1)}$, $P_{N_1}w$, $P_{N_2}w$, $P_{N_j}z_j$ for $j=3,\dots,k+2$, respectively.

  \if0
  As the result, we can obtain
  \EQQS{
    J_t
    \le CT^{1/4}
    (\|w\|_{X^{s-2,1}}+\|w\|_{L_t^\I H_x^{s-1}})
    \|w\|_{L_t^\I H_x^{s-1}}
  }
  for $s\ge s(\al)$.
  \fi

  \noindent
  \textbf{Case 3: $N_1\gg k N_4$ and $N_2\ll N_3$ (or $ k=1 $ and $N_2\ll N_3$).}

  Note that in this case $N_1\sim N_3\sim N \gg N_2\vee N_4$.
  We further divide the contribution of $J_t$ into three cases:
  \begin{itemize}
    \item  $N_2\gg  k N_4$ or $ k=1$,
    \item $ k N_4\gtrsim N_2 \gtrsim N_5 $ (or $k=2$ and $N_4\gtrsim N_2$),
    \item $N_2 \ll N_5$ ($k\ge 3$).
  \end{itemize}

  \underline{Subcase 3.1: $N_2\gg k N_4$ or $ k=1$.}
Recall that we have $ N_2\ge 1 $ and thus this subcase contains the case $ N_4 =0 $.
  We take the extensions $\check{w}=\rho_T(w)$ (resp. $\check{z}_j=\rho_T(z_j)$ for $j\ge 3$) of $w$ (resp. $z_j$ for $j\ge3$) defined in \eqref{def_ext}.
  For simplicity, with a slight abuse of notation, we shall use the following notation:
  \EQQS{
    J_\I^{(3.1)}(u_1,\cdots,u_{k+2})
    :=\sum_{N\gg 1}\sum_{N_1,\dots,N_{k+2}} N^{2(s-1)}
      \bigg|\int_\R\int_\T\Pi(u_1,u_2)
      \prod_{j=3}^{k+2}u_jdxdt'\bigg|.
  }
  Setting $R=N_1^{1/3}N_2^{4/3}$, we divide $J_t$ as
  \EQQS{
    J_t
    &\le J_\I^{(3.1)}(P_{N_1}\1_{t,R}^{\textrm{high}}\check{w},
      P_{N_2}\check{w},P_{N_3}\1_t\check{z}_3,P_{N_4}\check{z}_4,\cdots, P_{N_{k+2}}\check{z}_{k+2})\\
    &\quad+J_\I^{(3.1)}(P_{N_1}\1_{t,R}^{\textrm{low}}\check{w},
      P_{N_2}\check{w},
      P_{N_3}\1_{t,R}^{\textrm{high}}\check{z}_3,
      P_{N_4}\check{z}_4,\cdots,
      P_{N_{k+2}}\check{z}_{k+2})\\
    &\quad+J_\I^{(3.1)}(P_{N_1}\1_{t,R}^{\textrm{low}}\check{w},
      P_{N_2}\check{w},
      P_{N_3}\1_{t,R}^{\textrm{low}}\check{z}_3,
      P_{N_4}\check{z}_4,\cdots,
      P_{N_{k+2}}\check{z}_{k+2})\\
    &=:J_{\I,1}^{(3.1)}+J_{\I,2}^{(3.1)}+J_{\I,3}^{(3.1)}.
  }
  % \EQQS{
  % J_t
  % &\le \sum_{N\gg 1}\sum_{N_1,\dots,N_{k+2}} N^{2(s-1)}
  %   \bigg|\int_\R\int_\T
  %    \Pi(P_{N_1}\1_{t,R}^{\textrm{high}}\check{w},P_{N_2}\check{w})
  %    \prod_{j=3}^{k+2}P_{N_j}\check{z}_j \, dxdt'\bigg|\\
  % &\quad+\sum_{N\gg 1}\sum_{N_1,\dots,N_{k+2}} N^{2(s-1)}
  %   \bigg|\int_\R\int_\T
  %    \Pi(P_{N_1}\1_{t,R}^{\textrm{low}}\check{w},P_{N_2}\check{w})
  %    P_{N_3}\1_{t,R}^{\textrm{high}}\check{z}_3
  %    \prod_{j=4}^{k+2}P_{N_j}\check{z}_j \, dxdt'\bigg|\\
  % &\quad+\sum_{N\gg 1}\sum_{N_1,\dots,N_{k+2}} N^{2(s-1)}
  %    \bigg|\int_\R\int_\T
  %     \Pi(P_{N_1}\1_{t,R}^{\textrm{low}}\check{w},P_{N_2}\check{w})
  %     P_{N_3}\1_{t,R}^{\textrm{low}}\check{z}_3
  %     \prod_{j=4}^{k+2}P_{N_j}\check{z}_j \, dxdt'\bigg|
  % =:\sum_{n=1}^3 J_{\I,n}^{(3.1)}.
  % }
  For $J_{\I,1}^{(3.1)}$, we see from \eqref{eq4.1} that $\|\1_{t,R}^{\textrm{high}}\|_{L^1}\lesssim T^{1/4}N_1^{-1/4}N_2^{-1}$, which gives
  \EQQS{
    J_{\I,1}^{(3.1)}
    &\lesssim \sum_{N_1,\dots,N_{k+2}}N_1^{2s-1}
    \|\1_{t,R}^{\textrm{high}}\|_{L^1}\|P_{N_1}\check{w}\|_{L_t^\I L_x^2}
    \|P_{N_2}\check{w}\|_{L_{t,x}^\I}
    \|P_{N_3}\check{z}_3\|_{L_t^\I L_x^2}
    \prod_{j=4}^{k+2}\|P_{N_j}\check{z}_j\|_{L_{t,x}^\I}\\
    &\lesssim_k T^{1/4} C_0^{k-1} \sum_{N_1,N_2}N_1^{2s-5/4}N_2^{-1}
    \|P_{N_1}\check{w}\|_{L_t^\I L_x^2}
    \|P_{N_2}\check{w}\|_{L_{t,x}^\I}
    \|P_{\sim N_1}\check{z}_3\|_{L_t^\I L_x^2}\\
    &\lesssim_k T^{1/4} C_0^{k}
    \|\check{w}\|_{L_t^\I H_x^{s-1}}^2
    \lesssim_k T^{1/4} C_0^{k}
    \|w\|_{L_T^\I H_x^{s-1}}^2.
  }
  Here, in the first inequality, we used $N_2^{2s-1}\le N_1^{2s-1}$ since $s\ge s(\al)>1/2$.
  In the last inequality, we also used \eqref{eq2.1single}.
  In a similar manner, we can also estimate $J_{\I,2}^{(3.1)}$ by the same bound as that of $J_{\I,1}^{(3.1)}$.
  % \EQQS{
  %   J_{\I,2}^{(3.1)}
  %   \lesssim_k T^{1/4}C_0^{k}
  %   \|\check{w}\|_{L_t^\I H_x^{s-1}}^2
  %   \lesssim_k T^{1/4}
  %   \|w\|_{L_T^\I H_x^{s-1}}^2.
  % }
  For $J_{\I,3}^{(3.1)}$, we note that $|\Om_{k+1}|\gtrsim N_2 N_1^\al=:L$ by Lemma \ref{lem_res1}.
  This enables us to decompose $J_{\I,3}^{(3.1)}$ as
  \EQQS{
    J_{\I,3}^{(3.1)}
    &\le J_\I^{(3.1)}(P_{N_1}Q_{\gtrsim L}(\1_{t,R}^{\textrm{low}}\check{w}),
      P_{N_2}\check{w},
      P_{N_3}\1_{t,R}^{\textrm{low}}\check{z}_3,\cdots,
      P_{N_{k+2}}\check{z}_{k+2})\\
    &\quad +J_\I^{(3.1)}(P_{N_1}Q_{\ll L}(\1_{t,R}^{\textrm{low}}\check{w}),
      P_{N_2}\check{w},
      P_{N_3}Q_{\gtrsim L}(\1_{t,R}^{\textrm{low}}\check{z}_3),\cdots,
      P_{N_{k+2}}\check{z}_{k+2})\\
    &\quad +J_\I^{(3.1)}
      (P_{N_1}Q_{\ll L}(\1_{t,R}^{\textrm{low}}\check{w}),
      P_{N_2}Q_{\gtrsim L}\check{w},
      P_{N_3}Q_{\ll L}(\1_{t,R}^{\textrm{low}}\check{z}_3),\cdots,
      P_{N_{k+2}}\check{z}_{k+2})+\cdots\\
    &\quad
     +J_\I^{(3.1)}(P_{N_1}Q_{\ll L}(\1_{t,R}^{\textrm{low}}\check{w}),
        P_{N_2}Q_{\ll L}\check{w},
        P_{N_3}Q_{\ll L}(\1_{t,R}^{\textrm{low}}\check{z}_3),\cdots,
        P_{N_{k+2}}Q_{\gtrsim L}\check{z}_{k+2})\\
    &=:J_{\I,3,1}^{(3.1)}+\cdots+J_{\I,3,k+2}^{(3.1)},
  }
  where $J_{\I,3,n}^{(3.1)}$ for $4\le n\le k+2$ corresponds to the term in which $Q_{\gtrsim L}$ lands on $P_{N_n}\check{z}_n$.
  % \EQQS{
  %   J_{\I,3}^{(3.1)}
  %   &\le\sum N^{2(s-1)}
  %     \bigg|\int_\R\int_\T
  %      \Pi(P_{N_1}Q_{\gtrsim N_2 N_1^\al}(\1_{t,R}^{\textrm{low}}\check{w}),P_{N_2}\check{w}) P_{N_3}\1_{t,R}^{\textrm{low}}\check{z}_3
  %      \prod_{j=4}^{k+2}P_{N_j}\check{z}_j\, dxdt'\bigg|\\
  %   &\quad+\sum N^{2(s-1)}
  %     \bigg|\int_\R\int_\T
  %      \Pi(P_{N_1}Q_{\ll N_2 N_1^\al}(\1_{t,R}^{\textrm{low}}\check{w}),P_{N_2}\check{w})P_{N_3}Q_{\gtrsim N_2 N_1^\al}(\1_{t,R}^{\textrm{low}}\check{z}_3)
  %      \prod_{j=4}^{k+2}P_{N_j}\check{z}_j \, dxdt'\bigg|\\
  %   &\quad+\sum N^{2(s-1)}
  %     \bigg|\int_\R\int_\T
  %      \Pi(P_{N_1}Q_{\ll N_2 N_1^\al}(\1_{t,R}^{\textrm{low}}\check{w}),P_{N_2}Q_{\gtrsim N_2 N_1^\al}\check{w})\\
  %   &\quad\quad\times P_{N_3}Q_{\ll N_2 N_1^\al}
  %     (\1_{t,R}^{\textrm{low}}\check{z}_3)
  %      \prod_{j=4}^{k+2}P_{N_j}\check{z}_j \, dxdt'\bigg|\\
  %   &\quad+\sum_{n=4}^{k+2}\sum N^{2(s-1)}
  %     \bigg|\int_\R\int_\T
  %      \Pi(P_{N_1}Q_{\ll N_2 N_1^\al}(\1_{t,R}^{\textrm{low}}\check{w}),P_{N_2}Q_{\ll KN_1^\al}\check{w})\\
  %   &\quad\quad\times P_{N_3}Q_{\ll N_2 N_1^\al}
  %     (\1_{t,R}^{\textrm{low}}\check{z}_3)
  %   \bigg(\prod_{j=4}^{n-1}P_{N_j}
  %       Q_{\ll N_2N_1^\al}\check{z}_j\bigg)P_{N_n}Q_{\gtrsim N_2 N_1^\al}\check{z}_j
  %       \prod_{j=n+1}^{k+2}P_{N_j}\check{z}_j\,  dxdt'\bigg|\\
  %   &=:\sum_{j=1}^{k+2}J_{\I,3,j}^{(3.1)},
  % }
  % where $\sum$ denotes $\sum_{N\gg 1}\sum_{N_1,\dots,N_{k+2}}$ for simplicity.
  It is worth noting that $R\ll L$.
  Then, the H\"older inequality, \eqref{eq4.6}, \eqref{eq4.7} and Lemma \ref{extensionlem} imply that
  \EQQS{
    &J_{\I,3,1}^{(3.1)}\\
    &\lesssim \sum_{N_1,\dots,N_{k+2}}N_1^{2s-1}
    \|P_{N_1}Q_{\gtrsim L}(\1_{t,R}^{\textrm{low}}\check{w})\|_{L_{t,x}^2}
    \|P_{N_2}\check{w}\|_{L_{t,x}^\I}
    \|P_{N_3}\1_{t,R}^{\textrm{low}}\check{z}_3\|_{L_{t,x}^2}
    \prod_{j=4}^{k+2}\|P_{N_j}\check{z}_j\|_{L_{t,x}^\I}\\
    &\lesssim_k C_0^{k-1}\sum_{N_1,N_2}
      N_1^{2s-1-\al}
      \|P_{N_1}\check{w}\|_{X^{0,1}}
      \|P_{N_2}\check{w}\|_{L_{t}^\I H_x^{-1/2}}
      \|P_{\sim N_1}\1_{t,R}^{\textrm{low}}\check{z}_3\|_{L_{t,x}^2}\\
    &\lesssim_k C_0^{k-1} \sum_{N_1,N_2}
      N_1^{2s-1-\al}
      \|P_{N_1}\check{w}\|_{X^{0,1}}
      \|P_{N_2}\check{w}\|_{L_{t}^\I H_x^{-1/2}}
      \|P_{\sim N_1}\1_{t}\check{z}_3\|_{L_{t,x}^2}\\
    &\quad+T^{1/4} C_0^{k-1} \sum_{N_1,N_2}
      N_1^{2s-13/12-\al}N_2^{-1/3}
      \|P_{N_1}\check{w}\|_{X^{0,1}}
      \|P_{N_2}\check{w}\|_{L_{t}^\I H_x^{-1/2}}
      \|P_{\sim N_1}\check{z}_3\|_{L_t^\I L_x^2}\\
    &\lesssim_k T^{1/4} C_0^{k}
      \|\check{w}\|_{L_t^\I H_x^{s-1}}\|\check{w}\|_{X^{s-2,1}}
    \lesssim_k T^{1/4} C_0^{k}
      \|w\|_{Z_T^{s-1}}\|w\|_{L_T^\I H_x^{s-1}}.
  }
  By the same way, we can obtain
  \EQQS{
    J_{\I,3,2}^{(3.1)}
    \lesssim_k T^{1/4} C_0^{k-1}
      \|\check{w}\|_{L_t^\I H_x^{s-1}}^2
      \|\check{z}_3\|_{X^{s-1,1}}
    \lesssim_k T^{1/4} C_0^{k}
     \|w\|_{L_T^\I H_x^{s-1}}^2.
  }
  Next, we consider the contribution of $J_{\I,3,3}^{(3.1)}$.
  Lemma \ref{extensionlem}, the H\"older inequality and \eqref{eq4.3} show
  \EQQS{
    J_{\I,3,3}^{(3.1)}
    &\lesssim \sum_{N_1,\dots,N_{k+2}}N_1^{2s-1}
      \|P_{N_1}\1_{t,R}^{\textrm{low}}\check{w}\|_{L_{t,x}^2}
      \|P_{N_2}Q_{\gtrsim L} \check{w}\|_{L_{t}^2 L_x^\I}
      \|P_{N_3}\check{z}\|_{L_{t}^\I L_x^2}
      \prod_{j=4}^{k+2}\|P_{N_j}\check{z}_j\|_{L_{t,x}^\I}\\
    &\lesssim_k T^{1/2} C_0^{k-1}
      \sum_{N_1\ge N_2}N_1^{2s-1-\al}N_2^{-1}
      \|P_{N_1}\check{w}\|_{L_{t}^\I L_x^2}
      \|D_x^{1/2} P_{N_2} \check{w}\|_{X^{0,1}}
      \|P_{\sim N_1}\check{z}\|_{L_{t}^\I L_x^2}\\
    &\lesssim_k T^{1/2} C_0^{k-1}
      \|\check{w}\|_{L_t^\I H_x^{s-1}} \|\check{z}_3\|_{L_t^\I H_x^{s}}
    \sum_{N_1\ge N_2} N_1^{-\al} N_2^{-1}\|P_{N_2}\check{w}\|_{X^{1/2,1}}\\
    &\lesssim_k T^{1/2} C_0^{k}
    \|\check{w}\|_{X^{s-2,1}}
    \|\check{w}\|_{L_t^\I H_x^{s-1}}
    \lesssim T^{1/2} C_0^{k}
    \|w\|_{Z_T^{s-1}}
    \|w\|_{L_T^\I H_x^{s-1}}
  }
  since $s\ge s(\al)>1/2$.
  Similarly, we obtain
  \EQQS{
    &J_{\I,3,4}^{(3.1)}\\
    &\lesssim_k C_0^{k-2} \sum_{N_1,N_2,N_4}N_1^{2s-1}
    \|P_{N_1}\1_{t,R}^{\textrm{low}}\check{w}\|_{L_{t,x}^2}
    \|P_{N_2}\check{w}\|_{L_{t,x}^\I}
    \|P_{\sim N_1}\check{z}_3\|_{L_{t}^\I L_x^2}
    \|P_{N_4}Q_{\gtrsim L}\check{z}_4\|_{L_t^2 L_x^\I}\\
    &\lesssim_k T^{1/2} C_0^{k-2}
     \sum_{N_1,N_2,N_4}N_1^{2s-\al-1} N_2^{-1}
    \|P_{N_1}\check{w}\|_{L_{t}^\I L_x^2}
    \|P_{N_2} \check{w}\|_{L_{t}^\I H_x^{1/2}}\\
    &\quad\quad\times
    \|P_{\sim N_1}\check{z}_3\|_{L_{t}^\I L_x^2}
    \|P_{N_4}\check{z}_4\|_{X^{1/2,1}}\\
    &\lesssim T^{1/2} C_0^{k-2}
    \|\check{w}\|_{L_t^\I H_x^{s-1}}
    \|\check{w}\|_{L_t^\I H_x^{-1/2}}
    \|\check{z}_3\|_{L_t^\I H_x^{s}}
    \sum_{N_1\ge N_4} N_1^{-\al}
    \|P_{N_4}\check{z}_4\|_{X^{1/2,1}}\\
    &\lesssim T^{1/2}C_0^{k-1}
    \|\check{z}_4\|_{X^{s(\al)-1,1}}
    \|\check{w}\|_{L_t^\I H_x^{s-1}}^2
    \lesssim_k T^{1/2} C_0^k
    \|w\|_{L_T^\I H_x^{s-1}}^2.
  }
  We see from the above argument that $J_{\I,3,j}^{(3.1)}$ for $5\le j\le k+2$ can be estimated by the same bound as that of $J_{\I,3,4}^{(3.1)}$.
  % The same argument yields for $5\le j\le k+2$,
  % \EQQS{
  %   J_{\I,3,j}^{(3.1)}
  %   \lesssim_k   T^{1/2}C_0^{k}
  %   \|w\|_{L_T^\I H_x^{s-1}}^2.
  % }

  \underline{Subcase 3.2: $N_5 \lesssim  N_2\lesssim kN_4$ (or $k=2$ and $N_4\gtrsim N_2$).} Note that the cases $ N_2=0 $ or $ N_4=0 $ have
been already treated so that we can assume that $ N_2\ge 1 $ and $ N_4\ge
1$.
  It suffices to consider the case $N_5\lesssim N_2\lesssim N_4$ since $k\ge 1$.
  First we set
  \EQQS{
    J_t
    &\le \sum_{N\gg1}\sum_{N_1,\dots,N_{k+2}}N^{2(s-1)}
    \bigg|\int_0^t\int_\T \p_x P_N^2 P_{N_1}wP_{N_2}w\prod_{j=3}^{k+2}P_{N_j}z_j\, dxdt'\bigg|\\
    &\quad+\sum_{N\gg1}\sum_{N_1,\dots,N_{k+2}}N^{2(s-1)}
    \bigg|\int_0^t\int_\T  P_{N_1}w\p_x P_N^2P_{N_2}w\prod_{j=3}^{k+2}P_{N_j}z_j\,  dxdt'\bigg|\\
    &=:I_t^{(3.2)}+J_t^{(3.2)}.
  }
  For $I_t^{(3.2)}$, we further divide it as in Case 3 in Proposition \ref{prop_apri}:
  \EQQS{
    &I_t^{(3.2)}\\
    &\le \sum N^{2(s-1)}
    \bigg|\int_0^t\int_\T \p_x P_N^2 P_{N_1}w P_{N_3}z_3
    P_{\lesssim  k N_5\vee 1} \bigg(P_{N_2}w\prod_{j=4}^{k+2}P_{N_j}z\bigg) dxdt'\bigg|\\
    &\quad+ \sum\sum_{k N_5\vee 1\ll M\lesssim  N_4} N^{2(s-1)}
    \bigg|\int_0^t\int_\T \p_x P_N^2 P_{N_1}w P_{N_3}z_3
    P_{M} \bigg(P_{N_2}w\prod_{j=4}^{k+2}P_{N_j}z_j \, \bigg) dxdt'\bigg|\\
    &=:I_{t,1}^{(3.2)}+I_{t,2}^{(3.2)},
  }
  where $\sum$ denotes $\sum_{N\gg1}\sum_{N_1,\dots,N_{k+2}}$ for the sake of simplicity.
  Remark that if $k=2$, then the term $I_{t,1}^{(3.2)}$ does not appear, and it holds that $I_t^{(3.2)}\le I_{t,2}^{(3.2)}$ with the summation $\sum_{M\lesssim N_4}$ instead of $\sum_{k N_5\vee 1\ll M\lesssim  N_4}$.
  For $I_{t,1}^{(3.2)}$, H\"older's, Bernstein's and Minkowski's inequalities, \eqref{eq_stri2} and \eqref{eq_stri4} give
  \EQQS{%\label{eq5.2}
  \begin{aligned}
    I_{t,1}^{(3.2)}
    &\lesssim\sum_{N_1,\dots,N_{k+2}}N_1^{2s-1}(k^{3/4} N_5^{3/4}\vee 1)
    \int_0^t\|P_{N_1}w\|_{L_x^2}\|P_{N_3}z_3\|_{L_x^2}
    \|P_{N_2}w\|_{L_x^4}\\
    &\quad\times\|P_{N_4}z_4\|_{L_x^4}
    \|P_{N_5}z_5\|_{L_x^4}
    \prod_{j=6}^{k+2}\|P_{N_j}z_j\|_{L_x^\I}dt'\\
    &\lesssim_k C_0^{k-3}
      \|w\|_{L_T^\I H_x^{s-1}}\|z_3\|_{L_T^\I H_x^{s}}
      \sum_{ N_4\gtrsim N_2 \gtrsim  N_5} \bigg(\frac{N_2}{N_4}\bigg)^{5/12}
      \bigg(\frac{N_5\vee 1}{N_4}\bigg)^{1/6}\\
    &\quad\times \|D_x^{-5/12}P_{N_2}w\|_{L_T^3 L_x^4}
       \|D_x^{7/12}P_{N_4}z_4\|_{L_T^3 L_x^4}
       \|D_x^{7/12}P_{N_5}z_5\|_{L_T^3 L_x^4}\\
    &\lesssim_k T^{5/8} C_0^{k-2} G(C_0)
      \|w\|_{L_T^\I H_x^{s(\al)-1}}
      \|w\|_{L_T^\I H_x^{s-1}}\prod_{i=4}^5 \|z_i\|_{L_T^\I H_x^{s(\al)}}^2\\
    &\lesssim_k T^{5/8} C_0^{k} G(C_0) \|w\|_{L_T^\I H_x^{s-1}}^2
  \end{aligned}
  }
  since $s(\al)\ge 7/12+\be(\al)$.
  In order to estimate $I_{t,2}^{(3.2)}$, we further decompose it.
  We take the extensions $\check{w}=\rho_T(w)$ (resp. $\check{z}_j=\rho_T(z_j)$ for $j\ge 3$) of $w$ (resp. $z_j$ for $j\ge3$) defined in \eqref{def_ext}.
  As in Subcase 3.1, with a slight abuse of notation, we shall use the followig notation:
  \EQQS{
    &I_{\I,2}^{(3.2)}(u_1,\cdots,u_{k+2})\\
    &:=\sum_{N\gg1}\sum_{N_1,\dots,N_{k+2}}
      \sum_{kN_5\vee 1 \ll M\lesssim  N_4} N^{2(s-1)}
      \bigg|\int_\R\int_\T (\p_x P_N^2 u_1) u_3
      P_{M} \bigg(u_2\prod_{j=4}^{k+2}u_j \, \bigg) dxdt'\bigg|.
  }
  Putting $R=N_1^{1/3}M^{4/3}$, we see that
  \EQQS{
    I_{t,2}^{(3.2)}
    &\le I_{\I,2}^{(3.2)}(P_{N_1}\1_{t,R}^{\textrm{high}}\check{w},
      P_{N_2}\check{w},P_{N_3}\1_t\check{z}_3,P_{N_4}\check{z}_4,\cdots, P_{N_{k+2}}\check{z}_{k+2})\\
    &\quad+I_{\I,2}^{(3.2)}(P_{N_1}\1_{t,R}^{\textrm{low}}\check{w},
      P_{N_2}\check{w},
      P_{N_3}\1_{t,R}^{\textrm{high}}\check{z}_3,
      P_{N_4}\check{z}_4,\cdots,
      P_{N_{k+2}}\check{z}_{k+2})\\
    &\quad+I_{\I,2}^{(3.2)}(P_{N_1}\1_{t,R}^{\textrm{low}}\check{w},
      P_{N_2}\check{w},
      P_{N_3}\1_{t,R}^{\textrm{low}}\check{z}_3,
      P_{N_4}\check{z}_4,\cdots,
      P_{N_{k+2}}\check{z}_{k+2})\\
    &=:I_{\I,2,1}^{(3.2)}+I_{\I,2,2}^{(3.2)}+I_{\I,2,3}^{(3.2)}.
  }
  % \EQQS{
  %   I_{t,2}^{(3.2)}
  %   &\le \sum N^{2(s-1)}
  %   \bigg|\int_\R\int_\T \p_x P_N^2 P_{N_1}\1_{t,R}^{\textrm{high}}
  %    \check{w} P_{N_3}\check{z}_3
  %   P_{M} \bigg(P_{N_2}\check{w}\prod_{j=4}^{k+2}P_{N_j}\check{z}_j \, \bigg) dxdt'\bigg|\\
  %   &\quad+\sum N^{2(s-1)}
  %   \bigg|\int_\R\int_\T \p_x P_N^2 P_{N_1}\1_{t,R}^{\textrm{low}}
  %   \check{w} P_{N_3}\1_{t,R}^{\textrm{high}}\check{z}_3
  %   P_{M} \bigg(P_{N_2}\check{w}\prod_{j=4}^{k+2}P_{N_j}\check{z}_j \, \bigg) dxdt'\bigg|\\
  %   &\quad+\sum N^{2(s-1)}
  %   \bigg|\int_\R\int_\T \p_x P_N^2 P_{N_1}\1_{t,R}^{\textrm{low}}
  %   \check{w} P_{N_3}\1_{t,R}^{\textrm{low}}\check{z}_3
  %   P_{M} \bigg(P_{N_2}\check{w}\prod_{j=4}^{k+2}P_{N_j}\check{z}_j\, \bigg) dxdt'\bigg|\\
  %   &=:I_{\I,2,1}^{(3.2)}+I_{\I,2,2}^{(3.2)}+I_{\I,2,3}^{(3.2)},
  % }
  % where $\sum$ denotes $\sum_{N\gg1}\sum_{N_1,\dots,N_{k+2}}\sum_{kN_5\vee 1 \ll M\lesssim  N_4}$ for simplicity.
  For $I_{\I,2,1}^{(3.2)}$, we see from \eqref{eq4.1} that
  $\|\1_{t,R}^{\textrm{high}}\|_{L^1}\lesssim T^{1/4}N_1^{-1/4}M^{-1}$.
  Then, we have
  \EQQS{
    I_{\I,2,1}^{(3.2)}
    &\lesssim \sum_{N_1,\dots,N_{k+2}}
      \sum_{kN_5\vee 1\ll M\lesssim  N_4}
      N_1^{2s-1}\|\1_{t,R}^{\textrm{high}}\|_{L^1}
      \|P_{N_1}\check{w}\|_{L_t^\I L_x^2}
      \|P_{N_3}\check{z}_3\|_{L_t^\I L_x^2}\\
    &\quad\times\bigg
      \|P_{M} \bigg(P_{N_2}\check{w}
        \prod_{j=4}^{k+2}P_{N_j}
        \check{z}_j\, \bigg)\bigg\|_{L_{t,x}^\I}\\
    &\lesssim T^{1/4}\sum_{N_1,\dots,N_{k+2}}\sum_{ M\lesssim  N_4}N_1^{2s-5/4}\|P_{N_1}\check{w}\|_{L_t^\I L_x^2}\|P_{N_3}\check{z}_3\|_{L_t^\I L_x^2}\\
    &\quad\times\bigg\|P_{N_2}\check{w}
    \prod_{j=4}^{k+2}P_{N_j}\check{z}_j \, \bigg\|_{L_{t}^\I H_x^{-1/2}}.
  }
  Note that  \eqref{eq2.3} leads to
  \EQS{\label{eq5.3}
  \sum_{N_5,\dots,N_{k+2}}\bigg\|P_{N_2}\check{w}
  \prod_{j=4}^{k+2}P_{N_j}\check{z}_j\bigg\|_{H_x^{-1/2}}
  \lesssim_k C_0^{k-2} \|P_{N_4} \check{z}_4 \|_{H^{s_0}} \|P_{N_2}\check{w}\|_{H_x^{-1/2}}
  }
  with $1/2<s_0<s$.
  This together with \eqref{eq4.5} and \eqref{eq2.1single} shows that
  \EQQS{
    I_{\I,2,1}^{(3.2)}
    &\lesssim_k T^{1/4} C_0^{k-2}
    \|\check{w}\|_{L_t^\I H_x^{s-1}}\|\check{w}\|_{L_t^\I H_x^{s-1}}
    \|\check{z}_3\|_{L_t^\I H_x^{s}}  \|\check{z}_4\|_{L_t^\I H_x^{s}}
    \sum_{N_1}N_1^{-1/4}\\
    &\lesssim_k T^{1/4} C_0^{k}
   \|w\|_{L_T^\I H_x^{s-1}}^2.
  }
  Similarly, $I_{\I,2,2}^{(3.2)}$ can be estimated by the same bound as above.
  For $I_{\I,2,3}^{(3.2)}$, Lemma \ref{lem_res2} shows $|\Om_{k+1}|\gtrsim MN_1^\al\sim |\xi_2+\xi_4|N_1^\al$.
  Therefore, setting $L:=MN_1^\al$, we have
  \EQQS{
    I_{\I,2,3}^{(3.2)}
    &\le I_{\I,2}^{(3.2)}(P_{N_1}Q_{\gtrsim L}(\1_{t,R}^{\textrm{low}}\check{w}),
      P_{N_2}\check{w},
      P_{N_3}\1_{t,R}^{\textrm{low}}\check{z}_3,\cdots,
      P_{N_{k+2}}\check{z}_{k+2})\\
    &\quad +I_{\I,2}^{(3.2)}(P_{N_1}Q_{\ll L}(\1_{t,R}^{\textrm{low}}\check{w}),
      P_{N_2}\check{w},
      P_{N_3}Q_{\gtrsim L}(\1_{t,R}^{\textrm{low}}\check{z}_3),\cdots,
      P_{N_{k+2}}\check{z}_{k+2})\\
    &\quad +I_{\I,2}^{(3.2)}
      (P_{N_1}Q_{\ll L}(\1_{t,R}^{\textrm{low}}\check{w}),
      P_{N_2}Q_{\gtrsim L}\check{w},
      P_{N_3}Q_{\ll L}(\1_{t,R}^{\textrm{low}}\check{z}_3),\cdots,
      P_{N_{k+2}}\check{z}_{k+2})+\cdots\\
    &\quad
     +I_{\I,2}^{(3.2)}
       (P_{N_1}Q_{\ll L}(\1_{t,R}^{\textrm{low}}\check{w}),
        P_{N_2}Q_{\ll L}\check{w},
        P_{N_3}Q_{\ll L}(\1_{t,R}^{\textrm{low}}\check{z}_3),\cdots,
        P_{N_{k+2}}Q_{\gtrsim L}\check{z}_{k+2})\\
    &=:I_{\I,2,3,1}^{(3.2)}+\cdots+I_{\I,2,3,k+2}^{(3.2)},
  }
  where $I_{\I,2,3,n}^{(3.2)}$ for $4\le n\le k+2$ corresponds to the term in which $Q_{\gtrsim L}$ lands on $P_{N_n}\check{z}_n$.
%   \EQQS{
%     I_{\I,2,3}^{(3.2)}
%     &\le \sum N^{2(s-1)}
%     \bigg|\int_\R\int_\T \p_x P_N^2 P_{N_1}Q_{\gtrsim MN_1^\al} (\1_{t,R}^{\textrm{low}} \check{w}) P_{N_3}\1_{t,R}^{\textrm{low}}\check{z}_3
%     P_{M} \bigg(P_{N_2}\check{w}\prod_{j=4}^{k+2}P_{N_j}\check{z}_j\,\bigg) dxdt'\bigg|\\
%     &\quad+\sum N^{2(s-1)}
%     \bigg|\int_\R\int_\T \p_x P_N^2 P_{N_1}Q_{\ll MN_1^\al} (\1_{t,R}^{\textrm{low}} \check{w})
%     P_{N_3}Q_{\gtrsim MN_1^\al}(\1_{t,R}^{\textrm{low}}\check{z}_3)
%     P_{M} \bigg(P_{N_2}\check{w}\prod_{j=4}^{k+2}P_{N_j}\check{z}_j\, \bigg) dxdt'\bigg|\\
%     &\quad+\sum N^{2(s-1)}
%     \bigg|\int_\R\int_\T \p_x P_N^2 P_{N_1}Q_{\ll MN_1^\al} (\1_{t,R}^{\textrm{low}} \check{w})
%     P_{N_3}Q_{\ll MN_1^\al}(\1_{t,R}^{\textrm{low}}\check{z}_3)\\
%     &\quad\quad\times P_{M} \bigg(P_{N_2}Q_{\gtrsim MN_1^\al}\check{w}\prod_{j=4}^{k+2}P_{N_j}\check{z}_j\bigg) dxdt'\bigg|\\
%     &\quad+\sum_{n=4}^{k+2}\sum N^{2(s-1)}
%     \bigg|\int_\R\int_\T \p_x P_N^2 P_{N_1}Q_{\ll MN_1^\al} (\1_{t,R}^{\textrm{low}} \check{w})
%     P_{N_3}Q_{\ll MN_1^\al}(\1_{t,R}^{\textrm{low}}\check{z}_3)\\
%     &\quad\quad\times P_{M} \bigg(P_{N_2}Q_{\ll MN_1^\al}\check{w}\bigg(\prod_{j=4}^{n-1}P_{N_j}Q_{\ll MN_1^\al}\check{z}_j\bigg)
%     P_{N_n}Q_{\gtrsim MN_1^\al}\check{z}_n \prod_{j=n+1}^{k+2}P_{N_j}\check{z}_j\bigg) dxdt'\bigg|\\
%     &=:\sum_{j=1}^{k+2}I_{\I,2,3,j}^{(3.2)},
%   }
%   where $\sum$ denotes $\sum_{N\gg1}\sum_{N_1,\dots,N_{k+2}}\sum_{1\vee k
% N_5\ll M\lesssim N_4}$ for simplicity.
  The contribution of $I_{\I,2,3,1}^{(3.2)}$ is estimated, thanks to Lemma \ref{extensionlem}, \eqref{eq4.9}, \eqref{eq4.6}, \eqref{eq4.7} and \eqref{eq5.3}, by
  \EQQS{
    &I_{\I,2,3,1}^{(3.2)}\\
    &\lesssim \sum_{N_1,\dots,N_{k+2}}
      \sum_{1\vee kN_5\ll M\lesssim  N_4}
      N_1^{2s-1-\al}M^{-1}\|P_{N_1}\check{w}\|_{X^{0,1}}
      \|P_{N_3}\1_{t,R}^{\textrm{low}}\check{z}_3\|_{L_{t,x}^2}\\
    &\quad\times\bigg\|P_{M} \bigg(P_{N_2}\check{w}\prod_{j=4}^{k+2}P_{N_j}\check{z}_j\bigg)\bigg\|_{L_{t,x}^\I}\\
    &\lesssim  \sum_{N_1,\dots,N_{k+2}}
       \sum_{M\lesssim  N_4}
      N_1^{2s-2}\|P_{N_1}\check{w}\|_{X^{0,1}}
      \|P_{N_3}\1_{t,R}^{\textrm{low}}\check{z}_3\|_{L_{t,x}^2}
      \bigg\|P_{N_2}\check{w}\prod_{j=4}^{k+2}
      P_{N_j}\check{z}_j\bigg\|_{L_t^\I H^{-1/2}_x}\\
    &\lesssim_k C_0^{k-2} \|\check{w}\|_{L_t^\I H_x^{s-1}}
      \sum_{N_1\gtrsim N_4}\sum_{M\lesssim  N_4}
      (N_1^{2s-2}\|P_{N_1}\check{w}\|_{X^{0,1}}
      \|P_{\sim N_1}\1_{t}\check{z}_3\|_{L_{t,x}^2}  \|P_{N_4}\check{z}_4\|_{L^\I_t  H^{s_0}}\\
    &\quad\quad +T^{1/4} N_1^{2s-25/12} M^{-1/3}
      \|P_{N_1}\check{w}\|_{X^{0,1}}
      \|P_{\sim N_1}\check{z}_3\|_{L_{t}^\I L_x^2}
      \|P_{N_4}\check{z}_4\|_{L^\I_t  H^{s_0}})\\
    &\lesssim_k T^{1/4} C_0^{k-2}
    \|\check{w}\|_{L_t^\I H_x^{s-1}}\|\check{w}\|_{X^{s-2,1}}
    \|\check{z}_3\|_{L_t^\I H_x^{s}}
    \|\check{z}_4\|_{L_t^\I H_x^{s}}
    \lesssim_k T^{1/4} C_0^{k}
    \|w\|_{Z_T^{s-1}}\|w\|_{L_T^\I H_x^{s-1}}
  }
  with $1/2<s_0<s$.
  Here, we used $N_1^{-\al}\le N_1^{-1}$ since $\al\ge1$.
  Similarly, we can estimate $I_{\I,2,3,2}^{(3.2)}$ by
  \EQQS{
    I_{\I,2,3,2}^{(3.2)}
    \lesssim_k T^{1/4}C_0 ^{k}
    \|\check{w}\|_{L_t^\I H_x^{s-1}}^2
    \lesssim_k T^{1/4} C_0^{k}
      \|w\|_{L_T^\I H_x^{s-1}}^2 .
  }
  Next, for $I_{\I,2,3,2}^{(3.2)}$, thanks to Lemma \ref{extensionlem}, \eqref{eq4.1}, \eqref{eq4.2}, \eqref{eq4.3} and \eqref{eq5.3}, we have
  \EQQS{
    &I_{\I,2,3,3}^{(3.2)}\\
    &\lesssim T^{1/2}\sum_{N_1,\dots,N_{k+2}}
    \sum_{M\lesssim  N_4}
      N_1^{2s-1} M
      \|P_{N_1}\check{w}\|_{L_t^\I L_x^2}
      \|P_{N_3}\check{z}_3\|_{L_t^\I L_x^2}\\
    &\quad\times\bigg\|(P_{N_2}Q_{\gtrsim L} \check{w})
    \prod_{j=4}^{k+2}P_{N_j}\check{z}\bigg\|_{L_{t}^2 H_x^{-1/2}}\\
    &\lesssim_k T^{1/2} C_0^{k-2} \|\check{w}\|_{X^{s-2,1}}
      \sum_{N_1\gtrsim N_4}
      \sum_{M\lesssim  N_4}
      N_1^{2s-2} \|P_{N_1}\check{w}\|_{L_t^\I L_x^2}
      \|P_{\sim N_1}\check{z}_3\|_{L_t^\I L_x^2}
      \|P_{N_4}\check{z}_4\|_{L^\I_t  H^{s_0}}\\
    &\lesssim_k T^{1/2} C_0^{k-2}
    \|\check{z}_3\|_{L_t^\I H_x^{s}}
    \|\check{z}_4\|_{L_t^\I H_x^{s}}
    \|\check{w}\|_{X^{s-2,1}}
    \|\check{w}\|_{L_t^\I H_x^{s-1}}
    \lesssim_k T^{1/2} C_0^{k}
    \|w\|_{Z_T^{s-1}}
    \|w\|_{L_T^\I H_x^{s-1}}
  }
  with $1/2<s_0<s$.
  Here, we used $N_1^{-\al}\le N_1^{-1}$.
  By a similar argument, we also have
  \EQQS{
    &I_{\I,2,3,4}^{(3.2)}\\
    &\lesssim T^{1/2}\sum_{N_1,\dots,N_{k+2}}
      \sum_{M\lesssim  N_4}
      N_1^{2s-1} M
      \|P_{N_1}\check{w}\|_{L_t^\I L_x^2}
      \|P_{N_3}\check{z}\|_{L_t^\I L_x^2}\\
    &\quad\times
      \bigg\|(P_{N_2}Q_{\ll MN_1^\al}\check{w})
      (P_{N_4}Q_{\gtrsim MN_1^\al}\check{z}_4)
      \prod_{j=5}^{k+2}P_{N_j}\check{z}_j\, \bigg\|_{L_{t}^2 H_x^{-1/2}}\\
    &\lesssim_k T^{1/2} C_0^{k-2}
    \|\check{w}\|_{L_t^\I H_x^{s-1}}
    \sum_{N_1\gtrsim N_4}
    \sum_{M\lesssim  N_4}
    N_1^{2s-2}\|P_{N_1}\check{w}\|_{L_t^\I L_x^2}
    \|P_{\sim N_1}\check{z}_3\|_{L_t^\I L_x^2}
    \|P_{N_4}\check{z}_4\|_{X^{s_0,1}}\\
    &\lesssim_k T^{1/2} C_0^{k-2}
    \|\check{w}\|_{L_t^\I H_x^{s-1}}
    \|\check{z}_4\|_{X^{s(\al)-1,1}}
    \sum_{N_1} N_1^{s_0-s(\al)} N_1^{2s-1}
    \|P_{N_1}\check{w}\|_{L_t^\I L_x^2}
    \|P_{\sim N_1}\check{z}_3\|_{L_t^\I L_x^2}\\
    &\lesssim T^{1/2} C_0^{k-2}
    \|\check{z}_3\|_{L_t^\I H_x^{s}}
    \|\check{z}_4\|_{X^{s(\alpha)-1,1}}
    \|\check{w}\|_{L_t^\I H_x^{s-1}}^2
    \lesssim_k T^{1/2} C_0^{k} \|w\|_{L_T^\I H_x^{s-1}}^2.
  }
  Here, we chose $s_0\in (1/2,s(\al))$ and used $N_1^{-\al}\le N_1^{-1}$.
  Similarly, we can estimate $I_{\I,2,3,n}^{(3.2)}$ for $5\le n\le k+2$ by the same bound as that of $I_{\I,2,3,4}^{(3.2)}$,
  % \EQQS{
  %   \sum_{n=5}^{k+2} I_{\I,2,3,n}^{(3.2)}
  %   \lesssim_k T^{1/2}C_0^{k}
  %   \|w\|_{L_T^\I H_x^{s-1}}^2
  %   \lesssim_k T^{1/2}
  %   \|w\|_{L_T^\I H_x^{s-1}}^2.
  % }
  which completes the estimate of the contribution of $I_t^{(3.2)}$.
  On the other hand, the contribution of $J_t^{(3.2)}$ can be controlled by $I_t^{(3.2)}$ since $s>1/2$ and $N_2^{2s-1}\le N_1^{2s-1}$.

%   Indeed, we set
%   \EQQS{
%     J_t^{(3.2)}
%     &\le \sum_{N\gg1}\sum_{N_1,\dots,N_{k+2}}N^{2(s-1)}
%     \bigg|\int_0^t\int_\T P_{N_1}w P_{N_3}z
%     P_{\lesssim N_5} \bigg(\p_x P_N^2 P_{N_2}w\prod_{j=4}^{k+2}P_{N_j}z\bigg) dxdt'\bigg|\\
%     &\quad+ \sum_{N\gg1}\sum_{N_1,\dots,N_{k+2}}\sum_{N_5\ll M\lesssim N_2} N^{2(s-1)}
%     \bigg|\int_0^t\int_\T P_{N_1}w P_{N_3}z
%     P_{M} \bigg(\p_x P_N^2 P_{N_2}w\prod_{j=4}^{k+2}P_{N_j}z\bigg) dxdt'\bigg|\\
%     &=:J_{t,1}^{(3.2)}+J_{t,2}^{(3.2)}.
%   }
%   Observe that $J_{t,1}^{(3.2)}$ is bounded by the right hand side of the
% first inequality in \eqref{eq5.2}.

  \underline{Subcase 3.3: $N_2 \ll N_5$ ($k\ge 3$).} Recall that we can assume $N_2\ge 1 $ and note  that we can  also assume that $ N_5 \ge 1 $ since $ N_5=0 $ is included in Subcase 3.2: $N_2\gtrsim N_5$.
  As in Subcase 3.2, we evaluate $I_t^{(3.2)}$ and $J_t^{(3.2)}$.
  Define $L_k$ for $k\ge 3$ so that $L_3:=N_2$ and $L_k:=N_2\vee kN_6$ for $k\ge 4$.
  We decompose
  \EQQS{
    I_t^{(3.2)}
    &\le \sum N^{2(s-1)}
    \bigg|\int_0^t\int_\T \p_x P_N^2 P_{N_1}w P_{N_3}z_3
    P_{\lesssim L_k} \bigg(P_{N_2}w\prod_{j=4}^{k+2}P_{N_j}z_j\bigg) dxdt'\bigg|\\
    &\quad+ \sum \sum_{L_k\ll M\lesssim N_4} N^{2(s-1)}
    \bigg|\int_0^t\int_\T \p_x P_N^2 P_{N_1}w P_{N_3}z_3
    P_{M} \bigg(P_{N_2}w\prod_{j=4}^{k+2}P_{N_j}z_j\bigg) dxdt'\bigg|\\
    &=:I_{t,1}^{(3.3)}+I_{t,2}^{(3.3)},
  }
  where $\sum$ denotes $\sum_{N\gg1}\sum_{N_1,\dots,N_{k+2}}$ for simplicity.
  In what follows, we may assume that $k\ge 4$ since the case $k=3$ can
be treated by the same way as Subcase 3.2.
  For $I_{t,1}^{(3.3)}$,
  H\"older's, Bernstein's and Minkowski's inequalities, \eqref{eq_stri2} and \eqref{eq_stri4} give
  \EQQS{
    &I_{t,1}^{(3.3)}\\
    &\lesssim\biggr(\sum_{N_1,\dots,N_{k+2}\atop N_2\ge k N_6}N_1^{2s-1}N_2^{3/4}
      +\sum_{N_1,\dots,N_{k+2}\atop N_2< k N_6}N_1^{2s-1}(k N_6)^{3/4}\biggr)\\
    &\quad\quad\times\int_0^t\|P_{N_1}w\|_{L_x^2}
      \|P_{N_3}z_3\|_{L_x^2}
      \|P_{N_2}w\|_{L_x^4}\|P_{N_4}z_4\|_{L_x^4}
      \|P_{N_5}z_5\|_{L_x^4}
      \prod_{j=6}^{k+2}\|P_{N_j}z_j\|_{L_x^\I}dt'\\
    &\lesssim_k C_0^{k-2}
      \|w\|_{L_T^\I H_x^{s-1}}
      \sum_{N_4\ge N_5\gg  N_2}
     \|D_x^{-5/12}P_{N_2}w\|_{L_T^3 L_x^4}
        \prod_{i=4}^5  \bigg(\frac{N_2}{N_i}\bigg)^{7/12} \|D_x^{7/12}P_{N_i}z_i\|_{L_T^3 L_x^4}\\
    &\quad + C_0^{k-3} \|w\|_{L_T^\I H_x^{s-1}}
      \sum_{N_4\ge N_5\ge  N_2\vee N_6}  \bigg(\frac{N_2}{N_4}\bigg)^{5/12}
      \|D_x^{-5/12}P_{N_2}w\|_{L_T^3 L_x^4}\\
    &\quad\quad\times\bigg(\frac{N_6}{N_4}\bigg)^{1/6}
     \|D_x^{7/12}P_{N_4}z_4\|_{L_T^3 L_x^4}
     \bigg(\frac{N_6}{N_5}\bigg)^{7/12}
     \|D_x^{7/12}P_{N_5}z_5\|_{L_T^3 L_x^4}
     \|P_{N_6} z_6 \|_{L_{T,x}^\I} \\
    &\lesssim_k T^{5/8} C_0^{k} G(C_0)
      \|w\|_{L_T^\I H_x^{s-1}}^2
      }
  since $s(\al)\ge 7/12+\be(\al)$.
  On the other hand, $I_{t,3}^{(3.3)}$ can be estimated by the same way as $I_{t,2}^{(3.2)}$.
  We also notice that $J_t^{(3.2)}$ is controlled by $I_{t}^{(3.2)}$.
  This concludes the proof.
   \end{proof}

  \section{local and global-well-posedness}
  \subsection{Unconditional local well-posedness in $ H^s(\T)$ for $ s\ge
s(\alpha) $.}
  \subsubsection{Unconditional uniqueness}
  Let $ u_0\in H^s(\T) $ with $ s\ge s(\al) $ and let $u,v$ be two solutions to the Cauchy problem \eqref{eq1}-\eqref{initial} that belong to $L^\infty_T H^s $. According to Lemma \ref{lem1}, we know that
  $u,v\in Z^s_T $ and  Propositon \ref{difdif}  together with \eqref{estdiffXregular}  ensure that $ u\equiv v $ on $[0,T_0] $ with $ 0<T_0\le T $ that only depends on $\|u\|_{Z^{s(\al)}_T}+\|v\|_{Z^{s(\al)}_T} $. Therefore $ u(T_0)=v(T_0) $ and we can reiterate the same argument on $ [T_0,T] $. This proves that $u\equiv v $ on $[0,T] $  after a finite number of iteration.
  \subsubsection{Local existence} We fix $ \alpha\in [1,2] $.
 It is well known (see \cite{ABFS89}) that  the Cauchy problem \eqref{eq1}--\eqref{initial}, with $L_{\alpha+1}$ satisfying Hypotheses \ref{hyp1}, is  locally well-posed in $ H^s(\T) $ for $ s>3/2 $ with a minimal time of existence $ T=T(\|u_0\|_{H^{\frac{3}{2}+}})$. Indeed, for any $ s\in \R $ and any smooth function $ g $ the $H^s$-scalar product
 $ (L_{\al+1}g,g)_{H^s} $ vanishes, and the classical energy method (which consists of the parabolic regularization, the Kato--Ponce commutator estimate \cite{KP88} and the Bona--Smith argument \cite{BS75}) can be applied.
 % $ (\!({\mathcal L}_{\alpha+1} f, f)\!)_{H^s} $ is vanishing and ....
 So let $u\in C([0,T_0];H^\infty(\T)) $ be a smooth solution to  \eqref{eq1} emanating from a smooth initial data $u_0\in H^\infty(\T) $. According to Lemma \ref{lem1} and Proposition \ref{prop_apri} with $ w_N \equiv 1$, there exists an increasing entire function $ G $ non negative  on $ \R_+ $ such that
  \EQS{\label{LWP}
    \|u\|_{L_T^\I H^s}^2
    \le \|u_0\|_{H^s}^2 + T^{1/4} G(\|u\|_{L_T^\I H_x^{s(\al)}})
    \|u\|_{L_T^\I H^s}^2\; .
  }
 for any $ 0<T\le \min(1,T_0) $ and any $2\ge  s\ge s(\alpha) $.

 Taking $ s=s(\al) $ in \eqref{LWP} and setting $ T_1=G(2\|u_0\|_{H^{s(\al)}})^{-4}$ , we deduce from a continuity argument  that for all
$ 0<T<\min(T_0, T_1) $ it holds
 $ \|u\|_{L_T^\I H^{s(\al)}}^2\le 2  \|u_0\|_{H^{s(\al)}}^2$. Re-injecting this in \eqref{LWP} we obtain that for any $2\ge s \ge  s(\al)  $ and any $ 0<T\le \min(T_0,T_1) $ it holds
 $ \|u\|_{L_T^\I H^{s}}^2\le 2  \|u_0\|_{H^{s}}^2$. Therefore the local well-posedness result in $ H^s(\T) $ for $ s>3/2 $ ensures that $u $ does exist on $ [0,T_1] $ with $ u\in C([0,T_1];H^\infty(\T)) $.

 Now let us fix $u_0\in H^{s}(\T) $ with $2\ge  s\ge s(\al)$.  Setting $ u_{0,n}=P_{\le n} u_0 $, it is clear that the sequence
  $ \{u_{0,n}\}_{n\ge 1} $ belongs to $ H^{\infty} (\T)$ and converges to
$ u_0 $ in $H^s(\T) $. We deduce from  above that the emanating sequence of solutions $\{u_{n}\}_{n\ge 1} $ to \eqref{eq1} is included in $ C([0,T];H^\infty(\R)) $ with
   $ T=G(2\|u_0\|_{H^{s(\al)}})^{-4}$ and satisfies  $ \|u_n\|_{L_T^\infty  H^{s}}^2\le 2  \|u_0\|_{H^{s}}^2$.
   Moreover, we infer from Propositon \ref{difdif}  together with \eqref{estdiffXregular}  that   $\{u_{n}\}_{n\ge 1} $ is a Cauchy sequence in $ C([0,T]; H^{s-1}(\T)) $ and thus converges to some function $u\in L^\infty(]0,T[;H^s (\T) ) $ strongly in  $ C([0,T]; H^{s'} (\T) ) $ for any $ s'<s$ and in particular in $C([0,T] \times \T) $.We can thus pass to the limit on the nonlinear term $ f(u_n) $ and $ u $ satisfies \eqref{eq1}-\eqref{initial} at least in the distributional sense.

  \subsubsection{Strong continuity in $H^s(\T)$}
  Let $u$ be the solution emanating from $u_0\in H^s(\T)$.
  It suffices to check that $u(t)$ is continuous in $H^s(\T)$ at $t=0$ thanks to time translation invariance and reversibility in time (invariance  by the change of variables $(t,x) \mapsto (-t,-x) $) of \eqref{eq1} together with the  uniqueness of the solution.
  Since $u\in C([0,T];H^{s'}(\T))\cap L^\I(]0,T[; H^s(\T))$ for $s'<s$, the standard argument (see [Theorem 2.1, \cite{S66}] for instance) yields that $u(t)$ is weakly continuous in $H^s(\T)$ on $[0,T]$.
  In particular, it holds that $\|u_0\|_{H^s}\le\liminf_{t\to0}\|u(t)\|_{H^s}$.
  On the other hand, we see from \eqref{P} with $\om_N\equiv 1$ that
  $\limsup_{t\to 0}\|u(t)\|_{H^s}\le \|u_0\|_{H^s}$
  since $\|u(t)\|_{H^s}\le \|u\|_{L_T^\I H_x^s}$ for any $t\in[0,T]$ (see for instance [Lemma 2.4, \cite{MPV19}]).
  It thus follows that $\lim_{t\to0}\|u(t)\|_{H^s}=\|u_0\|_{H^s}$.
  This together with the weak continuity shows that  $u(t)$ is continuous in $H^s(\T)$ at $t=0$.

   \subsubsection{Continuity with respect to initial data}
Here, we make use of the frequency enveloppe $ \omega$.  Following \cite{KT1},  for any sequence $ \{u_{0,k}\}_{k\ge 1} $ converging to $u_0$ in $
 H^s(\T) $
there exists a  dyadic sequence $\{\om_N\}$ of positive numbers satisfies
$\om_N\le \om_{2N}\le \de\om_N$ for $N\ge1$ such that
\EQS{\label{env}
\|u_0\|_{H^s_\om} <\infty , \quad  \sup_{k\ge 1} \|u_{0,k} \|_{H^s_\om} <\infty \quad \text{and} \quad  \om_N \nearrow +\infty \; .
}
By Remark \ref{rem_envelope}, we may assume that $1<\de\le 2$.
% \sout{Applying Lemma \ref{lem1} and Proposition \ref{prop_apri}  with such sequence $\{\om_N\}$ we obtain that the solution $ u $ emanating from $ u_0 $
% belongs to $ L^\infty_T H^s_\om $. This together with the continuity of $t \mapsto u(t) $ in $ H^{s'}(\T) $, for $s<s'$, ensures that $ u\in C([0,T];H^s(\T)) $.}
Now let $ \{u_{0,k}\}_{k\ge 1}  \subset
H^s(\T)  $, with $\sup_{k\ge 1} \|u_{0,k}\|_{H^s} \le 2 \|u_0\|_{H^s} $,  that converges to $ u_0 $ in $  H^s(\T) $ and let  $\{\om_N\}$ satisfying \eqref{env}. Applying Lemma \ref{lem1} and Proposition \ref{prop_apri}  with $\{\om_N\}$ we infer that the sequence of solution $\{u_k\}_{k\ge 1} $ emanating from $ \{u_{0,k}\}_{k\ge 1} $ is bounded in
$ L^\infty_T H^s_\om $ with $T=G(4\|u_0\|_{H^{s(\al)}})^{-4}$. This together with the Lipschitz bound \eqref{eq_difdif} in $ H^{s-1}(\T) $ ensures $\{u_k\}_{k\ge 1} $  converges to $ u $ in $C([0,T];H^s(\T)) $.

\section{Global existence}
\eqref{eq1} is Hamiltonian and enjoys at least two conservation laws that
correspond to the conservation of the mass and of the energy.
The local-wellposedness result ensures that \eqref{eq1} is wellposed in the energy space $ H^{\alpha/2}(\T) $ for $ \alpha\in [\sqrt{2},2] $. In this section we take advantage of this to prove some global existence result in this case.
Note that since  $  p_{\alpha+1} $ is odd, $\frac{p_{\alpha+1} (\xi)}{\xi}  $ is well-defined on $ \Z $.
 We denote by $\partial_x^{-1} L_{\alpha+1} $   the Fourier multiplier by
 $\frac{p_{\alpha+1} (k)}{k}$. \subsection{Conservation laws}
Assume that the  $ u_0\in H^3(\T) $ and let $u\in C([0,T];H^3(\T)) $ be the associated solution to \eqref{eq1}. Multiplying \eqref{eq1} by $ u $ and integrating  by parts, using that $p_{\alpha+1} $ is odd and $ u $ is real valued, we infer that $ M(u)=\int_{\T} u^2(x)\, dx $ is an invariant of the motion.
In the same way multiplying  \eqref{eq1} by $\partial_x^{-1}L_{\alpha+1} u + f(u) $ and integrating by parts we infer that
\EQQS{
E(u)=\frac{1}{2} \int_{\T} u  \partial_x^{-1} L_{\alpha+1} u +\int_{\T}
F(u)
}
where
\EQQS{
F(x)= \sum_{k=0}^{+\infty} \frac{f^{(k)}(0)}{(k+1)!} x^{k+1} , \quad
\text{i.e.} \quad F(x)=\int_0^x f(y)\, dy  \; .
}
is conserved along the flow. Now from Hypothesis \ref{hyp1} we infer that

\begin{equation}\label{sup1}  p_{\alpha+1}(\xi) \sim \xi^{\al+1} , \quad \forall  \xi \ge \xi_0 \; .
\end{equation}
Therefore  $\frac{p_{\alpha+1} (\xi)}{\xi}>0 $ for $ |\xi|\ge \xi_0$ and denoting by  $ \Lambda_{\alpha/2} $ the Fourier multiplier defined by
\EQQS{
  \widehat{ \Lambda_{\alpha/2} u}(k)= \sqrt{\frac{p_{\alpha+1} (k)}{k}} \hat{u}(k) \1 _{|k|\ge \xi_0} 
}
we may  thus rewritte $\int_{\T} u  \partial_x^{-1} L_{\alpha+1} u$ as
\EQQS{
  \int_{\T} u  \partial_x^{-1} L_{\alpha+1} u= \sum_{1\le |k|\le \xi_0}  \frac{p_{\alpha+1} (\xi)}{\xi} |\hat{u}(k)|^2
  +\int_{\T} |\Lambda_{\alpha/2}  P_{\ge \xi_0} u|^2
}
with
\begin{equation}\label{sup2}
\int_{\T} |\Lambda_{\alpha/2}  P_{\ge \xi_0} u|^2  \sim \| P_{\ge \xi_0} u\|_{H^{\alpha/2}}^2
\quad \text{and} \quad \Bigl| \sum_{1\le |k|\le \xi_0}  \frac{p_{\alpha+1} (\xi)}{\xi} |\hat{u}(k)|^2\Bigr|
\le K_0\|u\|_{L^2}^2 \;
\end{equation}
for some fixed $ K_0>0 $.
In particular,  $ u\mapsto \int_{\T} u  \partial_x^{-1} L_{\alpha+1} u $ is continuous from
$ H^{\alpha/2}(\T) $ into $ \R $ which ensure that these conservation laws remain valid for  $ u_0\in H^{\alpha/2}(\T) $ thanks to the continuity with respect to initial data of the flow-map.
\subsection{Arbitrary large initial data}
We assume that we are in one of the two following situations: \\
$ \bullet   $ Case 1: $|F(x) |\lesssim (1+|x|^{p+1})$ for some $ 0<p<2\alpha+1  $.\\
$ \bullet  $ Case 2:  There exists $ B>0  $ such that $ F(x) \le B ,
\; \forall x\in \R $.

In Case 1,  it holds
 \EQQS{
 \bigg|\int_{\T} F(u) \bigg| \lesssim 1+\|u\|_{L^{p+1}}^{p+1} \lesssim 1+ \|u\|_{H^{\frac{1}{2} -\frac{1}{p+1}} }^{p+1}
 \lesssim  1+ \|u\|_{L^2}^{p+1-\frac{p-1}{\alpha} } \|u\|_{H^{\alpha/2}}^\frac{p-1}{\alpha}.
 }
Therefore since  $ p<2\alpha +1 $ and thus $\frac{p-1}{\alpha}<2$  , the conservation of the $ L^2$-norm together with Young's inequality and \eqref{sup2} lead to
\EQQS{
  \bigg|\int_{\T} F(u(t)) \bigg| \le C_0+ C_1(\|u_0\|_{L^2})  \|u(t)\|_{H^{\alpha/2}}^\frac{p-1}{\alpha} \le C_0 +  \frac{1}{4} \int_{\T} |\Lambda_{\alpha/2} P_{\ge \xi_0} u|^2 + C_2(\|u_0\|_{L^2}) \; .
}
  The conservation of the energy and \eqref{sup2} then ensure that the trajectory of $ u $ remains bounded in $ H^{\alpha/2}(\T) $. This, together
with the local well-posedness result, leads to the global existence of the solution in $ H^s(\T) $ for $ s\ge \alpha/2 $.
  Typical examples of Case 1 are:\\
  $ \bullet $  $f(x)$ is a  polynomial function of degree strictly less than $ 2\al +1$. \\
  $\bullet $ $f(x)$ is a polynomial function of $ \sin(x) $ and $\cos(x) $.   \\
   \vspace{2mm}\\
 In Case 2, it holds $ \int_{\T} F(u) \le 2 \pi B $ and
  the conservation of the energy  together with \eqref{sup2} ensure  that
for any $ t\in [0,T] $
 \EQQS{
   \|u(t)\|_{H^{\alpha/2}}^2 \lesssim E(u_0) +M(u_0)+ \int_{\T} F(u) \lesssim E(u_0)+M(u_0)+ 2 \pi B
 }
 that  leads to the global existence result.  Typical examples of Case 2
are:\\
   $ \bullet $ $ f(x) $ is a  polynomial function of odd degree with $ \displaystyle\lim_{x\to +\infty} f(x)=-\infty $. \\
  $\bullet $ $f(x) =-\exp(x) $ or $ f(x)=-\sinh(x) $.
\subsection{Small initial data}
We set
\EQQS{
\tilde{F} (x) = \sum_{k=2}^{+\infty} \frac{f^{(k)}(0)}{(k+1)!} x^{k+1} =F(x)- \frac{f'(0)}{2} x^2
\quad \text{and} \quad G(x) = \sum_{k=2}^{+\infty} \frac{|f^{k}(0)|}{(k+1)!} x^{k-2}\; .
}
Then the conservations of $ M $ and $ E $ lead to
\begin{eqnarray}
E(u) +\Bigl( \frac{K_0+1-f'(0)}{2}  \Bigr) M(u)  &= &
 \frac{1}{2}  \int_{\T} \Bigl((K_0+1) |u|^2+u  \partial_x^{-1} L_{\alpha+1} u\Bigr) +
\int_{\T} \tilde{F}(u) \nonumber \\
& = & E(u_0) +\Bigl( \frac{K_0+1-f'(0)}{2}  \Bigr) M(u_0)\; , \label{sup3}
\end{eqnarray}
where $ K_0 >0 $ is the constant interfering in \eqref{sup2}.\\
We observe that $\displaystyle \lim_{x\to 0} G(x) =| f'(0)/6| $ and thus there exists $ \delta_0>0 $ such that for $\|u\|_{H^{\alpha/2}} <\delta_0$ it holds
$ |G(u)| \le G(\|u\|_{L^\infty_x}) \le |f'(0)| $. Under this restriction we thus have
\EQQS{
  \bigg|\int_{\T} \tilde{F}(u) \bigg| \le |f'(0)| \int_{\T} |u|^3 \le C_0  |f'(0)| \|u\|_{H^{\alpha/2}}^3.
}
Therefore according to \eqref{sup2}-\eqref{sup3} there exists a positive constant $ C_1 >0 $ that depends on $ L_{\alpha+1} $ such that
\EQQS{
  C_1\|u\|_{H^{\alpha/2}}^2-C_0 | f'(0)| \|u\|_{H^{\alpha/2}}^3 \le E(u) +\Bigl( K_0+1-\frac{f'(0)}{2}  \Bigr) M(u) \; .
}
We set
\EQQS{
C=\max\Bigl(\frac{C_0 | f'(0)|}{C_1}, \frac{1}{2\delta_0}\Bigr)\quad \text{and} \quad
p(x)=  C_1 (x^2- C x^3 ).
}
It is easy to check that $p(x) \le C_1 x^2-C_0 | f'(0)| x^3 $ on $ \R_+ $ and that
 $ p(\frac{1}{2C} )=\frac{C_1}{8} (\frac{C_1}{C})^2 >0  $.

Let $ s\ge \alpha/2 $. We infer from above that for any  $u_0\in H^s(\T) $
such that
\begin{equation}\label{glob1}
  E(u_0) +\Bigl( K_0+1-\frac{f'(0)}{2}  \Bigr) M(u_0)<
\frac{C_1}{8} \Bigl(\frac{C_1}{C}\Bigr)^2
\end{equation}
 the solution $ u\in C([0,T] ; H^s(\T)) $ emanating from $ u_0 $ satisfies
\begin{equation}\label{glob2}
p(\|u(t)\|_{H^{\alpha/2}}) < E(u_0) +\Bigl( K_0+1-\frac{f'(0)}{2}  \Bigr)
M(u_0)<\frac{C_1}{8} \Bigl(\frac{C_1}{C}\Bigr)^2
\end{equation}
for any $ t\in [0,T] $ such that $ \|u(t)\|_{H^{\alpha/2}}<\delta_0 $. Since $
p(1/2C) =\frac{C_1}{8} (\frac{C_1}{C})^2$ and $ 1/2C \le \delta_0 $, a continuity argument ensures that   for $ u_0\in H^{\alpha/2}(\T) $ satisfying \eqref{glob1} with $ \|u_0\|_{H^{\alpha/2}}<\delta_0 $, the emanating solution $ u $ verifies  \eqref{glob2} for all times and thus can be extended globally in times
 according to the LWP result.
Finally since of course, by the continuity of the energy in $ H^{\alpha/2}(\T) $, \eqref{glob1} is satisfied for $   \|u_0\|_{H^{\alpha/2}}$ small
enough, this leads to  the global existence result in $ H^s(\T) $, $ s\ge
\alpha/2 $, for initial data with $ H^{\alpha/2}$-norm small enough.

\section*{Acknowlegdements}

The second author was supported by JSPS KAKENHI Grant Number JP20J12750, JSPS Overseas Challenge Program for Young Researchers and Iizuka Takeshi Scholarship Foundation. This work was conducted during a visit of the second author at Institut Denis Poisson (IDP) of Universit\'e de Tours in France. The second author is deeply grateful to IDP for its kind hospitality.
The authors are grateful to the anonymous referees for valuable  remarks and suggestions that improved this manuscript.

\end{document}